\documentclass[reqno,12pt]{article}%
\usepackage{amsmath}
\usepackage{graphicx}
\usepackage{amsfonts}
\usepackage{amssymb}%
 \makeatletter
\newtheorem{theorem}{Theorem}[section]

\newtheorem{fact}[theorem]{Fact}

\newtheorem{corollary}[theorem]{Corollary}

\newtheorem{lemma}[theorem]{Lemma}

\newtheorem{problem}[theorem]{Problem}
\newtheorem{question}[theorem]{Question}
\newtheorem{proposition}[theorem]{Proposition}

\newenvironment{proof}[1][Proof]{\noindent{\textbf {#1}  }}  {\hfill$\Box$\bigskip}

\begin{document}

\title{Analytic methods for uniform hypergraphs}
\author{Vladimir Nikiforov\thanks{Department of Mathematical Sciences, University of
Memphis, Memphis TN 38152, USA; email: \textit{vnikifrv@memphis.edu}}}
\maketitle

\begin{abstract}
This paper develops analityc methods for investigating uniform hypergraphs.
Its starting point is the spectral theory of $2$-graphs, in particular, the
largest and the smallest eigenvalues of $2$-graphs. On the one hand, this
simple setup is extended to weighted $r$-graphs, and on the other, the
eigenvalues-numbers $\lambda$ and $\lambda_{\min}$ are generalized to
eigenvalues-functions $\lambda^{\left(  p\right)  }$ and $\lambda_{\min
}^{\left(  p\right)  }$, which encompass also other graph parameters like
Lagrangians and number of edges. The resulting theory is new even for
$2$-graphs, where well-settled topics become challenges again.

The paper covers a multitude of topics, with more than a hundred concrete
statements to underpin an analytic theory for hypergraphs. Essential among
these topics are a Perron-Frobenius type theory and methods for extremal
hypergraph problems.

Many open problems are raised and directions for possible further research are outlined.

\vspace{3in}

\pagebreak

\end{abstract}
\tableofcontents

\newpage

This paper outlines an analytical method in hypergraph theory, shaped after
the spectral theory of $2$-graphs. After decades of polishing, spectral
methods for $2$-graphs reside on a solid ground, with traditions settled both
in tools and problems. Naturally we want similar comfort and convenience for
spectra of hypergraphs. Recently several researchers have contributed to this
goal, but the endeavor is far from completed, and even the central concepts
are not in stone yet. To give results that remain relevant in these dynamic
times, we take a somewhat conservative viewpoint and focus only on two
fundamental concepts, very likely to be of top interest in the nearest future.
More precisely, we study parameters similar to the largest and the smallest
eigenvalues of $2$-graphs, which are by far the most studied graph eigenvalues
anyway. We define these parameters variationally, like the Rayleigh principle
defines the extremal eigenvalues of Hermitian matrices. This approach goes
along the work of Lim \cite{Lim05} on eigenvalues of hypermatrices, but
historically, the same idea has been suggested back in 1930 by Lusternik and
Schnirelman \cite{LuSh30}. We show that our concepts fit well also with the
algebraic definitions of eigenvalues proposed by Qi \cite{Qi05}. These
references are for orientation only, as the paper is mostly
self-contained.\medskip

More important, we parametrize our eigenvalues with a real parameter; thus
instead of a pair of eigenvalues-numbers, with each graph $G$ we associate a
pair of real functions $\lambda^{\left(  p\right)  }\left(  G\right)  $ and
$\lambda_{\min}^{\left(  p\right)  }\left(  G\right)  $. On the one hand, this
choice covers the extremal $Z$-eigenvalues of Qi, but more importantly the
function $\lambda^{\left(  p\right)  }$ naturally brings together other
fundamental parameters, like the Lagrangian and the number of edges. These
ideas extend the approach of Keevash, Lenz and Mubayi \cite{KLM13}, which in
turn builds upon Friedman and Wigderson \cite{FrWi95}. Other relevant
contributions in a similar vein are by Cooper and Dutle \cite{CoDu11} and by
Pearson and Zhang \cite{PeZh12}.\medskip\ 

Some of the problems presented below extend well-known problems for $2$-graphs
to hypergraphs, sometimes with similar solutions as well. But we also present
a few completely new topics for which $2$-graphs do not suggest even the
slightest clue; most likely such topics will prevail in the future study of
hypergraphs. Interestingly, the fundamental linear or multilinear algebraic
concepts shape the landscape of the new theory, but the proof methods are
essentially nonlinear. Thus, real analysis is more usable than linear
algebraic techniques. Very likely, differential manifolds theory will provide
important new tools.\medskip

We proceed with the outline of the individual sections.

In Section \ref{Def} we define the parameters $\lambda^{\left(  p\right)
}\left(  G\right)  $ and $\lambda_{\min}^{\left(  p\right)  }\left(  G\right)
$ for a graph $G$ and a real number $p\geq1$. We introduce eigenvectors and
extend the definitions weighted graphs, which are essentially equivalent to
nonnegative symmetric hypermatrices.

Section \ref{props} starts with calculations of concrete $\lambda^{\left(
p\right)  }$ and $\lambda_{\min}^{\left(  p\right)  }$ and then continues with
extensions of well-known results for $2$-graphs. We also initiate a systematic
study of $\lambda^{\left(  p\right)  }\left(  G\right)  $ and $\lambda_{\min
}^{\left(  p\right)  }\left(  G\right)  $ as functions of $p$ for any fixed
graph $G$.

Section \ref{eqs} investigates systems of nonlinear equations for
$\lambda^{\left(  p\right)  }$ and $\lambda_{\min}^{\left(  p\right)  }$
arising using Lagrange multipliers. It is shown that $\lambda^{\left(
p\right)  }$ and $\lambda_{\min}^{\left(  p\right)  }$ comply with the
eigenvalue definitions of Qi. Also, we discuss $\lambda^{\left(  p\right)  }$
of regular graphs, which turn out to be a difficult problem for some hypergraphs.

Section \ref{facts} is intended to prepare the reader for Perron-Frobenius
type theorems for hypergraphs. A careful selection of examples should fend off
hasty expectations arising from $2$-graphs. In particular, it is shown that
graphs as simple as cycles pose difficult problems about $\lambda^{\left(
p\right)  }$. In fact, the analysis of the $r$-cycles answers in the negative
a question of Pearson and Zhang \cite{PeZh12}.

Section \ref{PFsec} studies Perron-Frobenius type questions for $\lambda
^{\left(  p\right)  }.$ We examine in detail the traditional topics for
symmetric nonnegative matrices and offer extensions for hypergraphs. Most
solutions are new and rather complicated, e.g., we introduce the new notion of
\emph{graph tightness}, which extends graph connectedness. Very likely the
results proved in this section are precursors of corresponding results for
nonnegative hypermatrices.

Section \ref{OPs} presents relations between $\lambda^{\left(  p\right)  },$
$\lambda_{\min}^{\left(  p\right)  }$ and various graph operations like
blow-up, sum, and join. We state and prove simple analogs of the celebrated
Weyl's inequalities for sums of Hermitian matrices and give several
applications. We also discuss some results and problems of Nordhaus-Stewart type.

Section \ref{Props} is dedicated to relations of $\lambda^{\left(  p\right)
}$ to partiteness, chromatic number, degrees, and graph linearity. Some of
these relations are well-known for $2$-graphs, but others are specific to
hypergraphs. The section ends up with a few bounds on the minimum and maximum
entries of eigenvectors to $\lambda^{\left(  p\right)  },$ useful in applications.

Section \ref{Props_m} focuses on bounds on $\lambda_{\min}^{\left(  p\right)
}.$ First we give essentially best possible bounds on $\lambda_{\min}^{\left(
p\right)  }$ in terms of the graph order and size. Next we establish for which
graphs $G$ the equality $\lambda_{\min}^{\left(  p\right)  }\left(  G\right)
=-\lambda^{\left(  p\right)  }\left(  G\right)  $ holds. For $2$-graphs these
are the bipartite graphs; for hypergraphs the relevant property is
\textquotedblleft having an odd transversal.\textquotedblright\ We answer also
a question of Pearson and Zhang\ about symmetry of the algebraic spectrum of a hypergraph.

Section \ref{Exts} is dedicated to extremal problems for hypergraphs, a topic
that has been developed in a recent paper by the author. The main theorem here
is that spectral extremal and edge extremal problems are asymptotically
equivalent. This is a new result even for $2$-graphs.

Section \ref{Rands} is a very brief excursion in random hypergraphs. Two
theorems are stated about $\lambda^{\left(  q\right)  }$ and $\lambda_{\min
}^{\left(  q\right)  }$ of the random graph $G^{r}\left(  n,p\right)  $ for
fixed $p>0$.

Section \ref{Cons} is a summary of the main topics of the paper; it outlines
directions for further research, and raises several problems and questions.

Section \ref{basics} contains reference material and a glossary of hypergraph
terms. There is basic information on classical inequalities and on polynomial
forms. Parts of the paper may seem easier if the reader skims throughout this
section beforehand.

\newpage

\section{\label{Def}The basic definitions}

Given a nonempty set $V,$ write $V^{\left(  r\right)  }$ for the family of
all\ $r$-subsets of $V.$ An $r$-uniform hypergraph (= $r$\emph{-graph})
consists of a set of \emph{vertices} $V=V\left(  G\right)  $ and a set of
\emph{edges} $E\left(  G\right)  \subset V^{\left(  r\right)  }.$ For
convenience we identify $G$ with the indicator function of $E\left(  G\right)
$, that is to say, $G:V^{\left(  r\right)  }\rightarrow\left\{  0,1\right\}  $
and $G\left(  e\right)  =1$ iff $e\in E\left(  G\right)  .$ Further, $v\left(
G\right)  $ stands for the number of vertices, called the \emph{order} of $G;$
$\left\vert G\right\vert $ stands for the number of edges, called the
\emph{size} of $G.$ If $v\left(  G\right)  =n$ and $V\left(  G\right)  $ is
not defined explicitly, it is assumed that $V\left(  G\right)  =[n]=\left\{
1,\ldots,n\right\}  ;$ this assumption is crucial for our notation.

In this paper \textquotedblleft graph\textquotedblright\ stands for
\textquotedblleft uniform hypergraph\textquotedblright; thus,
\textquotedblleft ordinary\textquotedblright\ graphs are referred to as
\textquotedblleft$2$-graphs\textquotedblright. If $r\geq2,$ we write
$\mathcal{G}^{r}$ for the family of all $r$-graphs, and $\mathcal{G}%
^{r}\left(  n\right)  $ for the family of all $r$-graphs of order $n.$

Given a $G\in\mathcal{G}^{r}\left(  n\right)  ,$ the polynomial form (=
\emph{polyform}) of $G$ is a function $P_{G}\left(  \mathbf{x}\right)
:\mathbb{R}^{n}\rightarrow\mathbb{R}^{1}$ defined for any vector $\left[
x_{i}\right]  \in\mathbb{R}^{n}$ as
\[
P_{G}\left(  \left[  x_{i}\right]  \right)  :=r!\sum_{\left\{  i_{1}%
,\ldots,i_{r}\right\}  \in E\left(  G\right)  }x_{i_{1}}\cdots x_{i_{r}}.
\]
If $r=2,$ the polyform $P_{G}\left(  \left[  x_{i}\right]  \right)  $ is the
well-known quadratic form
\[
2\sum_{\left\{  i,,j\right\}  \in E\left(  G\right)  }x_{i}x_{j},
\]
so polyforms naturally extend quadratic forms to hypergraphs.

Note that $P_{G}\left(  \mathbf{x}\right)  $ is a homogenous polynomial of
degree $r$ and has a continuous derivative in each variable. More details
about $P_{G}\left(  \mathbf{x}\right)  $ can be found in Section \ref{Pfor}.
Let us note that the coefficient $r!$ makes our results consistent with a
large body of work on hypermatrices.

\subsection{The largest and the smallest eigenvalues of an $r$-graph}

Let $G\in\mathcal{G}^{r}\left(  n\right)  .$ Define the \emph{largest
eigenvalue} $\lambda\left(  G\right)  $ of $G$ as%

\[
\lambda\left(  G\right)  :=\max_{\left\vert \mathbf{x}\right\vert _{r}=1}%
P_{G}\left(  \mathbf{x}\right)  ,
\]
and the $\emph{smalles}t$\emph{ eigenvalue} $\lambda_{\min}\left(  G\right)  $
as%
\[
\lambda_{\min}\left(  G\right)  :=\min_{\left\vert \mathbf{x}\right\vert
_{r}=1}P_{G}\left(  \mathbf{x}\right)  .
\]
If $G$ has no edges, we let $\lambda\left(  G\right)  =\lambda_{\min}\left(
G\right)  =0.$

Note that the condition $\left\vert \mathbf{x}\right\vert _{r}=1$ describes
$\mathbb{S}_{r}^{\left(  n-1\right)  },$ the $\left(  n-1\right)
$-dimensional unit sphere in the $l^{r}$ norm in $\mathbb{R}^{n}$; see Section
\ref{Nors} for more details. Since $\mathbb{S}_{r}^{\left(  n-1\right)  }$ is
a compact set, and $P_{G}\left(  \mathbf{x}\right)  $ is continuous,
$P_{G}\left(  \mathbf{x}\right)  $ attains its minimum and maximum on
$\mathbb{S}_{r}^{\left(  n-1\right)  },$ hence $\lambda\left(  G\right)  $ and
$\lambda_{\min}\left(  G\right)  $ are well defined. Also, note that if $r=2,$
the Rayleigh principle states that the above equations indeed define the
largest and the smallest eigenvalues of $G$.

\subsection{Introduction of $\lambda^{\left(  p\right)  }$ and $\lambda_{\min
}^{\left(  p\right)  }$}

For an $r$-graph $G\ $the parameters $\lambda\left(  G\right)  $ and
$\lambda_{\min}\left(  G\right)  $ are special in many ways; however, great
insight comes from the study of the functions $\lambda^{\left(  p\right)
}\left(  G\right)  $ and $\lambda_{\min}^{\left(  p\right)  }\left(  G\right)
$ defined for any real number $p\geq1$ as
\begin{align}
\lambda^{\left(  p\right)  }\left(  G\right)   &  :=\max_{\left\vert
\mathbf{x}\right\vert _{p}=1}P_{G}\left(  \mathbf{x}\right)  ,\label{defla}\\
\lambda_{\min}^{\left(  p\right)  }\left(  G\right)   &  :=\min_{\left\vert
\mathbf{x}\right\vert _{p}=1}P_{G}\left(  \mathbf{x}\right)  . \label{defsa}%
\end{align}
Here the condition $\left\vert \mathbf{x}\right\vert _{p}=1$ describes the
$\left(  n-1\right)  $-dimensional unit sphere $\mathbb{S}_{p}^{n-1}$ in the
$l^{p}$ norm, which is compact; since $P_{G}\left(  \mathbf{x}\right)  $ is
continuous, $\lambda^{\left(  p\right)  }\left(  G\right)  $ and
$\lambda_{\min}^{\left(  p\right)  }\left(  G\right)  $ are well
defined.\medskip

Note that $\lambda^{\left(  r\right)  }\left(  G\right)  =\lambda\left(
G\right)  $ and $\lambda_{\min}^{\left(  r\right)  }\left(  G\right)
=\lambda_{\min}\left(  G\right)  .$ Also note that $\lambda^{\left(  1\right)
}\left(  G\right)  $ is another much studied graph parameter, known as the
Lagrangian\footnote{Let us note that this use of the name \emph{Lagrangian} is
at odds with the tradition. Indeed, names as \emph{Laplacian, Hessian,
Gramian, Grassmanian}, etc., usually denote a structured object like matrix,
operator, or manifold, and not just a single number.} of $G$. So
$\lambda^{\left(  p\right)  }\left(  G\right)  $ links $\lambda\left(
G\right)  $ to a body of previous work on hypergraph problems. The parameter
$\lambda^{\left(  p\right)  }\left(  G\right)  $ has been introduced by
Keevash, Lenz and Mubayi \cite{KLM13}, although they require $p>1.$ It seems
that little is known about $\lambda^{\left(  p\right)  }\left(  G\right)  $
and $\lambda_{\min}^{\left(  p\right)  }\left(  G\right)  $ even for
$2$-graphs.\medskip

For any $p\geq1,$ if $\mathbf{x}$ is a vector such that $\left\vert
\mathbf{x}\right\vert _{p}=\left[  x_{i}\right]  =1$ and $\lambda^{\left(
p\right)  }\left(  G\right)  =P_{G}\left(  \mathbf{x}\right)  ,$ then the
vector $\mathbf{x}^{\prime}=\left[  \left\vert x_{i}\right\vert \right]  $
satisfies $\left\vert \mathbf{x}^{\prime}\right\vert _{p}=1$ and so%
\[
\lambda^{\left(  p\right)  }\left(  G\right)  =P_{G}\left(  \mathbf{x}\right)
\leq P_{G}\left(  \mathbf{x}^{\prime}\right)  \leq\lambda^{\left(  p\right)
}\left(  G\right)  ,
\]
implying that $\lambda^{\left(  p\right)  }\left(  G\right)  =P_{G}\left(
\mathbf{x}^{\prime}\right)  .$ Therefore, there is always a nonnegative vector
$\mathbf{x}$ such that $\left\vert \mathbf{x}\right\vert _{p}=1$ and
$\lambda^{\left(  p\right)  }\left(  G\right)  =P_{G}\left(  \mathbf{x}%
\right)  .$ This implies also the following observations.

\begin{proposition}
\label{pro_ls}If $p\geq1$ and $G\in\mathcal{G}^{r}\left(  n\right)  $, then
\[
\lambda^{\left(  p\right)  }\left(  G\right)  =\max_{\left\vert \mathbf{x}%
\right\vert _{p}=1}\left\vert P_{G}\left(  \mathbf{x}\right)  \right\vert .
\]
In particular, $\lambda^{\left(  p\right)  }\left(  G\right)  \geq\left\vert
\lambda_{\min}^{\left(  p\right)  }\left(  G\right)  \right\vert $ or,
equivalently, $\lambda_{\min}^{\left(  p\right)  }\left(  G\right)
\geq-\lambda^{\left(  p\right)  }\left(  G\right)  .$
\end{proposition}

If $r$ is odd, then $P_{G}\left(  \mathbf{x}\right)  $ is odd, and so
$\lambda_{\min}^{\left(  p\right)  }\left(  G\right)  =-\lambda^{\left(
p\right)  }\left(  G\right)  ;$ thus $\lambda_{\min}^{\left(  p\right)
}\left(  G\right)  $ can give new information only if $r$ is even.

\subsection{Eigenvectors}

If $G\in\mathcal{G}^{r}\left(  n\right)  $ and $\left[  x_{i}\right]  $ is an
$n$-vector such that $\left\vert \left[  x_{i}\right]  \right\vert _{p}=1$ and
$\lambda^{\left(  p\right)  }\left(  G\right)  =P_{G}\left(  \left[
x_{i}\right]  \right)  ,$ then $\left[  x_{i}\right]  $ will be called an
\emph{eigenvector }to $\lambda^{\left(  p\right)  }\left(  G\right)  .$ For
$\lambda_{\min}^{\left(  p\right)  }\left(  G\right)  $ eigenvectors are
defined the same way. For convenience we write $\mathbb{S}_{p,+}^{n-1}$ for
the set of nonnegative $n$-vectors $\mathbf{x}$ with $\left\vert
\mathbf{x}\right\vert _{p}=1.$ Thus, $\lambda^{\left(  p\right)  }\left(
G\right)  $ always has an eigenvector in $\mathbb{S}_{p,+}^{n-1}$. The
following inequalities relate $\lambda^{\left(  p\right)  }\left(  G\right)  $
and $\lambda_{\min}^{\left(  p\right)  }\left(  G\right)  $ to arbitrary
vectors $\mathbf{x}$.

\begin{proposition}
\label{pro_b}Let $p\geq1.$ If $G\in\mathcal{G}^{r}\left(  n\right)  $ and
$\mathbf{x}$ is a real vector, then
\[
P_{G}\left(  \mathbf{x}\right)  \leq\lambda^{\left(  p\right)  }\left(
G\right)  \left\vert \mathbf{x}\right\vert _{p}^{r},
\]
with equality if and only if $\mathbf{x}=0$ or $\left\vert \mathbf{x}%
\right\vert _{p}^{-1}\mathbf{x}$\ is an eigenvector to $\lambda^{\left(
p\right)  }\left(  G\right)  .$ Also,
\[
P_{G}\left(  \mathbf{x}\right)  \geq\lambda_{\min}^{\left(  p\right)  }\left(
G\right)  \left\vert \mathbf{x}\right\vert _{p}^{r},
\]
with equality if and only if $\mathbf{x}=0$ or $\left\vert \mathbf{x}%
\right\vert _{p}^{-1}\mathbf{x}$\ is an eigenvector to $\lambda_{\min
}^{\left(  p\right)  }\left(  G\right)  .$
\end{proposition}

Note that in our definition eigenvectors to $\lambda^{\left(  p\right)
}\left(  G\right)  $ always have length $1$ in the $l^{p}$ norm. This seems to
be a strong restriction of the traditional concept in Linear Algebra; in fact
this restriction gives convenience at a negligible loss, because invariant
subspaces are irrelevant to the study of $\lambda^{\left(  p\right)  }\left(
G\right)  .$ Indeed, for $2$-graphs the eigenvectors of $\lambda\left(
G\right)  $ or $\lambda_{\min}\left(  G\right)  $ span invariant subspaces,
but if $r\geq3$ and $G\in\mathcal{G}^{r},$ the set of eigenvectors to
$\lambda\left(  G\right)  $ or to $\lambda_{\min}\left(  G\right)  $ could be
just a finite set. The situation with $\lambda^{\left(  p\right)  }\left(
G\right)  $ for $p\neq r$ could be even more rigid. We shall consider many
examples below.

\subsection{\label{WGs}Putting weights on edges}

We introduce here \emph{weighted }$r$\emph{-graphs}, a natural and useful
extension of $r$-graphs. Thus, a \emph{weighted }$r$\emph{-graph} $G$ with set
of \emph{vertices} $V$ is a nonnegative real function $G:V^{\left(  r\right)
}\rightarrow\left[  0,\infty\right)  .$ The set of \emph{edges} $E\left(
G\right)  $ of $G$ is defined $E\left(  G\right)  =\left\{  e:e\in V^{\left(
r\right)  }\text{ and }G\left(  e\right)  >0\right\}  ,$ that is to say,
$E\left(  G\right)  $ is the support of $G.$ The \emph{order} $v\left(
G\right)  $ of $G$ is the cardinality of $V$ and the \emph{size} is defined as
$\left\vert G\right\vert =\sum\left\{  G\left(  e\right)  :e\in V^{\left(
r\right)  }\right\}  .$ Weighted $r$-graphs provide the natural setup for many
a statement about graphs, say for Weyl's inequalities in Proposition
\ref{pro_Weyl} or interlacing inequalities.

We write $\mathcal{W}^{r}$ for the family of all weighted $r$-graphs and
$\mathcal{W}^{r}\left(  n\right)  $ for the family of all weighted $r$-graphs
of order $n.$ As usual, $\left\vert G\right\vert _{p}$ stands for the $l^{p}%
$-norm of $G$ and $\left\vert G\right\vert _{\infty}$ is the maximum of $G.$
Clearly $\mathcal{W}^{r}\left(  n\right)  $ is a complete metric space in any
$l^{p}$ norm, $1\leq p\leq\infty.$ Also, if $H\in\mathcal{W}^{r}$ and
$G\in\mathcal{W}^{r}$, $H$ is a called a \emph{subgraph} of $G$, if $V\left(
H\right)  \subset V\left(  G\right)  ,$ and $e\in E\left(  H\right)  $ implies
$H\left(  e\right)  =G\left(  e\right)  ;$ a subgraph $H$ of $G$ is called
\emph{induced }if $e\in E\left(  G\right)  $ and $e\subset V\left(  H\right)
$ implies $H\left(  e\right)  =G\left(  e\right)  .$

Given a vector $\left[  x_{i}\right]  \in\mathbb{R}^{n},$ the \emph{polyform}
of $G$ is defined as
\[
P_{G}\left(  \left[  x_{i}\right]  \right)  :=r!\sum_{\left\{  i_{1}%
,\ldots,i_{r}\right\}  \in E\left(  G\right)  }G\left(  \left\{  i_{1}%
,\ldots,i_{r}\right\}  \right)  x_{i_{1}}\cdots x_{i_{r}},
\]
and the definitions of $\lambda\left(  G\right)  ,$ $\lambda_{\min}\left(
G\right)  ,$ $\lambda^{\left(  p\right)  }\left(  G\right)  ,$ $\lambda_{\min
}^{\left(  p\right)  }\left(  G\right)  $ and \emph{eigenvectors} are the same
as above.

\begin{proposition}
If $G\in\mathcal{W}^{r}\left(  n\right)  $ and $H\in\mathcal{W}^{r}\left(
n\right)  ,$ then
\[
E\left(  G+H\right)  =E\left(  G\right)  \cup E\left(  H\right)  \text{
\ \ \ \ and \ \ \ \ }E\left(  G\cdot H\right)  =E\left(  G\right)  \cap
E\left(  H\right)  .
\]
For every $\mathbf{x}\in\mathbb{R}^{n},$
\[
P_{G+H}\left(  \mathbf{x}\right)  =P_{G}\left(  \mathbf{x}\right)
+P_{H}\left(  \mathbf{x}\right)  .
\]

\end{proposition}

Many results in this paper smoothly extend from graphs to weighted graphs at
no additional cost. However, for simplicity we shall avoid a systematic
extension, leaving it to the interested reader; e.g., one can see that
Proposition \ref{pro_b} holds with no change for weighted graphs as well.

\section{\label{props}Basic properties of $\lambda^{\left(  p\right)  }$ and
$\lambda_{\min}^{\left(  p\right)  }$}

For a start, let us find $\lambda^{\left(  p\right)  }\left(  K_{r}%
^{r}\right)  $ and $\lambda_{\min}^{\left(  p\right)  }\left(  K_{r}%
^{r}\right)  $ and their eigenvectors, where $K_{r}^{r}$ is the $r$-graph of
order $r$ consisting of a single edge. In this case, $P_{K_{r}^{r}}\left(
\mathbf{x}\right)  =r!x_{1}\cdots x_{r}.$ Letting $\left\vert x_{1}\right\vert
^{p}+\cdots+\left\vert x_{r}\right\vert ^{p}=1,$ the AM-GM inequality
(\ref{AM-GM}) implies that%
\[
x_{1}\cdots x_{r}\leq\left\vert x_{1}\right\vert \cdots\left\vert
x_{r}\right\vert \leq\left(  \frac{\left\vert x_{1}\right\vert ^{p}%
+\cdots+\left\vert x_{r}\right\vert ^{p}}{r}\right)  ^{r/p}=r^{-r/p},
\]
with equality holding if and only if $x_{1}=\pm r^{-1/p},\ldots,x_{r}=\pm
r^{-1/p},$ and $x_{1}\cdots x_{r}>0.$ Therefore,
\begin{equation}
\lambda^{\left(  p\right)  }\left(  K_{r}^{r}\right)  =r!/r^{r/p}.
\label{laed}%
\end{equation}
The eigenvectors to $\lambda^{\left(  p\right)  }\left(  K_{r}^{r}\right)  $
are all vectors of the type $\left(  \pm r^{-1/p},\ldots,\pm r^{-1/p}\right)
,$ with the product of the entries being positive. Likewise we see that
$\lambda_{\min}^{\left(  p\right)  }\left(  K_{r}^{r}\right)  =-r!/r^{r/p}$
and the eigenvectors to $\lambda_{\min}^{\left(  p\right)  }\left(  K_{r}%
^{r}\right)  $ are all vectors of the type $\left(  \pm r^{-1/p},\ldots,\pm
r^{-1/p}\right)  ,$ with the product of the entries being negative.

Hence, for $r\geq3$ we get a situation, which is impossible for $2$-graphs;
let us summarize these findings.

\begin{fact}
\label{mev}If $r\geq3,$ then both $\lambda^{\left(  p\right)  }\left(
K_{r}^{r}\right)  $ and $\lambda_{\min}^{\left(  p\right)  }\left(  K_{r}%
^{r}\right)  $ have $r$ linearly independent eigenvectors.
\end{fact}

Trivially, $\lambda^{\left(  p\right)  }$ is monotone with respect to edge weights:

\begin{proposition}
Let $p\geq1,$ $r\geq2$ , let $G\in\mathcal{W}^{r}$ and $H\in\mathcal{W}^{r}$.
If $V\left(  G\right)  =V\left(  H\right)  $ and $H\left(  U\right)  \leq
G\left(  U\right)  $ for each $U\in\left(  V\left(  G\right)  \right)
^{\left(  r\right)  },$ then $\lambda^{\left(  p\right)  }\left(  H\right)
\leq\lambda^{\left(  p\right)  }\left(  G\right)  .$
\end{proposition}

Next, note that $\lambda^{\left(  p\right)  }\left(  G\right)  $ is monotone
with respect to edge addition and $\lambda_{\min}^{\left(  p\right)  }\left(
G\right)  $ is monotone with respect to vertex addition.

\begin{proposition}
\label{pro_s}If $G\in\mathcal{W}^{r},$ $H\in\mathcal{W}^{r}$ , and $H$ is a
subgraph of $G,$ then
\begin{equation}
\lambda^{\left(  p\right)  }\left(  H\right)  \leq\lambda^{\left(  p\right)
}\left(  G\right)  . \label{subl}%
\end{equation}
If $H$ is an induced subgraph of $G,$ then
\[
\lambda_{\min}^{\left(  p\right)  }\left(  G\right)  \leq\lambda_{\min
}^{\left(  p\right)  }\left(  H\right)  .
\]
Hence, $\lambda_{\min}^{\left(  p\right)  }\left(  G\right)  <0$, unless $G$
has no edges.
\end{proposition}

Let us note that the conditions for strict inequality in (\ref{subl}) are not
obvious; we postpone the discussion to Corollary \ref{corPFa} below. Somewhat
unexpected is the fact that if $G\in\mathcal{G}^{r}\ $has no isolated vertices
and $p>r,$ then $\lambda^{\left(  p\right)  }\left(  G\right)  $ is always
larger than $\lambda^{\left(  p\right)  }$ of each of its proper subgraphs;
this statement is stated in detail in Theorem \ref{PFar} and Corollary
\ref{corPFar}.\medskip

Further, Propositions \ref{pro_s} and \ref{pro_b}, and the convexity of
$x^{s}$ for $x\geq0$, $s\geq1,$ imply the following facts, which are as expected.

\begin{proposition}
\label{pro_sub}Let $1\leq p\leq r,$ $G_{1},\ldots,G_{k}$ be pairwise vertex
disjoint $r$-graphs. If $G$ is their union, then%
\begin{align*}
\lambda^{\left(  p\right)  }\left(  G\right)   &  =\max\left\{  \lambda
^{\left(  p\right)  }\left(  G_{1}\right)  ,\ldots,\lambda^{\left(  p\right)
}\left(  G_{k}\right)  \right\} \\
\lambda_{\min}^{\left(  p\right)  }\left(  G\right)   &  =\min\left\{
\lambda_{\min}^{\left(  p\right)  }\left(  G_{1}\right)  ,\ldots,\lambda
_{\min}^{\left(  p\right)  }\left(  G_{k}\right)  \right\}  .
\end{align*}

\end{proposition}

Again, if $p>r$ and $G\in\mathcal{G}^{r},$ the statements are different.

\begin{theorem}
\label{rl}Let $p>r\geq2$ and let $G_{1},\ldots,G_{k\text{ }}$ be the
components of an $r$-graph $G.$ If $G$ has no isolated vertices, then
\begin{equation}
\lambda^{\left(  p\right)  }\left(  G\right)  =\left(  \sum_{i=1}^{k}\left(
\lambda^{\left(  p\right)  }\left(  G_{i}\right)  \right)  ^{p/\left(
p-r\right)  }\right)  ^{\left(  p-r\right)  /p} \label{in2}%
\end{equation}
and%
\[
\lambda_{\min}^{\left(  p\right)  }\left(  G\right)  =-\left(  \sum_{i=1}%
^{k}\left\vert \lambda_{\min}^{\left(  p\right)  }\left(  G_{i}\right)
\right\vert ^{p/\left(  p-r\right)  }\right)  ^{\left(  p-r\right)  /p}.
\]

\end{theorem}

\begin{proof}
We shall prove only (\ref{in2}). Let $G_{1},\ldots,G_{k\text{ }}$ and $G$ be
as required. Let $\mathbf{x}\in\mathbb{S}_{p}^{n-1}$ be an eigenvector to
$\lambda^{\left(  p\right)  }\left(  G\right)  $ and let $\mathbf{y}_{i}$ be
the restriction of $\mathbf{x}$ to the set $V\left(  G_{i}\right)  $. Now, by
Proposition \ref{pro_b},%
\[
\lambda^{\left(  p\right)  }\left(  G\right)  =P_{G}\left(  \mathbf{x}\right)
=\sum_{i=1}^{k}P_{G_{i}}\left(  \mathbf{y}_{i}\right)  \leq\sum_{i=1}%
^{k}\lambda^{\left(  p\right)  }\left(  G_{i}\right)  \left\vert
\mathbf{y}_{i}\right\vert _{p}^{r}.
\]
Letting $s=p/r,$ $t=p/\left(  p-r\right)  ,$ we have $1/s+1/t=r/p+\left(
p-r\right)  /p=1,$ and applying H\"{o}lder's inequality (\ref{Holdin}), we get%
\begin{align*}
\sum_{i=1}^{k}\lambda^{\left(  p\right)  }\left(  G_{i}\right)  \left\vert
\mathbf{y}_{i}\right\vert _{p}^{r}  &  \leq\left(  \sum_{i=1}^{k}\left(
\lambda^{\left(  p\right)  }\left(  G_{i}\right)  \right)  ^{t}\right)
^{1/t}\left(  \sum_{i=1}^{k}\left(  \left\vert \mathbf{y}_{i}\right\vert
_{p}^{r}\right)  ^{s}\right)  ^{1/s}\\
&  =\left(  \sum_{i=1}^{k}\left(  \lambda^{\left(  p\right)  }\left(
G_{i}\right)  \right)  ^{p/\left(  p-r\right)  }\right)  ^{\left(  p-r\right)
/p}\left(  \sum_{i=1}^{k}\left\vert \mathbf{y}_{i}\right\vert _{p}^{p}\right)
^{r/p}=\left(  \sum_{i=1}^{k}\left(  \lambda^{\left(  p\right)  }\left(
G_{i}\right)  \right)  ^{p/\left(  p-r\right)  }\right)  ^{\left(  p-r\right)
/p}.
\end{align*}
To prove equality in (\ref{in2}) for each $i=1,\ldots,k,$ choose an
eigenvector $\mathbf{z}_{i}$ to $\lambda^{\left(  p\right)  }\left(
G_{i}\right)  ;$ then scale each $\mathbf{z}_{i}$ so that $\sum_{i=1}%
^{k}\left\vert \mathbf{z}_{i}\right\vert _{p}^{p}=1$ and $\left(  \left\vert
\mathbf{z}_{1}\right\vert ^{s},,\ldots,\left\vert \mathbf{z}_{k}\right\vert
^{s}\right)  $ is collinear to $\left(  \left(  \lambda^{\left(  p\right)
}\left(  G_{1}\right)  \right)  ^{t},\ldots,\left(  \lambda^{\left(  p\right)
}\left(  G_{k}\right)  \right)  ^{t}\right)  .$ Now, letting $\mathbf{u}$ be
equal to $\mathbf{z}_{i}$ within $V\left(  G_{i}\right)  $ for $i=1,\ldots,k$
$,$ we see that $\left\vert \mathbf{u}\right\vert _{p}=1$ and
\[
\lambda^{\left(  p\right)  }\left(  G\right)  \geq P_{G}\left(  \mathbf{u}%
\right)  =\left(  \sum_{i=1}^{k}\left(  \lambda^{\left(  p\right)  }\left(
G_{i}\right)  \right)  ^{t}\right)  ^{1/t}\left(  \sum_{i=1}^{k}\left(
\left\vert \mathbf{z}_{i}\right\vert _{p}^{r}\right)  ^{s}\right)
^{1/s}=\left(  \sum_{i=1}^{k}\left(  \lambda^{\left(  p\right)  }\left(
G_{i}\right)  \right)  ^{p/\left(  p-r\right)  }\right)  ^{\left(  p-r\right)
/p},
\]
completing the proof of (\ref{in2}).
\end{proof}

\subsection{Bounds on $\lambda^{\left(  p\right)  }$ in terms of order and
size}

For $2$-graphs it is known that $\lambda\left(  G\right)  \leq n-1,$ with
equality holding for complete graphs. Write $\left(  n\right)  _{r}$ for the
falling factorial $n\left(  n-1\right)  \cdots\left(  n-r+1\right)  .$
Maclaurin's and the PM inequalities (\ref{Maclin1}) and (\ref{PMin}) imply an
absolute upper bound on $\lambda^{\left(  p\right)  }\left(  G\right)  ;$ the
conditions for equality in these inequalities imply that\ the bound is
attained precisely for complete graphs.

\begin{proposition}
\label{pro_com}If $p\geq1$ and $G\in\mathcal{W}^{r}\left(  n\right)  ,$ then
$\lambda^{\left(  p\right)  }\left(  G\right)  \leq\left\vert G\right\vert
_{\infty}\left(  n\right)  _{r}/n^{r/p}.$ Equality holds if and only if $G$ is
a constant.

Let $n>r,$ and $G\in\mathcal{G}^{r}\left(  n\right)  $ be a complete graph. If
$r$ is even, then $\pm n^{-1/p}\mathbf{j}_{n}$ are the only eigenvectors to
$\lambda^{\left(  p\right)  }\left(  G\right)  $, and if $r$ is odd, the only
eigenvector is $n^{-1/p}\mathbf{j}_{n}$.
\end{proposition}

Although the above proposition elucidates the absolute maximum of
$\lambda^{\left(  p\right)  }$, the following fundamental bounds are more
flexible and usable.

\begin{theorem}
\label{pro1}Let $G\in\mathcal{W}^{r}\left(  n\right)  .$ If $p\geq1,$ then
\begin{equation}
\lambda^{\left(  p\right)  }\left(  G\right)  \geq r!\left\vert G\right\vert
/n^{r/p}. \label{ginl}%
\end{equation}
If $p>1,$ then
\begin{equation}
\lambda^{\left(  p\right)  }\left(  G\right)  \leq\left(  \left(  n\right)
_{r}/n^{r}\right)  ^{1/p}\left\vert r!G\right\vert _{p/\left(  p-1\right)  }.
\label{ginu}%
\end{equation}
If $p=1,$ then
\begin{equation}
\lambda^{\left(  1\right)  }\left(  G\right)  \leq\left(  n\right)  _{r}%
/n^{r}\left\vert G\right\vert _{\infty}. \label{ginu1}%
\end{equation}

\end{theorem}

\begin{proof}
Indeed, setting $\mathbf{x}=n^{-1/p}\mathbf{j}_{n}$ in (\ref{defla}), we
obtain
\[
\lambda^{\left(  p\right)  }\left(  G\right)  \geq r!\sum_{\left\{
i_{1},\ldots,i_{r}\right\}  \in E\left(  G\right)  }G\left(  \left\{
i_{1},\ldots,i_{r}\right\}  \right)  n^{-r/p}=\frac{r!\left\vert G\right\vert
}{n^{r/p}},
\]
proving (\ref{ginl}). Let now $\mathbf{x}=\left[  x_{i}\right]  $ be a
eigenvector to $\lambda^{\left(  p\right)  }\left(  G\right)  .$ H\"{o}lder's
inequality (\ref{Holdin}) with $p=1,$ $q=p$ and $k=\left\vert G\right\vert $
implies that
\begin{align*}
\sum_{\left\{  i_{1},\ldots,i_{r}\right\}  \in E}G\left(  \left\{
i_{1},\ldots,i_{r}\right\}  \right)  x_{i_{1}}\cdots x_{i_{r}}  &  \leq\left(
\sum_{e\in E}\left(  G\left(  e\right)  \right)  ^{p/\left(  p-1\right)
}\right)  ^{\left(  p-1\right)  /p}\left(  \sum_{\left\{  i_{1},\ldots
,i_{r}\right\}  \in E}\left\vert x_{i_{1}}\right\vert ^{p}\cdots\left\vert
x_{i_{r}}\right\vert ^{p}\right)  ^{1/p}\\
&  =\left\vert G\right\vert _{p/\left(  p-1\right)  }\left(  \sum_{\left\{
i_{1},\ldots,i_{r}\right\}  \in E}\left\vert x_{i_{1}}\right\vert ^{p}%
\cdots\left\vert x_{i_{r}}\right\vert ^{p}\right)  ^{1/p}.
\end{align*}
Now, letting $\mathbf{y}:=\left(  \left\vert x_{1}\right\vert ^{p}%
,\ldots,\left\vert x_{n}\right\vert ^{p}\right)  $ and applying Maclaurin's
inequality (\ref{Maclin1}) , we see that
\[
\frac{\mathbf{S}_{r}\left(  \mathbf{y}\right)  }{\binom{n}{r}}\leq\left(
\frac{\mathbf{S}_{1}\left(  \mathbf{y}\right)  }{n}\right)  ^{r}=\left(
\frac{\left\vert x_{1}\right\vert ^{p}+\cdots+\left\vert x_{n}\right\vert
^{p}}{n}\right)  ^{r}=n^{-r}.
\]
Therefore,%
\[
\lambda^{\left(  p\right)  }\left(  G\right)  \leq r!\left(  \binom{n}%
{r}/n^{r}\right)  ^{1/p}\left\vert G\right\vert _{p/\left(  p-1\right)  }%
\leq\left(  \left(  n\right)  _{r}/n^{r}\right)  ^{1/p}\left\vert
r!G\right\vert _{p/\left(  p-1\right)  },
\]
proving (\ref{ginu}).

Finally, if $p=1,$ inequality (\ref{ginu1}) follows by Maclaurin's inequality.
\end{proof}

Let us point out that the application of the PM and Maclaurin's inequalities
in the proof of (\ref{ginu}) is rather typical and works well for similar
upper bounds on $\lambda^{\left(  p\right)  }\left(  G\right)  $.

Simpler versions of inequality (\ref{ginl}) have been proved in \cite{CoDu11}
and \cite{PeZh12}. Note that (\ref{ginl}) generalizes the inequality of
Collatz and Sinogovits \cite{CoSi57} $\lambda\left(  G\right)  \geq2\left\vert
G\right\vert /n$ for $2$-graphs. Likewise, (\ref{ginu}) generalizes the
inequality of Wilf \cite{Wil86} for $2$-graphs:
\[
\lambda\left(  G\right)  \leq\sqrt{2\left(  1-1/n\right)  \left\vert
G\right\vert }.
\]

Since $0<\left(  n\right)  _{r}/n^{r}<1,$ from (\ref{ginu}) we obtain a
weaker, but simple and very usable inequality, involving just the $l^{1-1/p}$
norm of $G$.

\begin{corollary}
If $p>1$ and $G\in\mathcal{W}^{r}\left(  n\right)  ,$ then
\[
\lambda^{\left(  p\right)  }\left(  G\right)  \leq\left\vert r!G\right\vert
_{p/\left(  p-1\right)  }.
\]
If $\left\vert G\right\vert _{\infty}>0$, the inequality is strict. In
particular, if $G\in\mathcal{G}^{r}\left(  n\right)  ,$ then
\begin{equation}
\lambda^{\left(  p\right)  }\left(  G\right)  \leq r!\left\vert G\right\vert
^{1-1/p}. \label{upb}%
\end{equation}

\end{corollary}

Note that inequality (\ref{upb}) has been proved by Keevash, Lenz and Mubayi
in \cite{KLM13}. This useful bound is essentially tight as explained below.

\begin{proposition}
\label{cor}Let $p>1$ and $n\geq r.$ If $m\leq\binom{n}{r},$ there exists a
$G\in\mathcal{G}^{r}\left(  n\right)  $ such that
\[
\lambda^{\left(  p\right)  }\left(  G\right)  =\left(  1-o\left(  m\right)
\right)  \left(  r!\left\vert G\right\vert \right)  ^{1-1/p}.
\]

\end{proposition}

Indeed, given a natural $m,$ let $k$ satisfy $\binom{k-1}{r}<m\leq\binom{k}%
{r},$ and let $G\in\mathcal{G}^{r}\left(  n\right)  $ be a graph with
$\left\vert G\right\vert =m,$ such that $K_{k-1}^{r}\subset G\subset K_{k}%
^{r},$ i.e., $G$ has $n-k$ isolated vertices. Now, by (\ref{laed}) and
Proposition \ref{pro_s},
\[
\left(  k-1\right)  _{r}/\left(  k-1\right)  ^{r/p}=\lambda^{\left(  p\right)
}\left(  K_{k-1}^{r}\right)  \leq\lambda^{\left(  p\right)  }\left(  G\right)
\leq\lambda^{\left(  p\right)  }\left(  K_{k}^{r}\right)  =\left(  k\right)
_{r}/k^{r/p}.
\]

\subsection{\label{Fsec}$\lambda^{\left(  p\right)  }\left(  G\right)  $ as a
function of $p$}

Let us note that the introduction of the parameter $p$ in $\lambda^{\left(
p\right)  }$ and $\lambda_{\min}^{\left(  p\right)  }$ is not for want of
complications. Not only this parametrization is a constant source of new
insights, but also it plays a unification role. Indeed, assume that $G$ is a
fixed $r$-graph and consider $\lambda^{\left(  p\right)  }\left(  G\right)  $
as a function of $p.$ Since $\lambda^{\left(  p\right)  }\left(  G\right)  $
always has an eigenvector in $\mathbb{S}_{p,+}^{n-1}$, we can change the
variables in (\ref{defla}) obtaining the following equivalent definition of
$\lambda^{\left(  p\right)  }\left(  G\right)  :$
\begin{equation}
\lambda^{\left(  p\right)  }\left(  G\right)  :=\max_{\left\vert
y_{1}\right\vert +\cdots+\left\vert y_{n}\right\vert =1}r!\sum_{\left\{
i_{1},\ldots,i_{r}\right\}  \in E\left(  G\right)  }\left\vert y_{i_{1}%
}\right\vert ^{1/p}\cdots\left\vert y_{i_{r}}\right\vert ^{1/p}.
\label{altdef}%
\end{equation}
Now, this definition helps to see very clearly some essential features of
$\lambda^{\left(  p\right)  }\left(  G\right)  ,$ like the fact that
$\lambda^{\left(  p\right)  }\left(  G\right)  $ is increasing in $p.$ The
mean value theorem implies that the inequality%
\[
\left\vert y_{i_{1}}\right\vert ^{1/q}\cdots\left\vert y_{i_{r}}\right\vert
^{1/q}-\left\vert y_{i_{1}}\right\vert ^{1/p}\cdots\left\vert y_{i_{r}%
}\right\vert ^{1/p}\leq q-p
\]
whenever $q\geq p\geq1.$ Therefore,
\[
\left\vert \lambda^{\left(  q\right)  }\left(  G\right)  -\lambda^{\left(
p\right)  }\left(  G\right)  \right\vert \leq\left\vert q-p\right\vert
r!\left\vert G\right\vert ,
\]
whenever $q\geq1,$ $p\geq1$ and so $\lambda^{\left(  p\right)  }\left(
G\right)  $ is a Lipshitz function in $p.$ We get the following summary:

\begin{proposition}
If $G$ is a fixed $r$-graph and $p\geq1,$ then $\lambda^{\left(  p\right)
}\left(  G\right)  $ is increasing and continuous in $p.$ Also $\lambda
^{\left(  1\right)  }\left(  G\right)  <1$ and
\[
\lim_{p\rightarrow\infty}\lambda^{\left(  p\right)  }\left(  G\right)
=r!\left\vert G\right\vert .
\]

\end{proposition}

So, as a function of $p,$ $\lambda^{\left(  p\right)  }\left(  G\right)  $
seamlessly encompasses three fundamental graph parameters - the Lagrangian,
the spectral radius and the number of edges. Let us observe though that for
some graphs $\lambda^{\left(  p\right)  }\left(  G\right)  $ is not
continuously differentiable in $p.$ Indeed, if $G$ consists of two disjoint
complete $r$-graphs of order $n,$ then
\[
\lambda^{\left(  p\right)  }\left(  G\right)  =\left\{
\begin{array}
[c]{cc}%
\left(  n\right)  _{r}/n^{r/p} & \text{if }p\leq r;\text{ }\\
2\left(  n\right)  _{r}/\left(  2n\right)  ^{r/p} & \text{if }p\geq r.
\end{array}
\right.
\]
Note that the value of $\lambda^{\left(  p\right)  }\left(  G\right)  $ for
$p\geq r$ follows from Proposition \ref{pro_reg} below. Differentiating
$\lambda^{\left(  p\right)  }\left(  G\right)  $ for $p>r$ and $p<r,$ and
taking the limits as $p\rightarrow r,$ we see that
\begin{align*}
\lim_{p\rightarrow r^{+}}\frac{d}{dp}\left(  \lambda^{\left(  p\right)
}\left(  G\right)  \right)   &  =\lim_{p\rightarrow r^{+}}\frac{-2r\left(
n\right)  _{r}\log\left(  2n\right)  }{p^{2}\left(  2n\right)  ^{r/p}}%
=-\frac{\left(  n\right)  _{r}\log\left(  2n\right)  }{rn}\\
\lim_{p\rightarrow r^{-}}\frac{d}{dp}\left(  \lambda^{\left(  p\right)
}\left(  G\right)  \right)   &  =\lim_{p\rightarrow r^{-}}\frac{-r\left(
n\right)  _{r}\log\left(  n\right)  }{p^{2}\left(  n\right)  ^{r/p}}%
=-\frac{\left(  n\right)  _{r}\log n}{rn}.
\end{align*}
Hence $\lambda^{\left(  p\right)  }\left(  G\right)  $ is not continuously
differentiable at $r$. This situation is more general than it seems.

\begin{proposition}
Let $0\leq k\leq r-2$ and $G\in\mathcal{G}^{r}\left(  2r-k\right)  .$ It $G$
is a union of two edges sharing exactly $k$ vertices, then $\lambda^{\left(
p\right)  }\left(  G\right)  $ is not continuously differentiable at $p=r-k.$
\end{proposition}

However the following open questions seem relevant.

\begin{question}
Suppose that $G\in\mathcal{G}^{r}.$ Is $\lambda^{\left(  p\right)  }\left(
G\right)  $ continuously differentiable for $p>r?$ Is $\lambda^{\left(
p\right)  }\left(  G\right)  $ continuously differentiable for $p\neq k,$
$k=2,\ldots,r?$
\end{question}

In the following propositions we shall estimate how fast $\lambda^{\left(
p\right)  }\left(  G\right)  $ increases.

\begin{proposition}
\label{inch}If $p\geq1$ and $G\in\mathcal{G}^{r}\left(  n\right)  $, then the
function%
\[
h_{G}\left(  p\right)  :=\lambda^{\left(  p\right)  }\left(  G\right)
/n^{r/p}%
\]
is nonincreasing in $p,$ and
\[
\lim_{p\rightarrow\infty}h_{G}\left(  p\right)  =r!\left\vert G\right\vert .
\]

\end{proposition}

\begin{proof}
Let $p>q\geq1$ and let $\left[  x_{i}\right]  \in\mathbb{S}_{q}^{n-1}$ be an
eigenvector to $\lambda^{\left(  q\right)  }\left(  G\right)  .$ By the PM
inequality we have $\left\vert \left[  x_{i}\right]  \right\vert _{p}\geq
n^{1/p-1/q},$ and Proposition \ref{pro_b} implies that
\[
\lambda^{\left(  p\right)  }\left(  G\right)  \geq P_{G}\left(  \left[
x_{i}\right]  \right)  /\left\vert \left[  x_{i}\right]  \right\vert _{p}%
^{r}\geq\lambda^{\left(  q\right)  }\left(  G\right)  n^{r/q-r/p}.
\]

\end{proof}

Here is a similar statement involving the number of edges of $G,$ which can be
proved applying the PM inequality to the definition (\ref{altdef}).

\begin{proposition}
If $p\geq1$ and $G\in\mathcal{G}^{r}$, then the function
\[
f_{G}\left(  p\right)  :=\left(  \frac{\lambda^{\left(  p\right)  }\left(
G\right)  }{r!\left\vert G\right\vert }\right)  ^{p}%
\]
is nonincreasing in $p.$
\end{proposition}

\medskip

\subsection{$\lambda_{\min}^{\left(  p\right)  }\left(  G\right)  $ as a
function of $p$}

Some of the above properties of $\lambda^{\left(  p\right)  }$ can be proved
also for $\lambda_{\min}^{\left(  p\right)  }$. Thus, assume that $G$ is a
fixed $r$-graph and consider $\lambda^{\left(  p\right)  }\left(  G\right)  $
as a function of $p.$\ Taking an eigenvector $\left[  x_{i}\right]
\in\mathbb{S}_{p}^{n-1}$ and changing the variables by the one-to-one
correspondence $x_{i}\rightarrow y_{i}\left\vert y_{i}\right\vert ^{1/p-1},$
we see that%
\[
\lambda_{\min}^{\left(  p\right)  }\left(  G\right)  =\min_{\left\vert
y_{1}\right\vert +\cdots+\left\vert y_{n}\right\vert =1}r!\sum_{\left\{
i_{1},\ldots,i_{r}\right\}  \in E\left(  G\right)  }y_{i_{1}}\cdots y_{i_{r}%
}\left\vert y_{i_{1}}\cdots y_{i_{r}}\right\vert ^{1/p-1}.
\]
Some algebra gives that
\[
\left\vert \lambda_{\min}^{\left(  q\right)  }\left(  G\right)  -\lambda
_{\min}^{\left(  p\right)  }\left(  G\right)  \right\vert <\left\vert
q-p\right\vert r!\left\vert G\right\vert ,
\]
which implies also the following proposition.

\begin{proposition}
\label{bas_p}If $G$ is a fixed $r$-graph and $p\geq1,$ then $\lambda_{\min
}^{\left(  p\right)  }\left(  G\right)  $ is decreasing and continuous in $p.$
Also, $\lambda_{\min}^{\left(  1\right)  }\left(  G\right)  \geq-1$ and the
limit $\lim_{p\rightarrow\infty}\lambda_{\min}^{\left(  p\right)  }\left(
G\right)  $ exists.
\end{proposition}

One can figure our a description of the limit $\lim_{p\rightarrow\infty
}\lambda_{\min}^{\left(  p\right)  }\left(  G\right)  ,$ but its combinatorial
significance is not completely clear.\ If $G$ has an odd transversal, then
$\lim_{p\rightarrow\infty}\lambda_{\min}^{\left(  p\right)  }\left(  G\right)
=-r!\left\vert G\right\vert $, see Theorem \ref{thOT} below. For $2$-graphs
this is equivalent to $G$ being bipartite.

Like for $\lambda^{\left(  p\right)  }\left(  G\right)  ,$ we have the
following estimate for the rate of change of $\lambda_{\min}^{\left(
p\right)  }\left(  G\right)  $.

\begin{proposition}
If $p\geq1$ and $G\in\mathcal{G}^{r}\left(  n\right)  $, then the function
$g_{G}\left(  p\right)  :=\lambda_{\min}^{\left(  p\right)  }\left(  G\right)
/n^{r/p}$ is nondecreasing in $p.$
\end{proposition}

Let $0\leq k\leq r-2.$ Taking $G\in\mathcal{G}^{r}\left(  2r-k\right)  $ to be
the union of two edges sharing exactly $k$ vertices, we get a graph with
$\lambda_{\min}^{\left(  p\right)  }\left(  G\right)  =-\lambda^{\left(
p\right)  }\left(  G\right)  ,$ and, as above, we see that $\lambda_{\min
}^{\left(  p\right)  }\left(  G\right)  $ is not continuously differentiable
at $p=r-k.$ Here is a natural question.

\begin{question}
Suppose that $G\in\mathcal{G}^{r}.$ Is $\lambda_{\min}^{\left(  p\right)
}\left(  G\right)  $ continuously differentiable for $p>r?$ Is $\lambda_{\min
}^{\left(  p\right)  }\left(  G\right)  $ continuously differentiable for
$p\neq k,$ $k=2,\ldots,r?$
\end{question}

\section{\label{eqs}Eigenequations}

The Rayleigh principle and the Courant-Fisher inequalities allow to define
eigenvalues of Hermitian matrices as critical values of quadratic forms over
the unit sphere $\mathbb{S}_{2}^{n-1}$. From this variational definition the
standard definition via linear equations can be recovered using Lagrange
multipliers. We follow the same path for eigenvalues of hypergraphs; in our
case it is particularly simple because we are interested mostly in the largest
and the smallest eigenvalues. Thus, next we shall show that the variational
definitions (\ref{defla}) and (\ref{defsa}) lead to systems of equations
arising from Lagrange multipliers.

\subsection{The system of eigenequations}

Suppose that $G\in\mathcal{W}^{r}\left(  n\right)  $ and let $\left[
x_{i}\right]  \in\mathbb{S}_{p}^{n-1}$ be an eigenvector to $\lambda^{\left(
p\right)  }\left(  G\right)  $. If $p>1,$ the function
\[
g\left(  y_{1},\ldots,y_{n}\right)  :=\left\vert y_{1}\right\vert ^{p}%
+\ldots+\left\vert y_{n}\right\vert ^{p}%
\]
has continuous partial derivatives in each variable (see \ref{Nors}). Thus,
using Lagrange's method (Theorem \ref{LMT}), there exists a $\mu$ such that
for each $k=1,\ldots,n,$
\[
\mu px_{k}\left\vert x_{k}\right\vert ^{p-2}=\frac{\partial P_{G}\left(
\left[  x_{i}\right]  \right)  }{\partial x_{k}}=r!\sum_{\left\{
k,i_{1},\ldots,i_{r-1}\right\}  \in E\left(  G\right)  }G\left(  \left\{
k,i_{1},\ldots,i_{r-1}\right\}  \right)  x_{i_{1}}\cdots x_{i_{r-1}}.
\]
Now, multiplying the $k$'th equation by $x_{k}$ and adding them all, we find
that
\[
\mu p=\mu p\sum_{k\in V\left(  G\right)  }\left\vert x_{k}\right\vert
^{p}=\sum_{k\in V\left(  G\right)  }x_{k}\frac{\partial P_{G}\left(  \left[
x_{i}\right]  \right)  }{\partial x_{k}}=rP_{G}\left(  \left[  x_{i}\right]
\right)  =r\lambda^{\left(  p\right)  }\left(  G\right)  .
\]
Hence, we arrive at the following theorem.

\begin{theorem}
Let $G\in\mathcal{W}^{r}\left(  n\right)  $ and $p>1.$ If $\left[
x_{i}\right]  \in\mathbb{S}_{p}^{n-1}$ is an eigenvector to $\lambda^{\left(
p\right)  }\left(  G\right)  ,$ then $x_{1},\ldots,x_{n}$ satisfy the
equations%
\begin{equation}
\lambda^{\left(  p\right)  }\left(  G\right)  x_{k}\left\vert x_{k}\right\vert
^{p-2}=\frac{1}{r}\frac{\partial P_{G}\left(  \left[  x_{i}\right]  \right)
}{\partial x_{k}},\text{ \ \ \ \ }k=1,\ldots,n. \label{eequ}%
\end{equation}

\end{theorem}

For $p>1$ equations (\ref{eequ}) are a powerful tool in the study of
$\lambda^{\left(  p\right)  }\left(  G\right)  ,$ but since $\left\vert
x\right\vert $ is not differentiable at $0$, they are not always available for
$p=1.$

A similar argument for $\lambda_{\min}^{\left(  p\right)  }\left(  G\right)  $
leads to the following theorem.

\begin{theorem}
Let $G\in\mathcal{W}^{r}\left(  n\right)  $ and $p>1.$ If $\left[
x_{i}\right]  \in\mathbb{S}_{p}^{n-1}$ is an eigenvector to $\lambda_{\min
}^{\left(  p\right)  }\left(  G\right)  ,$ then $x_{1},\ldots,x_{n}$ satisfy
the equations%
\begin{equation}
\lambda_{\min}^{\left(  p\right)  }\left(  G\right)  x_{k}\left\vert
x_{k}\right\vert ^{p-2}=\frac{1}{r}\frac{\partial P_{G}\left(  \left[
x_{i}\right]  \right)  }{\partial x_{k}},\text{ \ \ \ \ }k=1,\ldots,n.
\label{eequsa}%
\end{equation}

\end{theorem}

Note that if $G$ is a $2$-graph with adjacency matrix $A,$ then (\ref{eequ})
and (\ref{eequsa}) reduce to the familiar equations
\[
A\mathbf{x}=\lambda\left(  G\right)  \mathbf{x}\text{ \ \ and \ \ \ }%
A\mathbf{y}=\lambda_{\min}\left(  G\right)  \mathbf{y}\text{.}%
\]
Therefore, we shall call equations (\ref{eequ}) and (\ref{eequsa}) the
\emph{eigenequations}\textbf{ }for $\lambda^{\left(  p\right)  }\left(
G\right)  $ and for $\lambda_{\min}^{\left(  p\right)  }\left(  G\right)  .$
In general, given $G\in\mathcal{W}^{r}\left(  n\right)  ,$ there may be many
different real or complex numbers $\lambda$ and $n$-vectors $\left[
x_{i}\right]  $ with $\left\vert \left[  x_{i}\right]  \right\vert _{p}=1$
satisfying the equations
\begin{equation}
\lambda x_{k}\left\vert x_{k}\right\vert ^{p-2}=\frac{1}{r}\frac{\partial
P_{G}\left(  \left[  x_{i}\right]  \right)  }{\partial x_{k}},\text{
\ \ \ \ }k=1,\ldots,n. \label{alga}%
\end{equation}
This multiplicity may remain even if we impose additional restrictions, like
$\left[  x_{i}\right]  \geq0$ or $\left[  x_{i}\right]  >0$, or $G$ being a
connected $r$-graph. Nevertheless, having a unique solution $\left(
\lambda,\left[  x_{i}\right]  \right)  $ to (\ref{alga}) is highly desirable;
the Perron-Frobenius type theory developed in Section \ref{PFsec} provides
some conditions that guarantee this property.

\subsection{Algebraic definitions of eigenvalues}

In this subsection we discuss some algebraic definitions of hypergraph
eigenvalues along the work of . Qi \cite{Qi05}, who proposed to define
eigenvalues of hypermatrices using equations similar to (\ref{alga}). If the
definition of Qi is applied to a graph $G\in\mathcal{W}^{r}\left(  n\right)
$, then an eigenvalue of $G$ is a complex number $\lambda$ which satisfies the
equation%
\begin{equation}
\lambda x_{k}^{r-1}=\frac{1}{r}\frac{\partial P_{G}\left(  \left[
x_{i}\right]  \right)  }{\partial x_{k}},\ \ \ \ k=1,\ldots,n, \label{alg}%
\end{equation}
for some nonzero vector $\left[  x_{i}\right]  \in\mathbb{C}^{n}.$ When
$x_{1},\ldots,x_{n}$ are real, $\lambda$ is also real and is called an $H$-eigenvalue.

Another definition suggested by Friedman and Wigderson \cite{FrWi95} and
developed by Qi defines a graph eigenvalues as solutions $\lambda$ of the
system%
\[
\lambda x_{k}=\frac{1}{r}\frac{\partial P_{G}\left(  \left[  x_{i}\right]
\right)  }{\partial x_{k}},\ \ \ \ k=1,\ldots,n,
\]
for some complex $x_{1},\ldots,x_{n}$ with $\left\vert x_{1}\right\vert
^{2}+\cdots+\left\vert x_{k}\right\vert ^{2}=1.$ In this case, $\lambda$ is
called an $E$-eigenvalue of $G;$ if $x_{1},\ldots,x_{n}$ are real, then
$\lambda$ is called a $Z$-eigenvalue of $G$.

It is not hard to see that these definitions fit with our setup for
$\lambda^{\left(  p\right)  }\left(  G\right)  $ and $\lambda_{\min}^{\left(
p\right)  }\left(  G\right)  $. Indeed, since $\lambda\left(  G\right)  $
satisfies (\ref{alg}) with a vector $\left[  x_{i}\right]  \in\mathbb{S}%
_{p,+}^{n-1},$ it is an $H$-eigenvalue in the definition of Qi; moreover,
$\lambda\left(  G\right)  $ has the largest absolute value among all
eigenvalues defined by (\ref{alg}).

\begin{proposition}
\label{pro_max}Let $G\in\mathcal{W}^{r}\left(  n\right)  .$ If the complex
number $\lambda$ satisfies the equations (\ref{alg}) for some nonzero complex
vector $\left[  x_{i}\right]  ,$ then $\left\vert \lambda\right\vert
\leq\lambda\left(  G\right)  .$
\end{proposition}

\begin{proof}
Suppose that $\lambda\in\mathbb{C}$ and $\left[  x_{i}\right]  \in
\mathbb{C}^{n}$ satisfy the system (\ref{alg}). For $k=1,\ldots,n$ the
triangle inequality implies that
\[
\left\vert \lambda\right\vert \left\vert x_{k}\right\vert ^{r-1}=\frac{1}%
{r}\left\vert \frac{\partial P_{G}\left(  \left[  x_{i}\right]  \right)
}{\partial x_{k}}\right\vert \leq\left(  r-1\right)  !\sum_{\left\{
k,i_{1},\ldots,i_{r-1}\right\}  \in E\left(  G\right)  }\left\vert x_{i_{1}%
}\right\vert \cdots\left\vert x_{i_{r-1}}\right\vert ,\ \ \ \ k=1,\ldots,n.
\]
Multiplying the $k$'th inequality by $\left\vert x_{k}\right\vert $ and adding
them all, we obtain
\[
\left\vert \lambda\right\vert \left(  \left\vert x_{1}\right\vert ^{r}%
+\cdots+\left\vert x_{n}\right\vert ^{r}\right)  \leq P_{G}\left(  \left(
\left\vert x_{1}\right\vert ,\ldots,\left\vert x_{n}\right\vert \right)
\right)  \leq\lambda\left(  G\right)  \left(  \left\vert x_{1}\right\vert
^{r}+\cdots+\left\vert x_{n}\right\vert ^{r}\right)  ,
\]
implying the assertion.
\end{proof}

Similarly $\lambda^{\left(  2\right)  }\left(  G\right)  $ is unique among all
$E$-eigenvalues; in fact, the proof is valid for all $p>1$: \emph{If the
complex number }$\lambda$\emph{ satisfies the equations}
\[
\lambda x_{k}\left\vert x_{k}\right\vert ^{p-2}=\frac{1}{r}\frac{\partial
P_{G}\left(  \left[  x_{i}\right]  \right)  }{\partial x_{k}}%
,\ \ \ \ k=1,\ldots,n,
\]
\emph{for some complex vector }$\left[  x_{i}\right]  $\emph{ with
}$\left\vert \left[  x_{i}\right]  \right\vert _{p}=1,$\emph{ then
}$\left\vert \lambda\right\vert \leq\lambda^{\left(  p\right)  }\left(
G\right)  .$

In the same spirit, one can show that $\lambda_{\min}\left(  G\right)  $ is
the smallest real solution to (\ref{alg}); we extend this fact for
$\lambda_{\min}^{\left(  p\right)  }\left(  G\right)  $.

\begin{proposition}
Let $G\in\mathcal{W}^{r}\left(  n\right)  $ and $p>1.$ If the real number
$\lambda$ and $\left[  x_{i}\right]  \in\mathbb{S}_{p}^{n-1}$ satisfy the
equations%
\[
\lambda_{\min}^{\left(  p\right)  }\left(  G\right)  x_{k}\left\vert
x_{k}\right\vert ^{p-2}=\frac{1}{r}\frac{\partial P_{G}\left(  \left[
x_{i}\right]  \right)  }{\partial x_{k}},\ \ \ \ k=1,\ldots,n,
\]
then $\lambda_{\min}^{\left(  p\right)  }\left(  G\right)  \leq\lambda.$ In
particular, if $p=r,$ and $\lambda$ and $\left[  x_{i}\right]  \in
\mathbb{S}_{p}^{n-1}$ satisfy the equations (\ref{alg}), then $\lambda_{\min
}\left(  G\right)  \leq\lambda.$
\end{proposition}

\subsection{\label{regs}Regular graphs and $\lambda^{\left(  p\right)  }$}

A weighted graph $G\in\mathcal{W}^{r}\left(  n\right)  $ is called (vertex)
\emph{regular} if all vertex degrees are equal, i.e., all vertex degrees are
equal to $r\left\vert G\right\vert /n.$ It is easy to see that for every
regular graph $G\in\mathcal{W}^{r}\left(  n\right)  ,$ there is a positive
$\lambda$ satisfying the eigenequations for $\lambda^{\left(  p\right)
}\left(  G\right)  $.

\begin{proposition}
If $G\in\mathcal{W}^{r}\left(  n\right)  $ is regular, then for every $p>1$
the number $\lambda=r!\left\vert G\right\vert n^{-r/p}$ and the vector
$\left[  x_{i}\right]  =n^{-1/p}\mathbf{j}_{n}$ satisfy the equations%
\begin{equation}
\lambda x_{k}^{p-1}=\frac{1}{r}\frac{\partial P_{G}\left(  \left[
x_{i}\right]  \right)  }{\partial x_{k}},\text{ \ \ \ \ }k=1,\ldots,n.
\label{eqr}%
\end{equation}
Conversely if for some $p>1$ there is a number $\lambda>0$ such that $\left[
x_{i}\right]  =n^{-1/p}\mathbf{j}_{n},$ satisfy the equations (\ref{eqr}),
then $G$ is regular.
\end{proposition}

We can easily relate these simple observations to $\lambda^{\left(  p\right)
}\left(  G\right)  .$

\begin{proposition}
If $G\in\mathcal{W}^{r}\left(  n\right)  $ and $\lambda^{\left(  p\right)
}\left(  G\right)  =r!\left\vert G\right\vert /n^{r/p}$ for some $p>1,$ then
$G$ is regular.
\end{proposition}

Moreover, if $p\geq r,$ the converse of this statement is true as well, and
Proposition \ref{procr} shows that the cycles $C_{n}^{r}$ are counterexamples
if $1<p<r.$

\begin{proposition}
\label{pro_reg}If $p\geq r$ and $G\in\mathcal{W}^{r}\left(  n\right)  $ is
regular, then $\lambda^{\left(  p\right)  }\left(  G\right)  =r!\left\vert
G\right\vert /n^{r/p}.$
\end{proposition}

We omit this proof as the statement is easy to prove, first for $p=r,$ and
then the general case by Proposition \ref{inch}. An important open problem
here is the following one:

\begin{problem}
\label{Pro_reg}Let $r\geq2,$ $1<p<r.$ Characterize all regular graphs
$G\in\mathcal{G}^{r}\left(  n\right)  $ such that
\begin{equation}
\lambda^{\left(  p\right)  }\left(  G\right)  =r!\left\vert G\right\vert
/n^{r/p}. \label{eq1}%
\end{equation}

\end{problem}

For example, if $G$ is a complete or a complete multipartite $r$-graph, or
$\left(  r-1\right)  $-set regular, then equality holds in (\ref{eq1}). On the
other hand, if $G$ is relatively sparse, like $\left\vert G\right\vert
=o\left(  n^{r/p}\right)  ,$ then (\ref{eq1}) fails for sure. Indeed, if
$\left\vert G\right\vert =o\left(  n^{r/p}\right)  ,$ then for $n$
sufficiently large,
\[
\lambda^{\left(  p\right)  }\left(  G\right)  \geq\lambda^{\left(  p\right)
}\left(  K_{r}\right)  =r!r^{r/p}>r!\left\vert G\right\vert n^{-r/p}.
\]
But note that (\ref{eq1}) may fail even if $G$ is quite dense; e.g., if $G$ is
the disjoint union of two complete $r$-graphs of order $n$, then for $n$
sufficiently large,%
\[
\lambda^{\left(  p\right)  }\left(  G\right)  =\left(  n\right)
_{r-1}n^{-r/p}>\frac{2\left(  n\right)  _{r}}{2^{r/p}n^{r/p}}=r!\cdot
2\binom{n}{r}\left(  2n\right)  ^{-r/p}=r!\left\vert G\right\vert \left(
2n\right)  ^{-r/p}.
\]
Therefore, it seems that Problem \ref{Pro_reg} is quite important, insofar
that its complete solution would most certainly relate $\lambda^{\left(
p\right)  }\left(  G\right)  $ to the local edge density of $G.$\bigskip

\subsection{Symmetric vertices and eigenvectors}

Let $G\in\mathcal{W}^{r}$ and let $u,v\in V\left(  G\right)  .$ For practical
calculations we wish to have structural conditions on $G,$ which would
guarantee that $x_{u}=x_{v}$ for any eigenvector $\left[  x_{i}\right]  $ to
$\lambda^{\left(  p\right)  }\left(  G\right)  .$ Thus, we say that $u$ and
$v$ are \emph{equivalent in }$G,$ in writing $u\sim v,$ if transposing $u$ and
$v$ and leaving the remaining vertices intact we get an automorphism of $G.$
Obviously, $u\sim v$ if every edge $e\in E\left(  G\right)  $ such that
$e\cap\left\{  u,v\right\}  \neq\varnothing$ satisfies
\[
\left\{  u,v\right\}  \subset e\text{ or }\left(  e\backslash\left\{
v\right\}  \right)  \cup\left\{  u\right\}  \in E\left(  G\right)  \text{ or
}\left(  e\backslash\left\{  u\right\}  \right)  \cup\left\{  v\right\}  \in
E\left(  G\right)  .
\]
Now equations (\ref{eequ}) imply the following lemma.

\begin{lemma}
\label{eqth}Let $G\in\mathcal{W}^{r}\left(  n\right)  $ and let $u\sim v.$ If
$p>1$ and $\left[  x_{i}\right]  \in\mathbb{S}_{p}^{n-1}$ is an eigenvector to
$\lambda^{\left(  p\right)  }\left(  G\right)  $ or to $\lambda_{\min
}^{\left(  p\right)  }\left(  G\right)  ,$ then $x_{u}=x_{v.}$.
\end{lemma}

\begin{proof}
Write $\lambda$ for $\lambda^{\left(  p\right)  }\left(  G\right)  $ or
$\lambda_{\min}^{\left(  p\right)  }\left(  G\right)  $ and let $\left[
x_{i}\right]  \in\mathbb{S}_{p}^{n-1}$ be an eigenvector to $\lambda.$ We
have
\[
\lambda x_{u}\left\vert x_{u}\right\vert ^{p-2}=\frac{1}{r}\frac{\partial
P_{G}\left(  \left[  x_{i}\right]  \right)  }{\partial x_{u}}\text{ \ \ \ and
\ \ }\lambda x_{v}\left\vert x_{v}\right\vert ^{p-2}=\frac{1}{r}\frac{\partial
P_{G}\left(  \left[  x_{i}\right]  \right)  }{\partial x_{v}}.
\]
Hence, using that $u$ and $v$ are equivalent, we see that
\[
\lambda x_{u}\left\vert x_{u}\right\vert ^{p-2}-\lambda x_{v}\left\vert
x_{v}\right\vert ^{p-2}=\left(  x_{v}-x_{u}\right)  \left(  r-1\right)
!\sum_{\left\{  u,v,i_{1},\ldots,i_{r-2}\right\}  }G\left(  \left\{
u,v,i_{1},\ldots,i_{r-2}\right\}  \right)  x_{i_{1}}\cdots x_{i_{r-2}.}%
\]
Since the function $f\left(  x\right)  :=x\left\vert x\right\vert ^{p-2}$ is
increasing in $x$ for every real $x,$ we see that $x_{v}-x_{u}=0,$ completing
the proof.
\end{proof}

Note that the symmetric vertices in general do not have equal entries, see
Proposition \ref{procr}, below. Lemma \ref{eqth} implies a practical statement
very similar to Corollary 12, in \cite{KLM13}.

\begin{corollary}
\label{corEX}Let $G\in\mathcal{W}^{r}\left(  n\right)  .$ If $V\left(
G\right)  $ be partitioned into equivalence classes by the relation
\textquotedblleft$\sim$\textquotedblright, then every eigenvector $\left[
x_{i}\right]  \in\mathbb{S}_{p}^{n-1}$ to $\lambda^{\left(  p\right)  }\left(
G\right)  $ or to $\lambda_{\min}^{\left(  p\right)  }\left(  G\right)  $ is
constant within each equivalence class.
\end{corollary}

The above corollary can be quite useful in calculating or estimating
$\lambda^{\left(  p\right)  }\left(  G\right)  ;$ for instance, to calculate
$\lambda^{\left(  p\right)  }$ and $\lambda_{\min}^{\left(  p\right)  }$ of a
$\beta$-star.

\begin{proposition}
\label{betas}Let $G\in\mathcal{G}^{r}\left(  \left(  r-1\right)  k+1\right)  $
and let $G\ $consist of $k$ edges sharing a single vertex. If $p\geq r-1,$
then $\lambda^{\left(  p\right)  }\left(  G\right)  =\left(  r!/r^{r/p}%
\right)  k^{1-\left(  r-1\right)  /p}.$ Also, $\lambda_{\min}^{\left(
p\right)  }\left(  G\right)  =-\left(  r!/r^{r/p}\right)  k^{1-\left(
r-1\right)  /p}$
\end{proposition}

\medskip

\section{\label{facts}Some warning illustrations}

Let us note again that best-known case of $\lambda^{\left(  p\right)  }\left(
G\right)  $ of a graph $G\in\mathcal{W}^{r}$, $\left(  r\geq2,p\geq1\right)
,$ is the largest eigenvalue of a $2$-graph. However, expectations based on
eigenvalues of $2$-graphs can collide with the real properties of
$\lambda^{\left(  p\right)  }\left(  G\right)  $ if $r\geq3$ or if $r=2$ and
$p\neq2.$ The purpose of the examples below is to deflect some wrong
expectations, and at the same time to outline limitations to Perron-Frobenius
type properties for $\lambda^{\left(  p\right)  }$.

\subsection{Zero always satisfies the eigenequations}

Let start with a simple observation. If $r\geq3,$ any vector with at most
$r-2$ nonzero entries satisfies the eigenequations (\ref{alga}) with
$\lambda=0;$ this follows trivially as every edge consists of $r$ distinct vertices.

\begin{proposition}
\label{zeroev}Let $n\geq r\geq3$ and $p>1.$ For every $G\in\mathcal{W}%
^{r}\left(  n\right)  ,$ there are $n$ linearly independent nonnegative
solutions $\left[  x_{i}\right]  $ to the equations
\[
0\cdot x_{k}^{p-1}=\frac{1}{r}\frac{\partial P_{G}\left(  \left[
x_{i}\right]  \right)  }{\partial x_{k}},\text{ \ \ \ \ }k=1,\ldots,n.
\]

\end{proposition}

This fact is impossible for $2$-graphs with edges, but it is unavoidable for
any uniform hypergraph.

\subsection{Strange eigenvectors in graphs with two edges}

In this subsection we discuss $r$-graphs formed by two edges with exactly $k$
vertices in common. This simple construction will give examples of
eigenvectors to $\lambda^{\left(  p\right)  }\left(  G\right)  ,$ which are
abnormal from the viewpoint of $2$-graphs, but natural for hypergraphs. These
examples also outline the scope of validity of some theorems in Section
\ref{PFsec}. Finally the reader can practice simple methods for evaluating
$\lambda^{\left(  p\right)  }$ and $\lambda_{\min}^{\left(  p\right)  }.$

First we shall discuss in some detail the following $3$-graph.

\begin{proposition}
\label{bad3}Let $G\in\mathcal{G}_{3}\left(  5\right)  $ and $G$ consist of two
edges sharing a single vertex. We have%
\[
\lambda^{\left(  2\right)  }\left(  G\right)  =2/\sqrt{3},
\]
and $\lambda^{\left(  2\right)  }\left(  G\right)  $ has infinitely many
positive eigenvectors and two nonnegative ones with zero entries.

If $1<p<2,$ then
\[
\lambda^{\left(  p\right)  }\left(  G\right)  =6\cdot3^{-3/p},
\]
and $\lambda^{\left(  p\right)  }\left(  G\right)  $ has no positive
eigenvector. There exists a positive $\lambda<$ $\lambda^{\left(  p\right)
}\left(  G\right)  $ and a positive vector $\left[  x_{i}\right]
\in\mathbb{S}_{p}^{4}$ satisfying the eigenequations for $\lambda^{\left(
p\right)  }\left(  G\right)  .$
\end{proposition}

\begin{proof}
Let $V\left(  G\right)  =\left[  5\right]  $ and let the two edges of $G$ be
$\left\{  1,2,3\right\}  $ and $\left\{  3,4,5\right\}  .$ Suppose that
$\left(  x_{1},x_{2},x_{3},x_{4},x_{5}\right)  \in\mathbb{S}_{2,+}^{4}$ is an
eigenvector to $\lambda^{\left(  2\right)  }\left(  G\right)  .$ Theorem
\ref{eqth} implies that $x_{1}=x_{2}$ and $x_{4}=x_{5}.$ Setting $x_{3}=x,$ we
have
\[
\lambda^{\left(  2\right)  }\left(  G\right)  =3!\max\left(  x_{1}x_{2}%
x_{3}+x_{3}x_{4}x_{5}\right)  =3!\cdot\max\left(  x_{1}^{2}+x_{4}^{2}\right)
x=3!\cdot\max_{0<x\leq1}\left(  \frac{1-x^{2}}{2}\right)  x=\frac{2}{\sqrt{3}%
}.
\]
Clearly, the maximum is attained for any vector $\left(  s,s,3^{-1/2}%
,t,t\right)  ,$ with $s^{2}+t^{2}=1/3.$

Let now $1<p<2$ and $\left(  x_{1},x_{2},x_{3},x_{4},x_{5}\right)
\in\mathbb{S}_{p,+}^{4}$ is an eigenvector to $\lambda^{\left(  p\right)
}\left(  G\right)  .$ We see that
\[
\lambda^{\left(  p\right)  }\left(  G\right)  =3!\max_{\left\vert
\mathbf{x}\right\vert _{p}=1}\left(  x_{1}x_{2}x_{3}+x_{3}x_{4}x_{5}\right)
=3!\max_{\left\vert \mathbf{x}\right\vert _{p}=1}\left(  x_{1}^{2}+x_{4}%
^{2}\right)  x_{3}.
\]
If $x_{3}$ is fixed, then $\max x_{1}^{2}+x_{4}^{2}$ subject to $x_{1}%
^{p}+x_{4}^{p}=\left(  1-x_{3}^{p}\right)  /2$ is attained if $x_{1}=0$ or
$x_{4}=0$ because $f\left(  y\right)  =y^{2/p\text{ }}$ is a convex function.
Therefore, $\lambda^{\left(  p\right)  }\left(  G\right)  $ has no positive
eigenvector. By (\ref{laed}) we find that
\[
\lambda^{\left(  p\right)  }\left(  G\right)  =3!\cdot3^{-3/p}=\frac
{6}{3^{3/p}}.
\]

\end{proof}

We notice that in the above proposition $G$ is connected, but $\lambda
^{\left(  2\right)  }\left(  G\right)  $ has infinitely many positive
eigenvectors, and if $1<p<2,$ then $\lambda^{\left(  p\right)  }\left(
G\right)  $ has no positive eigenvector at all. This is impossible for the
largest eigenvalue of a connected $2$-graph.

For $r\geq3$ the example of Proposition \ref{bad3} can be generalized as follows.

\begin{proposition}
\label{badl}Let $r\geq3,$ $1\leq k\leq r-2,$ let $G\in\mathcal{G}^{r}\left(
2r-k\right)  $ and let $G$ consists of two edges sharing precisely $k$
vertices. We have%
\[
\lambda^{\left(  r-k\right)  }\left(  G\right)  =\left(  r-1\right)
!/r^{1\left(  r-1\right)  },
\]
and the set of eigenvectors to $\lambda^{\left(  r-k\right)  }\left(
G\right)  $ contains a circle; in particular, infinitely many positive vectors
and two eigenvectors with $0$ entries.

If $1<p<r-k,$ then
\[
\lambda^{\left(  p\right)  }\left(  G\right)  =r!r^{-r/p}%
\]
and $\lambda^{\left(  p\right)  }\left(  G\right)  $ has no positive
eigenvector. Moreover, there exists a positive $\lambda<$ $\lambda^{\left(
p\right)  }\left(  G\right)  $ and a positive $\mathbf{x}\in\mathbb{S}%
_{p,+}^{2r-k-1},$ satisfying the eigenequations for $\lambda^{\left(
p\right)  }\left(  G\right)  $.
\end{proposition}

The weird properties of the examples in Proposition \ref{badl} seem due more
to the fact that $p\leq r-1$ than to the fact $r\geq3.$ However, a similar
phenomenon is observed also for $\lambda\left(  G\right)  =\lambda^{\left(
r\right)  }\left(  G\right)  $ as described below.

\begin{proposition}
\label{secl}Let $r\geq3$ and $G\in\mathcal{G}^{r}\left(  r+2\right)  .$ If $G$
consists of two edges sharing precisely $r-2$ vertices, then $\lambda=\left(
r-1\right)  !$ and the vector $\mathbf{x}=\left(  r^{-1/r},\ldots
,r^{-1/r},0,0\right)  $ satisfy the eigenequations for $\lambda\left(
G\right)  $%
\[
\lambda x_{k}^{r-1}=\frac{1}{r}\frac{\partial P_{G}\left(  \mathbf{x}\right)
}{\partial x_{k}},\text{ \ \ \ \ }k=1,\ldots,r+2,
\]
but $\lambda\left(  G\right)  >\left(  r-1\right)  !.$
\end{proposition}

\subsection{\label{Cycs}Strange eigenvectors in cycles}

The purpose of this subsection is to show that regular, connected graphs as
simple as cycles can also have strange eigenvectors of $\lambda^{\left(
p\right)  }.$ Along this example we also answer a question of Pearson and Zhang.

Let us begin with raising a question about $\lambda^{\left(  p\right)  }$ of
$2$-cycles:

\begin{question}
If $C_{n}$\emph{ }is the $2$-cycle of order\emph{ }$n$ and $1<p<2,$ what is
$\lambda^{\left(  p\right)  }\left(  C_{n}\right)  ?$
\end{question}

To answer this question one has to find
\[
\max_{\left\vert x_{1}\right\vert ^{p}+\cdots+\left\vert x_{n}\right\vert
^{p}}x_{1}x_{2}+\cdots+x_{n-1}x_{n}+x_{n}x_{1}.
\]
which is quite challenging for $n>4$. However, for $n=4$ there is a definite answer.

\begin{proposition}
If $p\geq1,$ then $\lambda^{\left(  p\right)  }\left(  C_{4}\right)
=2^{3-4/p}.$ If $p>1,$ the only nonnegative eigenvector to $\lambda^{\left(
p\right)  }\left(  C_{4}\right)  $ is $4^{-1/p}\mathbf{j}_{4}.$
\end{proposition}

Although we do not know $\lambda^{\left(  p\right)  }\left(  C_{n}\right)  $
precisely, we still can draw a number of puzzling conclusions. Since
$C_{n\text{ }}$ is connected it is not hard to see that every nonnegative
eigenvector to $\lambda^{\left(  p\right)  }\left(  C_{n}\right)  $ is
positive as shown in Theorem \ref{PFa0}. However, if $1<p<2$, then for $n$
sufficiently large the vector $n^{-1/p}\mathbf{j}_{n}$ is not an eigenvector
to $\lambda^{\left(  p\right)  }\left(  C_{n}\right)  ,$ because
\[
\lambda^{\left(  p\right)  }\left(  C_{n}\right)  \geq\lambda^{\left(
p\right)  }\left(  K_{2}\right)  =2\cdot2^{-2/p}>2n^{1-2/p}=P_{C_{n}}\left(
n^{-1/p}\mathbf{j}_{n}\right)  .
\]
Therefore, if $n$ is sufficiently large, any nonnegative eigenvector $\left(
x_{1},\ldots,x_{n}\right)  $ to $\lambda^{\left(  p\right)  }\left(
C_{n}\right)  $ has at least two distinct entries; hence, $\left(  x_{1}%
,x_{2},\ldots,x_{n}\right)  \neq\left(  x_{n},x_{1},\ldots,x_{n-1}\right)  ,$
and there are at least two positive eigenvectors to $\lambda^{\left(
p\right)  }\left(  C_{n}\right)  .$ These findings are summarized in the
following proposition.

\begin{proposition}
\label{proc2}For every $p\in\left(  1,2\right)  $ there exists an
$n_{0}\left(  p\right)  ,$ such that if $n>n_{0}\left(  p\right)  ,$ then
$\lambda^{\left(  p\right)  }\left(  C_{n}\right)  $ has at least two distinct
positive eigenvectors, different from $n^{-1/p}\mathbf{j}_{n}.$ In addition,
the value $\lambda=2n^{1-2/p}$ and the vector $\mathbf{x}=n^{-1/p}%
\mathbf{j}_{n}$ satisfy the eigenequations (\ref{eequ}) for $r=2$ and
$G=C_{n}$.
\end{proposition}

To extend this proposition to $r$-graphs, define the $r$-cycle $C_{n}^{r}$ of
order $n$ as\emph{: }$v\left(  C_{n}^{r}\right)  =\mathbb{Z}/n\mathbb{Z},$
\emph{the additive group of the integer remainders }$\operatorname{mod}$\emph{
}$n;$\emph{ the edges of }$C_{n}^{r}$\emph{ are all sets of the type}
$\left\{  i+1,\ldots,i+r\right\}  ,$ $i\in\mathbb{Z}/n\mathbb{Z}.$ In other
words, the vertices of $C_{n}^{r}$ can be arranged on a circle so that its
edges are all segments of $r$ consecutive vertices along the circle.

It is not hard to generalize the previous proposition as follows.

\begin{proposition}
\label{procr}For every $p\in\left(  1,r\right)  $ there exists an
$n_{0}\left(  p\right)  ,$ such that if $n>n_{0}\left(  p\right)  ,$ then
$\lambda^{\left(  p\right)  }\left(  C_{n}^{r}\right)  $ has at least two
distinct positive eigenvectors, different from $n^{-1/p}\mathbf{j}_{n}.$ In
addition, the value $\lambda=r!n^{1-r/p}$ and the vector $\mathbf{x}%
=n^{-1/p}\mathbf{j}_{n}$ satisfy the eigenequations (\ref{eequ}) for
$G=C_{n}^{r}$.
\end{proposition}

For $p=2$ and $r\geq3$ this example yields a negative answer to Question 4.9
of Pearson and Zhang \cite{PeZh12}.

\subsection{$\lambda^{\left(  p\right)  }$ and $\lambda_{\min}^{\left(
p\right)  }$ of $\beta$-stars}

Recall that a graph $G\in\mathcal{G}^{r}\left(  \left(  r-1\right)
k+1\right)  $ consisting of $k$ edges sharing a single vertex is called a
$\beta$-star. Proposition \ref{betas} gives $\lambda^{\left(  p\right)
}\left(  G\right)  $ and $\lambda_{\min}^{\left(  p\right)  }\left(  G\right)
$ whenever $p\geq r-1$. Using the method of Proposition \ref{bad3}, it is not
hard to obtain a more complete picture which sheds light on the possible
structure of eigenvectors of a simple $r$-graph.

\begin{proposition}
If $p>r-1,$ then%
\[
\lambda^{\left(  p\right)  }\left(  G\right)  =\left(  r!/r^{r/p}\right)
k^{1-\left(  r-1\right)  /p},\text{ \ \ \ \ \ }\lambda_{\min}^{\left(
p\right)  }\left(  G\right)  =-\left(  r!/r^{r/p}\right)  k^{1-\left(
r-1\right)  /p},
\]
and $\lambda^{\left(  p\right)  }\left(  G\right)  $ has a single eigenvector
$\mathbf{x}\in\mathbb{S}_{p,+}^{\left(  r-1\right)  k}.$

If $p<r-1,$ then \
\[
\lambda^{\left(  p\right)  }\left(  G\right)  =r!/r^{r/p},\text{
\ \ \ \ \ \ }\lambda_{\min}^{\left(  p\right)  }\left(  G\right)
=-r!/r^{r/p}.
\]
Each eigenvector $\mathbf{x}$ to $\lambda^{\left(  p\right)  }\left(
G\right)  $ or to $\lambda_{\min}^{\left(  p\right)  }\left(  G\right)  $ has
$r$ entries of modulus $r^{-1/p}$ belonging to a single edge and is zero elsewhere.

Finally, if $p=r-1,$ then $\lambda^{\left(  p\right)  }\left(  G\right)
=\left(  r-1\right)  !/r^{1\left(  r-1\right)  }$ and\ $\lambda^{\left(
p\right)  }\left(  G\right)  =-\left(  r-1\right)  !/r^{1\left(  r-1\right)
};$ each of the sets of eigenvectors to $\lambda^{\left(  p\right)  }\left(
G\right)  $ and to $\lambda_{\min}^{\left(  p\right)  }\left(  G\right)  $
contains a $\left(  k-1\right)  $-dimensional sphere.
\end{proposition}

\section{\label{PFsec}Elemental Perron-Frobenius theory for $r$-graphs}

The Perron-Frobenius theory of nonnegative matrices is\ extremely useful in
the study of the largest eigenvalue of $2$-graphs.\ Unfortunately, before the
work of Friedland, Gaubert and Han \cite{FGH11}, and of Cooper and Dutle
\cite{CoDu11}, the literature on nonnegative hypermatrices totally missed the
point for hypergraphs, as the adjacency hypermatrix of an $r$-graph is always
reducible if $r\geq3$. The papers \cite{FGH11} and \cite{CoDu11} put the study
of $\lambda\left(  G\right)  $ on a more solid ground, but none of these
papers gave a complete picture. The situation is additionally complicated with
the introduction of $\lambda^{\left(  p\right)  }\left(  G\right)  ,$ where
the dependence on $p$ has not been studied even for $r=2$. In this section we
make several steps in laying down a Perron-Frobenius type theory for
$\lambda^{\left(  p\right)  }\left(  G\right)  .$ The emerging complex picture
is essentially combinatorial; this is not surprising, as the Perron-Frobenius
theory for matrices builds on the combinatorial property \textquotedblleft
strong connectedness\textquotedblright\ of the matrix digraph.\medskip

Let us first state the starting point for $2$-graphs. Theorem \ref{PF} below
captures the three essential ingredients of what we refer to as the\emph{
Perron-Frobenius theory for }$2$\emph{-graphs}.

\begin{theorem}
\label{PF}Let $G$ be a connected $2$-graph with adjacency matrix $A$.

(a) If $A\mathbf{x}=\lambda\left(  G\right)  \mathbf{x}$ for some nonzero
$\mathbf{x},$ then $\mathbf{x}>0$ or $\mathbf{x}<0;$

(b) There exits a unique $\mathbf{x}>0$ such that $A\mathbf{x}=\lambda\left(
G\right)  \mathbf{x};$

(c) If $A\mathbf{y}=\mu\mathbf{y}$ for some number $\mu$ and vector
$\mathbf{y}\geq0$, then $\mu=\lambda\left(  G\right)  .$
\end{theorem}

We want to extend Theorem \ref{PF} to $\lambda^{\left(  p\right)  }\left(
G\right)  $ of an weighted $r$-graph $G$ for all $p>1,$ $r\geq2.$ Just a
cursory inspection of the examples in Section \ref{facts} shows that\ a
literal extension of Theorem \ref{PF} would fail in many points. First, Fact
\ref{mev} shows that even very simple connected graphs may have eigenvectors
to $\lambda\left(  G\right)  $ with entries of different sign. But as it turns
out the sign of the eigenvector entries is a nonissue for hypergraphs; we
postpone the discussion to Subsection \ref{ET}, and meanwhile focus only on
nonnegative eigenvectors. Each of the three clauses of Theorem \ref{PF} is
extended in a separate subsections below.

\subsection{Positivity of eigenvectors to $\lambda^{\left(  p\right)  }$}

Our first goal in this subsection is to extend clause (a) of Theorem \ref{PF}.
A serious obstruction to our plans comes from the example in Proposition
\ref{badl}, which shows that if $1<p\leq r-1,$ then this clause cannot be
literally extended to $\lambda^{\left(  p\right)  }$. We give a conditional
extension in Theorem \ref{PFa} below; however, we start with a simpler case,
which already contains the main idea. Note that for graphs our theorem is
stronger than Theorem 1.1 of \cite{FGH11}.

\begin{theorem}
\label{PFa0}Let $r\geq2,$ $p>r-1,$ $G\in\mathcal{W}^{r}\left(  n\right)  ,$
and $\left[  x_{i}\right]  \in\mathbb{S}_{p,+}^{n-1}.$ If $G$ is connected and
$\left[  x_{i}\right]  $ satisfies the equations
\begin{equation}
\lambda^{\left(  p\right)  }\left(  G\right)  x_{k}^{p-1}=\frac{1}{r}%
\frac{\partial P_{G}\left(  \left[  x_{i}\right]  \right)  }{\partial x_{k}%
},\ \ \ \ k=1,\ldots,n, \label{equa}%
\end{equation}
then $x_{1},\ldots,x_{n}$ are positive.
\end{theorem}

\begin{proof}
Our proof refines an idea of Cooper and Dutle \cite{CoDu11}, Lemma 3.3. Assume
that $p,$ $G,$ and $\left[  x_{i}\right]  $ are as required. Write $G_{0}$ for
the graph induced by the vertices with zero entries in $\left[  x_{i}\right]
,$ and assume for a contradiction that $G_{0}$ is nonempty. Since $G$ is
connected, there exists an edge $e$ such that
\[
U=V\left(  G_{0}\right)  \cap e\neq\varnothing\text{ \ and }W=e\backslash
V\left(  G_{0}\right)  \neq\varnothing.
\]
To finish the proof we shall construct a vector $\mathbf{y}\in\mathbb{S}%
_{p,+}^{n-1}$ such that $P_{G}\left(  \mathbf{y}\right)  >P_{G}\left(  \left[
x_{i}\right]  \right)  =\lambda^{\left(  p\right)  }\left(  G\right)  ,$ which
is the desired contradiction. Let $u\in W$ and for every sufficiently small
$\varepsilon>0,$ define a $\delta:=\delta\left(  \varepsilon\right)  $ by%
\[
\delta:=x_{u}-\sqrt[p]{x_{u}^{p}-\left\vert U\right\vert \varepsilon^{p}}.
\]
Clearly,%
\begin{equation}
\left\vert U\right\vert \varepsilon^{p}+\left(  x_{u}-\delta\right)
^{p}=x_{u}^{p}, \label{cond1}%
\end{equation}
and $\delta\left(  \varepsilon\right)  \rightarrow0$ as $\varepsilon
\rightarrow0.$ Since for each $v\in W,$ the entry $x_{v}$ is positive, we may
and shall assume that
\begin{equation}
\delta<\min_{v\in W}\left\{  x_{v}\right\}  /2\text{ \ \ \ and \ \ \ \ }%
\varepsilon<\min_{v\in W}\left\{  x_{j}\right\}  -\delta. \label{cond2}%
\end{equation}
Now, define the vector $\mathbf{y}=\left[  y_{i}\right]  $ by
\[
y_{i}:=\left\{
\begin{array}
[c]{ll}%
x_{i}+\varepsilon, & \text{if }i\in U;\\
x_{i}-\delta, & \text{if }i=u;\text{ }\\
x_{i}, & \text{if }i\notin U\cup\left\{  u\right\}  .
\end{array}
\right.
\]
First, (\ref{cond1}) and (\ref{cond2}) imply that $\left\vert \mathbf{y}%
\right\vert _{p}=\left\vert \mathbf{x}\right\vert _{p}=1$ and $\mathbf{y}%
\geq0;$ hence, $\mathbf{y}\in\mathbb{S}_{p,+}^{n-1}.$ Also, by Bernoulli's
inequality (\ref{Berin}), $x_{u}^{p}-\left(  x_{u}-\delta\right)  ^{p}%
>p\delta\left(  x_{u}-\delta\right)  ^{p-1}$ and so,
\[
r\varepsilon^{p}>\left\vert U\right\vert \varepsilon^{p}=x_{u}^{p}-\left(
x_{u}-\delta\right)  ^{p}>p\delta\left(  x_{u}-\delta\right)  ^{p-1}%
>p\delta\left(  \frac{x_{u}}{2}\right)  ^{p-1},
\]
implying that%
\[
\delta<r\frac{2^{p-1}}{x_{u}^{p-1}}\varepsilon^{p}.
\]
Further, set for short
\[
D:=\frac{\partial P_{G}\left(  \left[  x_{i}\right]  \right)  }{\partial
x_{u}}=r!\sum_{\left\{  u,i_{1},\ldots,i_{r-1}\right\}  \in E\left(  G\right)
}G\left(  \left\{  u,i_{1},\ldots,i_{r-1}\right\}  \right)  x_{i_{1}}\cdots
x_{i_{r-1}},
\]
and note that
\begin{align*}
P_{G}\left(  \mathbf{y}\right)  -P_{G}\left(  \mathbf{x}\right)   &  \geq
r!G\left(  e\right)
{\displaystyle\prod\limits_{i\in e}}
y_{i}-r!\delta\frac{\partial P_{G}\left(  \left[  x_{i}\right]  \right)
}{\partial x_{u}}\geq r!\left(  x_{u}-\delta\right)  \varepsilon^{r-1}-\delta
D\\
&  \geq r!G\left(  e\right)  \left(  \frac{x_{u}}{2}\right)  \varepsilon
^{r-1}-r\frac{2^{p-1}}{x_{u}^{p-1}}D\varepsilon^{p}\\
&  =\left(  r!G\left(  e\right)  \left(  \frac{x_{u}}{2}\right)  -\left(
r\frac{2^{p-1}}{x_{u}^{p-1}}D\right)  \varepsilon^{p-r+1}\right)
\varepsilon^{r-1}.
\end{align*}
In view of $p-r+1>0,$ if $\varepsilon$ is sufficiently small, then
$P_{G}\left(  \mathbf{y}\right)  -P_{G}\left(  \left[  x_{i}\right]  \right)
>0,$ contradicting that $P_{G}\left(  \mathbf{y}\right)  \leq P_{G}\left(
\left[  x_{i}\right]  \right)  $ and completing the proof.
\end{proof}

The examples in Proposition \ref{badl} show that the assertion of Theorem
\ref{PFa0} cannot be extended for $p\leq r-1.$ However we can force such
extensions by requiring stronger connectedness of $G,$ which we define next:
\medskip\ 

\emph{Let }$1\leq k\leq r-1$\emph{ and let }$G\in G^{r}.$\emph{ }$G$\emph{ is
called }$k$\emph{-tight,\ if }$E\left(  G\right)  \neq\varnothing$\emph{ and
for any proper set }$U\subset V\left(  G\right)  $\emph{ containing edges,
there is an edge }$e$\emph{ such that }$k\leq\left\vert e\cap U\right\vert
\leq r-1.\medskip$

Note that a graph is $1$-tight if and only if it is connected. Also if $G\ $is
$p$-tight, then it is $q$-tight for $1\leq q\leq p$. If $G\in\mathcal{G}%
^{r}\left(  2r-k\right)  $ consists of two edges with exactly $k$ vertices in
common, then $G$ is $k$-tight but not $\left(  k+1\right)  $-tight; hence one
can anticipate that the properties of the graphs in Proposition \ref{badl}
have something to do with their tightness; such connections do exist indeed.

\begin{theorem}
\label{PFa}Let $1\leq k\leq r-1,$ $p>r-k,$ $G\in\mathcal{W}^{r}\left(
n\right)  ,$ and $\left[  x_{i}\right]  \in\mathbb{S}_{p,+}^{n-1}.$ If $G$ is
$k$-tight and $\left[  x_{i}\right]  $ satisfies the equations
\[
\lambda^{\left(  p\right)  }\left(  G\right)  x_{k}^{p-1}=\frac{\partial
P_{G}\left(  \left[  x_{i}\right]  \right)  }{\partial x_{k}}%
,\ \ \ \ k=1,\ldots,n,
\]
then $x_{1},\ldots,x_{n}$ are positive.
\end{theorem}

\begin{proof}
Our proof is similar to the proof of Theorem \ref{PFa0}, so we omit some
details. Write $G_{0}$ for the graph induced by the vertices with zero entries
in $\left[  x_{i}\right]  ,$ and assume for a contradiction that $G_{0}$ is
nonempty. Note that $V\left(  G\right)  \backslash V\left(  G_{0}\right)  $
contains an edge, as $\lambda^{\left(  p\right)  }\left(  G\right)  >0;$ since
$G$ is $k$-tight, there is an edge $e$ one with $e\cap V\left(  G_{0}\right)
\neq\varnothing$ and $\left\vert e\backslash V\left(  G_{0}\right)
\right\vert \geq k$ let%
\[
U=V\left(  G_{0}\right)  \cap e\text{ \ and }W=e\backslash V\left(
G_{0}\right)  .
\]
As $\left\vert W\right\vert \geq k,$ we have%
\[
\left\vert U\right\vert =r-\left\vert W\right\vert \leq r-k.
\]
To finish the proof we shall construct a vector $\mathbf{y}\in\mathbb{S}%
_{p,+}^{n-1}$ such that $P_{G}\left(  \mathbf{y}\right)  >P_{G}\left(  \left[
x_{i}\right]  \right)  =\lambda^{\left(  p\right)  }\left(  G\right)  ,$ which
is the desired contradiction. Let $u\in W$ and for every sufficiently small
$\varepsilon>0,$ define a $\delta:=\delta\left(  \varepsilon\right)  $ by%
\[
\delta:=x_{u}-\sqrt[p]{x_{u}^{p}-\left\vert U\right\vert \varepsilon^{p}}.
\]
Clearly, $\left\vert U\right\vert \varepsilon^{p}+\left(  x_{u}-\delta\right)
^{p}=x_{u}^{p},$ and $\delta\left(  \varepsilon\right)  \rightarrow0$ as
$\varepsilon\rightarrow0.$ Since for each $v\in W,$ the entry $x_{v}$ is
positive, we may and shall assume that
\[
\delta<\min_{v\in W}\left\{  x_{v}\right\}  /2\text{ \ \ \ and \ \ \ \ }%
\varepsilon<\min_{v\in W}\left\{  x_{j}\right\}  -\delta.
\]
Now, define the vector $\mathbf{y}=\left[  y_{i}\right]  $ by
\[
y_{i}:=\left\{
\begin{array}
[c]{ll}%
x_{i}+\varepsilon, & \text{if }i\in U,\\
x_{i}-\delta, & \text{if }i=u,\text{ }\\
x_{i}, & \text{if }i\notin U\cup\left\{  u\right\}  ,
\end{array}
\right.
\]
and note that $\mathbf{y}\in\mathbb{S}_{p,+}^{n-1}.$ Also, as in Theorem
\ref{PFa0}, we find that%
\[
\delta<r\frac{2^{p-1}}{x_{u}^{p-1}}\varepsilon^{p}.
\]
Further, set for short
\begin{align*}
C  &  :=%
{\displaystyle\prod\limits_{i\in W\backslash\left\{  u\right\}  }}
x_{i},\\
D  &  :=\frac{\partial P_{G}\left(  \left[  x_{i}\right]  \right)  }{\partial
x_{u}}=r!\sum_{\left\{  u,i_{1},\ldots,i_{r-1}\right\}  \in E\left(  G\right)
}G\left(  \left\{  u,i_{1},\ldots,i_{r-1}\right\}  \right)  x_{i_{1}}\cdots
x_{i_{r-1}},
\end{align*}
and note that
\begin{align*}
P_{G}\left(  \mathbf{y}\right)  -P_{G}\left(  \mathbf{x}\right)   &  \geq
r!G\left(  e\right)
{\displaystyle\prod\limits_{i\in e_{j}}}
y_{i}-r!\delta\frac{\partial P_{G}\left(  \left[  x_{i}\right]  \right)
}{\partial x_{u}}=r!%
{\displaystyle\prod\limits_{i\in W}}
y_{i}%
{\displaystyle\prod\limits_{i\in U}}
y_{i}-\delta D\\
&  \geq r!G\left(  e\right)  \left(  x_{u}-\delta\right)  C\varepsilon
^{r-k}-\delta D\geq r!\left(  \frac{x_{u}}{2}\right)  C\varepsilon
^{r-k}-r\frac{2^{p-1}}{x_{u}^{p-1}}D\varepsilon^{p}\\
&  =\left(  r!G\left(  e\right)  \left(  \frac{x_{u}}{2}\right)  -\left(
r\frac{2^{p-1}}{x_{u}^{p-1}}D\right)  \varepsilon^{p-r+k}\right)
\varepsilon^{r-k}.
\end{align*}
In view of $p-r+k>0,$ if $\varepsilon$ is sufficiently small, then
$P_{G}\left(  \mathbf{y}\right)  -P_{G}\left(  \left[  x_{i}\right]  \right)
>0,$ contradicting that $P_{G}\left(  \mathbf{y}\right)  \leq P_{G}\left(
\left[  x_{i}\right]  \right)  $ and completing the proof.
\end{proof}

Armed with Theorem \ref{PFa} we can find how $\lambda^{\left(  p\right)
}\left(  G\right)  $ changes when taking subgraphs.

\begin{corollary}
\label{corPFa}Let $r\geq2,$ $r-1\geq k\geq1,$ $p>r-k,$ and $G\in
\mathcal{W}^{r}.$ If $G$ is $k$-tight and $H$ is a subgraph of $G,$ then
\[
\lambda^{\left(  p\right)  }\left(  H\right)  <\lambda^{\left(  p\right)
}\left(  G\right)  ,
\]
unless $H=G.$ In particular, if $p>r-1$ and $G$ is connected, then
$\lambda^{\left(  p\right)  }\left(  H\right)  <\lambda^{\left(  p\right)
}\left(  G\right)  $ for every proper subgraph $H$ of $G.$
\end{corollary}

\medskip

The examples of Proposition \ref{badl} show that Theorem \ref{PFa} is as good
as one can get, but they do not shed enough light on the case $p>r,$ which is
somewhat surprising, as the following theorem shows.

\begin{theorem}
\label{PFar}Let $p>r\geq2,$ $G\in\mathcal{W}^{r}\left(  n\right)  ,$ and
$\left[  x_{i}\right]  \in\mathbb{S}_{p,+}^{n-1}.$ If $G$ is nonzero and
$\left[  x_{i}\right]  $ is an eigenvector to $\lambda^{\left(  p\right)
}\left(  G\right)  ,$ then $x_{u}>0$ for each non-isolated vertex $u$.
\end{theorem}

We omit the proof which is almost the same as the proof of Theorem \ref{PFa0}.
Instead, let us make the following observation.

\begin{corollary}
\label{corPFar}Let $p>r\geq2$ and let $G\in\mathcal{W}^{r}$ and $H\in
\mathcal{W}^{r}$. If $H$ is subgraph of $G,$ then
\[
\lambda^{\left(  p\right)  }\left(  H\right)  <\lambda^{\left(  p\right)
}\left(  G\right)  ,
\]
unless $G$ has no edges or $H=G$.
\end{corollary}

In the light of the examples in Proposition \ref{badl} the notion of
$k$-tightness gives a pretty strong sufficient condition for the eigenvectors
of $\lambda^{\left(  p\right)  }\left(  G\right)  $ to have only nonzero
entries. Also, for $2$-graphs, one can easily see the following characterization.

\begin{proposition}
Let $G\in$ $\mathcal{W}_{2}$ and $1<p\leq2.$ There is an eigenvector to
$\lambda^{\left(  p\right)  }\left(  G\right)  $ with nonzero entries if and
only if $\lambda^{\left(  p\right)  }\left(  G\right)  =\lambda^{\left(
p\right)  }\left(  G^{\prime}\right)  $ for every component $G^{\prime}$ of
$G$.
\end{proposition}

However for $r\geq3,$ the corresponding problems are far from resolved:

\begin{problem}
Let $r\geq3$ and $1<p\leq r.$ Characterize all $G\in\mathcal{W}^{r}\left(
n\right)  ,$ such that all eigenvectors to $\lambda^{\left(  p\right)
}\left(  G\right)  $ have only nonzero entries.
\end{problem}

\begin{problem}
Let $r\geq3$ and $1<p\leq r.$ Characterize all $G\in\mathcal{W}^{r}\left(
n\right)  ,$ such that there is an eigenvector to $\lambda^{\left(  p\right)
}\left(  G\right)  $ with all entries nonzero.
\end{problem}

\medskip

\subsection{Uniqueness of the positive eigenvector to $\lambda^{\left(
p\right)  }$}

In this subsection we generalize clause (b) of Theorem \ref{PF}. The main
obstruction in this task is exemplified by the $r$-cycle $C_{n}^{r}:$ as
Proposition \ref{procr} shows if $1<p<r$, then $\lambda^{\left(  p\right)
}\left(  C_{n}^{r}\right)  $ always has at least two positive eigenvectors.
Finding precisely for which graphs $G\in\mathcal{G}^{r}$ there is a unique
positive eigenvector to $\lambda^{\left(  p\right)  }\left(  G\right)  $ is
currently an open problem. Note that tightness is not relevant in this
characterization, as the cycles $C_{n}^{r}$ are $\left(  r-1\right)  $-tight.
We give a limited solution below, leaving the general problem for future study.

Here is the proposed generalization of clause (b).

\begin{theorem}
\label{PFb}If $p\geq r\geq2$ and $G\in\mathcal{W}^{r}\left(  n\right)  .$ If
$G$ is connected, there exists a unique $\left[  x_{i}\right]  \in
\mathbb{S}_{p,+}^{n-1}$ satisfying the equations
\[
\lambda^{\left(  p\right)  }\left(  G\right)  x_{k}^{p-1}=\frac{1}{r}%
\frac{\partial P_{G}\left(  \left[  x_{i}\right]  \right)  }{\partial x_{k}%
},\ \ \ \ k=1,\ldots,n.
\]

\end{theorem}

For the proof of the theorem we shall need the following proposition.

\begin{proposition}
\label{Einth} Let $p\geq1,$ $G\in\mathcal{W}^{r}\left(  n\right)  ,$ and
$\left[  x_{i}\right]  \in\mathbb{S}_{p,+}^{n-1}.$ If $\left[  x_{i}\right]  $
satisfies the inequalities%
\begin{equation}
\lambda^{\left(  p\right)  }\left(  G\right)  x_{k}^{p-1}\leq\frac{1}{r}%
\frac{\partial P_{G}\left(  \left[  x_{i}\right]  \right)  }{\partial x_{k}%
},\ \ \ \ k=1,\ldots,n, \label{eigin}%
\end{equation}
then $\left[  x_{i}\right]  $ is an eigenvector to $\lambda^{\left(  p\right)
}\left(  G\right)  $ and equality holds in (\ref{eigin}) for each
$k=1,\ldots,n.$
\end{proposition}

\begin{proof}
Multiplying both sides of (\ref{eigin}) by $x_{k}$ and adding all
inequalities, we obtain%
\[
\lambda^{\left(  p\right)  }\left(  G\right)  \sum_{k=1}^{n}x_{k}^{p}\leq
\frac{1}{r}\sum_{k=1}^{n}x_{k}\frac{\partial P_{G}\left(  \mathbf{x}\right)
}{\partial x_{k}}=P_{G}\left(  \mathbf{x}\right)  \leq\lambda^{\left(
p\right)  }\left(  G\right)  \sum_{k=1}^{n}x_{k}^{p}.
\]
Therefore, equalities hold in (\ref{eigin}) and $\mathbf{x}$ is an eigenvector
to $\lambda\left(  G\right)  .$
\end{proof}

\begin{proof}
[\textbf{Proof of Theorem \ref{PFb}}]Let $\mathbf{x}=\left[  x_{i}\right]
\in\mathbb{S}_{p,+}^{n-1}$ and $\mathbf{y}=\left[  y_{i}\right]  \in
\mathbb{S}_{p,+}^{n-1}$ be two positive eigenvectors to $\lambda^{\left(
p\right)  }\left(  G\right)  ;$ we have to prove that $\mathbf{x}=\mathbf{y}$.
Define a vector $\mathbf{z}=\left[  z_{i}\right]  \in\mathbb{S}_{p,+}^{n-1}$
by the equations%
\[
z_{k}:=\sqrt[p]{\frac{x_{k}^{p}+y_{k}^{p}}{2}},\text{ \ \ \ }k=1,\ldots,n.
\]
Now, for each $k=1,\ldots,n,$ add the two equations%
\[
\lambda^{\left(  p\right)  }\left(  G\right)  x_{k}^{p}=\left(  r-1\right)
!\sum_{\left\{  k,i_{1}\cdots i_{r-1}\right\}  \in E\left(  G\right)  }%
x_{k}x_{i_{1}}\cdots x_{i_{r-1}}%
\]
and
\[
\lambda^{\left(  p\right)  }\left(  G\right)  y_{k}^{p}=\left(  r-1\right)
!\sum_{\left\{  k,i_{1}\cdots i_{r-1}\right\}  \in E\left(  G\right)  }%
y_{k}y_{i_{1}}\cdots y_{i_{r-1}},
\]
getting
\begin{align*}
\lambda^{\left(  p\right)  }\left(  G\right)  z_{k}^{p}  &  =\frac
{\lambda^{\left(  p\right)  }\left(  G\right)  x_{k}^{p}+\lambda^{\left(
p\right)  }\left(  G\right)  y_{k}^{p}}{2}\\
&  =\left(  r-1\right)  !\sum_{\left\{  k,i_{1}\cdots i_{r-1}\right\}  \in
E\left(  G\right)  }G\left\{  k,i_{1},\ldots,i_{r-1}\right\}  \frac
{x_{k}x_{i_{1}}\cdots x_{i_{r-1}}+y_{k}y_{i_{1}}\cdots y_{i_{r-1}}}{2}.
\end{align*}
Applying the generalized Cauchy-Schwarz inequality (\ref{CSgen}) to the
vectors $\left(  x_{k},y_{k}\right)  $ and $\left(  x_{i_{s}},y_{i_{s}%
}\right)  ,$ $1\leq s\leq r-1,$ and the PM inequality implies that
\begin{align*}
\frac{x_{k}x_{i_{1}}\cdots x_{i_{r-1}}+y_{k}y_{i_{1}}\cdots y_{i_{r-1}}}{2}
&  \leq\sqrt[r]{\frac{x_{k}^{r}+y_{k}^{r}}{2}}\prod_{s=1}^{r-1}\sqrt[r]%
{\frac{x_{i_{s}}^{r}+y_{i_{s}}^{r}}{2}}\leq\sqrt[p]{\frac{x_{k}^{p}+y_{k}^{p}%
}{2}}\prod_{s=1}^{r}\sqrt[p]{\frac{x_{i_{s}}^{p}+y_{i_{s}}^{p}}{2}}\\
&  =\sum_{\left\{  k,i_{1}\cdots i_{r-1}\right\}  \in E\left(  G\right)
}z_{k}z_{i_{1}}\cdots z_{i_{r-1}}.
\end{align*}
Therefore,
\begin{equation}
\lambda\left(  G\right)  z_{k}^{p-1}\leq\sum_{\left\{  k,i_{1}\cdots
i_{r-1}\right\}  \in E\left(  G\right)  }G\left\{  k,i_{1},\ldots
,i_{r-1}\right\}  z_{i_{1}}\cdots z_{i_{r-1}}=\frac{1}{r}\frac{\partial
P_{G}\left(  \mathbf{z}\right)  }{\partial z_{k}}\text{ \ \ \ \ }k=1,\ldots,n.
\label{in1}%
\end{equation}
and Lemma \ref{Einth} implies that equalities hold in (\ref{in1})$.$ By the
condition for equality in (\ref{CSgen}), if the vertices $i$ and $j$ are
contained in the same edge of $G,$ then there is a $c,$ such that $\left(
x_{i},y_{i}\right)  =c\left(  x_{j},y_{j}\right)  .$ Since $G$ is connected,
this assertion can be put simply as: for every vertex $i$ of $G,$ there is a
$c,$ such that $\left(  x_{i},y_{i}\right)  =c\left(  x_{1},y_{1}\right)  .$
Finally, this equation implies that $y_{i}=\left(  y_{1}/x_{1}\right)  x_{i},$
for $i=1,\ldots,n,$ and so $\mathbf{y}$ is collinear to $\mathbf{x},$ and so
$\mathbf{x}=\mathbf{y}.$
\end{proof}

Having had the experience with Theorem \ref{PFar}, we easily come up with the
following theorem.

\begin{theorem}
\label{PFbr}If $p>r\geq2$ and $G\in\mathcal{G}^{r}\left(  n\right)  $, there
is a unique $\left[  x_{i}\right]  \in\mathbb{S}_{p,+}^{n-1}$ satisfying the
equations
\[
\lambda^{\left(  p\right)  }\left(  G\right)  x_{k}^{p-1}=\frac{1}{r}%
\frac{\partial P_{G}\left(  \left[  x_{i}\right]  \right)  }{\partial x_{k}%
},\ \ \ \ k=1,\ldots,n.
\]

\end{theorem}

We omit this proof, which follows from Theorem \ref{PFar} and \ref{PFb}%
.\medskip

\subsection{Uniqueness of $\lambda^{\left(  p\right)  }$}

In this subsection we shall extend clause (c) of Theorem \ref{PF}. Note first
that a literal extension is impossible in view of Proposition \ref{zeroev}: if
$r\geq3,$ for every $r$-graph, the value $\lambda=0$ satisfies (\ref{equa})
with many nonnegative eigenvectors. Since this situation is unavoidable,it
seems reasonable to consider only positive $\lambda$ and nonnegative $\left[
x_{i}\right]  \in\mathbb{S}_{p,+}^{n-1}$ satisfying (\ref{equa}). But even
with these restriction, Propositions \ref{badl}, \ref{secl} and \ref{procr}
show that there are $r$-graphs for which clause (c) cannot possibly hold if
$1<p<r.$ In fact, the following problems is open.

\begin{problem}
Given $1<p<r,$ characterize all graphs $G\in\mathcal{G}^{r}\left(  n\right)  $
for which there is a unique $\lambda=\lambda^{\left(  p\right)  }\left(
G\right)  $ is the only positive number satisfying the equations
\[
\lambda x_{k}^{p-1}=\frac{1}{r}\frac{\partial P_{G}\left(  \left[
x_{i}\right]  \right)  }{\partial x_{k}},\ \ \ \ k=1,\ldots,n,
\]
for some $\left[  x_{i}\right]  \in\mathbb{S}_{p,+}^{n-1}.$
\end{problem}

However, clause (c) is an important practical issue, so we shall establish
necessary and sufficient conditions for its validity for $p\geq r$.

\begin{theorem}
\label{PFc}Let $p\geq r\geq2,$ $G\in\mathcal{G}^{r}\left(  n\right)  $ and
$\left[  x_{i}\right]  \in\mathbb{S}_{p,+}^{n-1}.$ If $G$ is $\left(
r-1\right)  $-tight and $\left[  x_{i}\right]  $ satisfies the equations
\[
\lambda x_{k}^{p-1}=\frac{1}{r}\frac{\partial P_{G}\left(  \left[
x_{i}\right]  \right)  }{\partial x_{k}},\ \ \ \ k=1,\ldots,n.
\]
for some $\lambda>0,$ then $x_{1},\ldots,x_{n}$ are positive and
$\lambda=\lambda^{\left(  p\right)  }\left(  G\right)  $.
\end{theorem}

Before starting the proof, let us note the following proposition, which is
useful in its own right.

\begin{proposition}
\label{pos}Let $r\geq2,$ $p>1,$ $G\in\mathcal{G}^{r}\left(  n\right)  ,$ and
$\left[  x_{i}\right]  \in\mathbb{S}_{p,+}^{n-1}.$ If $G$ is $\left(
r-1\right)  $-tight and $\left[  x_{i}\right]  $ satisfies the equations
\[
\lambda x_{k}^{p-1}=\frac{1}{r}\frac{\partial P_{G}\left(  \left[
x_{i}\right]  \right)  }{\partial x_{k}},\ \ \ \ k=1,\ldots,n,
\]
for some $\lambda>0,$ then $x_{1},\ldots,x_{n}$ are positive.
\end{proposition}

The proof follows from the observation that if $\left\{  i_{1},\ldots
,i_{r}\right\}  \in E\left(  G\right)  $ satisfies $x_{i_{1}}\cdots
x_{i_{r-1}}>0,$ then $x_{i_{r}}>0$.

\begin{proof}
[\textbf{Proof of Theorem \ref{PFc}}]We adapt an idea from \cite{CPZ08}. Let
$\left[  y_{i}\right]  \in\mathbb{S}_{p,+}^{n-1}$ be an eigenvector to
$\lambda^{\left(  p\right)  }\left(  G\right)  $. Proposition \ref{pos}
implies that $\left[  x_{i}\right]  >0$ and $\left[  y_{i}\right]  >0$. Let
\[
\sigma=\min\left\{  x_{1}/y_{1},\ldots,x_{n}/y_{n}\right\}  =x_{k}/y_{k}.
\]
Clearly $\sigma>0;$ also $\sigma\leq1,$ for otherwise $\left\vert \left[
x_{i}\right]  \right\vert _{p}>\left\vert \left[  y_{i}\right]  \right\vert
_{p},$ a contradiction. Further,
\[
\lambda x_{k}^{p-1}=\frac{1}{r}\frac{\partial P_{G}\left(  \left[
x_{i}\right]  \right)  }{\partial x_{k}}\geq\frac{1}{r}\sigma^{r-1}%
\frac{\partial P_{G}\left(  \left[  y_{i}\right]  \right)  }{\partial y_{k}%
}=\sigma^{r-1}\lambda^{\left(  p\right)  }\left(  G\right)  y_{k}^{p-1}%
=\sigma^{r-p}\lambda^{\left(  p\right)  }\left(  G\right)  x_{k}^{p-1}.
\]
implying that $\lambda^{\left(  p\right)  }\left(  G\right)  \leq\lambda.$ But
$\lambda=P_{G}\left(  \left[  x_{i}\right]  \right)  \leq\lambda^{\left(
p\right)  }\left(  G\right)  ,$ and so $\lambda=\lambda^{\left(  p\right)
}\left(  G\right)  ,$ completing the proof.
\end{proof}

One can think that the requirement $G$ to be $\left(  r-1\right)  $-tight in
Theorem \ref{PFc} is too strong. However, it is best possible, as the
following theorem and its corollary suggest.

\begin{theorem}
Let $r\geq2,$ $p>1,$ and $G\in\mathcal{G}^{r}\left(  n\right)  .$ If $G\ $is
not $\left(  r-1\right)  $-tight, there are $\lambda>0$ and $\left[
x_{i}\right]  \in\mathbb{S}_{p,+}^{n-1}$ such that
\[
\lambda x_{k}^{p-1}=\frac{1}{r}\frac{\partial P_{G}\left(  \left[
x_{i}\right]  \right)  }{\partial x_{k}},\ \ \ \ k=1,\ldots,n,
\]
but $\left[  x_{i}\right]  $ is not positive.
\end{theorem}

\begin{proof}
[\textbf{Sketch of a proof}]Let $e_{1}\in E\left(  G\right)  $ and let $U$ be
a set of vertices containing an edge, but no edge $e\in E\left(  G\right)  $
satisfies $\left\vert e\cap U\right\vert =r-1.$ Let $G_{1}=G\left[  U\right]
$ and set $\lambda=\lambda^{\left(  p\right)  }\left(  G_{1}\right)  ;$
clearly $\lambda>0.$ To get $\left[  x_{i}\right]  \in\mathbb{S}_{p,+}^{n-1},$
take a nonnegative eigenvector to $\lambda^{\left(  p\right)  }\left(
G_{1}\right)  $ and set $x_{v}=0$ for all $x_{v}\in V\left(  G\right)
\backslash V_{1}.$
\end{proof}

\medskip

If $r\geq3,$ we can see the decisive role of $\left(  r-1\right)  $-tightness
as laid down in the following statement.

\begin{corollary}
Let $p\geq r\geq3$ and $G\in\mathcal{G}^{r}\left(  n\right)  .$ If $G$ is
connected, but not $\left(  r-1\right)  $-tight, there exists $\left[
x_{i}\right]  \in\mathbb{S}_{p,+}^{n-1}$ satisfying the equations
\[
\lambda x_{k}^{p-1}=\frac{1}{r}\frac{\partial P_{G}\left(  \left[
x_{i}\right]  \right)  }{\partial x_{k}},\ \ \ \ k=1,\ldots,n,
\]
for some positive $\lambda<\lambda^{\left(  p\right)  }\left(  G\right)  $.
\end{corollary}

\bigskip

\subsection{The Collatz-Wielandt function of $\lambda^{\left(  p\right)  }$}

In this subsection we shall deduce Collatz-Wielandt type characterizations of
$\lambda^{\left(  p\right)  }\left(  G\right)  .$ The following theorem is
analogous to the Collatz-Wielandt minimax theorem.

\begin{theorem}
\label{th_CW}If $p>1$ and $G\in\mathcal{G}^{r}\left(  n\right)  ,$ then%
\[
\lambda^{\left(  p\right)  }\left(  G\right)  =\max_{\left[  x_{i}\right]
\in\mathbb{S}_{p,+}^{n-1}}\min_{x_{k}>0}\frac{\partial P_{G}\left(  \left[
x_{i}\right]  \right)  }{\partial x_{k}}x_{k}^{-p+1}.
\]

\end{theorem}

\begin{proof}
Since $\mathbb{S}_{p,+}^{n-1}$ is compact, there exists $\left[  y_{i}\right]
\in\mathbb{S}_{p,+}^{n-1}$ such that
\[
\lambda=\min_{y_{k}>0}\frac{1}{r}\frac{\partial P_{G}\left(  \left[
y_{i}\right]  \right)  }{\partial y_{k}}y_{k}^{-p+1}=\max_{\left[
x_{i}\right]  \in\mathbb{S}_{p,+}^{n-1}}\min_{x_{k}>0}\frac{\partial
P_{G}\left(  \left[  x_{i}\right]  \right)  }{\partial x_{k}}x_{k}^{-p+1}.
\]
This equation clearly implies that
\begin{equation}
\lambda y_{k}^{p-1}\leq\frac{1}{r}\frac{\partial P_{G}\left(  \left[
y_{i}\right]  \right)  }{\partial y_{k}},\text{ \ \ \ \ }k=1,\ldots,n.
\label{y}%
\end{equation}
Let $\left[  z_{i}\right]  \in\mathbb{S}_{p,+}^{n-1}$ be an eigenvector to
$\lambda^{\left(  p\right)  }\left(  G\right)  .$ Clearly, if $z_{k}>0,$ then
\[
\lambda^{\left(  p\right)  }\left(  G\right)  =\frac{\partial P_{G}\left(
\left[  z_{i}\right]  \right)  }{\partial z_{k}}z_{k}^{-p+1}=\min_{z_{j}%
>0}\frac{\partial P_{G}\left(  \left[  z_{i}\right]  \right)  }{\partial
z_{j}}z_{j}^{-p+1}%
\]
and therefore $\lambda\geq\lambda^{\left(  p\right)  }\left(  G\right)  .$
Substituting $\lambda^{\left(  p\right)  }\left(  G\right)  $ for $\lambda$ in
(\ref{y}), we get
\[
\lambda^{\left(  p\right)  }\left(  G\right)  y_{k}^{p-1}\leq\frac{1}{r}%
\frac{\partial P_{G}\left(  \left[  y_{i}\right]  \right)  }{\partial y_{k}%
},\text{ \ \ \ \ }k=1,\ldots,n.
\]
Now, Proposition \ref{Einth} implies that equality holds for each $k\in\left[
n\right]  $ and so $\lambda^{\left(  p\right)  }\left(  G\right)  =\lambda.$
\end{proof}

Note that no special requirements about $G$ are needed in Theorem \ref{th_CW}.
Usually such theorems require that $G$ is connected, but there is no
justification for such weakening of the statement.

Next we use the proof of Theorem \ref{PFc} to get a more flexible theorem.

\begin{theorem}
\label{th_in}Let $p\geq r\geq2,$ $G\in\mathcal{W}^{r}\left(  n\right)  ,$ and
$\left[  x_{i}\right]  \in\mathbb{S}_{p,+}^{n-1}$. If $\left[  x_{i}\right]
>0$ and $\left[  x_{i}\right]  $ satisfies the inequalities%
\begin{equation}
\lambda x_{k}^{p-1}\geq\frac{1}{r}\frac{\partial P_{G}\left(  \left[
x_{i}\right]  \right)  }{\partial x_{k}},\text{ \ \ \ }k=1,\ldots,n,
\label{ineqa}%
\end{equation}
for some real $\lambda,$ then $\lambda\geq\lambda^{\left(  p\right)  }\left(
G\right)  .$ If $\lambda^{\left(  p\right)  }\left(  A\right)  =\lambda,$ then
equality holds in (\ref{ineqa}) for all $k\in\left[  n\right]  $ unless $p=r$
and $G$ is disconnected.
\end{theorem}

\begin{proof}
Let $\left[  y_{i}\right]  \in\mathbb{S}_{p,+}^{n-1}$ be an eigenvector to
$\lambda^{\left(  p\right)  }\left(  G\right)  $. Let $\sigma:=\min\left\{
x_{i}/y_{i}:y_{i}>0\right\}  .$ Clearly $\sigma>0;$ also $\sigma\leq1,$ for
otherwise $\left\vert \left[  x_{i}\right]  \right\vert _{p}>\left\vert
\left[  y_{i}\right]  \right\vert _{p},$ a contradiction. Note that $x_{i}%
\geq\sigma y_{i}$ for every $i\in\left[  n\right]  .$ Since $x_{k}=\sigma
y_{k}$ for some $k\in\left[  n\right]  ,$ we see that
\[
\lambda x_{k}^{p-1}\geq\frac{1}{r}\frac{\partial P_{G}\left(  \left[
x_{i}\right]  \right)  }{\partial x_{k}}\geq\frac{1}{r}\sigma^{r-1}%
\frac{\partial P_{G}\left(  \left[  y_{i}\right]  \right)  }{\partial y_{k}%
}=\sigma^{r-1}\lambda^{\left(  p\right)  }\left(  G\right)  y_{k}^{p-1}%
=\sigma^{r-p}\lambda^{\left(  p\right)  }\left(  G\right)  x_{k}^{p-1}.
\]
implying that $\lambda^{\left(  p\right)  }\left(  G\right)  \leq\lambda.$ If
$\lambda^{\left(  p\right)  }\left(  G\right)  =\lambda,$ then $\sigma=1$ or
$r=p.$ If $\sigma=1,$ then $\left[  x_{i}\right]  =\left[  y_{i}\right]  $ and
so equalities hold in (\ref{ineqa}) for all $k\in\left[  n\right]  ,$ so
assume that $\sigma<1$ and $r=p.$ We see that $x_{j}=\sigma y_{j}$ for every
vertex $j$ which is contained in an edge together with $k.$ Therefore
$x_{i}=\sigma y_{i}$ for the component of $G$ containing $x_{k}$ and so $G$ is disconnected.
\end{proof}

It is not hard to see that the disconnected $r$-graph $G=K_{r}^{r}\cup
K_{r+1}^{r}$ satisfies inequalities (\ref{ineqa}) with $\lambda=\lambda
^{\left(  r\right)  }\left(  G\right)  $ and some vector $\left[
x_{i}\right]  \in\mathbb{S}_{r,+}^{n-1}$ such that not all inequalities
(\ref{ineqa}) are equalities. Another point to make here is that for $p<r$ the
assertion may not be true: indeed the cycle $C_{n}^{r}$ has a unique positive
eigenvector to vector $\lambda^{\left(  p\right)  }\left(  G\right)  $ for
every $p>1,$ which is different from $n^{-1/p}\mathbf{j}_{n}.$ The vector
$\left[  x_{i}\right]  =n^{-1/p}\mathbf{j}_{n}$ together with $\lambda
=\left\vert r!C_{n}^{r}\right\vert /n^{r/p}=$ $r!n^{1-rp}$ satisfies the
inequalities (\ref{ineqa}), but $\lambda^{\left(  p\right)  }\left(  C_{n}%
^{r}\right)  >$ $r!n^{1-rp}$.

The above theorem helps to prove another theorem, related to the
Collatz-Wielandt function.

\begin{theorem}
\label{tCW}Let $p\geq r\geq2$ and $G\in\mathcal{G}^{r}\left(  n\right)  .$ If
$G$ is connected, then%
\[
\lambda^{\left(  p\right)  }\left(  G\right)  =\inf_{\left[  x_{i}\right]
>0,\text{ }\left\vert \left[  x_{i}\right]  \right\vert _{p}=1}\max
_{j\in\left[  n\right]  }\frac{\partial P_{G}\left(  \left[  x_{i}\right]
\right)  }{\partial x_{j}}x_{j}^{-p+1}.
\]

\end{theorem}

\begin{proof}
Let $\left[  x_{i}\right]  \in\mathbb{S}_{p,+}^{n-1}$ be an eigenvector to
$\lambda^{\left(  p\right)  }\left(  G\right)  ;$ by Theorem \ref{PFa0},
$\left[  x_{i}\right]  >0.$ Further,
\[
\lambda^{\left(  p\right)  }\left(  G\right)  =\frac{\partial P_{G}\left(
\left[  x_{i}\right]  \right)  }{\partial x_{k}}x_{k}^{-p+1},\text{
\ \ \ \ }k=1,\ldots,n,
\]
and so
\[
\inf_{\left[  x_{i}\right]  \in\mathbb{S}_{p,++}^{n-1}}\max_{k\in\left[
n\right]  }\frac{\partial P_{G}\left(  \left[  x_{i}\right]  \right)
}{\partial x_{k}}x_{k}^{-p+1}\leq\lambda^{\left(  p\right)  }\left(  G\right)
.
\]
Let $\left[  y_{i}\right]  >0$ and $\left\vert \left[  y_{i}\right]
\right\vert _{p}=1.$ Clearly, the value
\[
\lambda=\max_{k}\frac{\partial P_{G}\left(  \left[  y_{i}\right]  \right)
}{\partial y_{k}}y_{k}^{-p+1}%
\]
satisfies
\[
\lambda y_{k}^{p-1}\geq\frac{\partial P_{G}\left(  \left[  y_{i}\right]
\right)  }{\partial y_{k}},\text{\ \ \ \ }k=1,\ldots,n,
\]
and by Theorem \ref{th_in}, $\lambda\geq\lambda^{\left(  p\right)  }\left(
G\right)  ;$ hence,
\[
\inf_{\left[  x_{i}\right]  >0,\text{ }\left\vert \left[  x_{i}\right]
\right\vert _{p}=1}\max_{j\in\left[  n\right]  }\frac{\partial P_{G}\left(
\left[  x_{i}\right]  \right)  }{\partial x_{j}}x_{j}^{-p+1}\geq
\lambda^{\left(  p\right)  }\left(  G\right)  ,
\]
proving the theorem.
\end{proof}

As above, the cycle $C_{n}^{r}$ shows that Theorem \ref{tCW} cannot be
extended for $p<r$ even for $\left(  r-1\right)  $-tight graphs.\medskip

\subsection{A recap for $2$-graphs}

Let us note that the concept of $k$-tightness is essentially irrelevant for
$2$-graphs because $k$ can only take the value $1,$ and $1$-tight graph is the
same as connected. Thus, for reader's sake, we shall summarize the above
results for $2$-graphs in a single theorem parallel to Theorem \ref{PF}%
.\medskip

\begin{theorem}
\label{PFp2}Let $G$ be a connected $2$-graph.

(a) If $p>1$ and $\left[  x_{i}\right]  \in\mathbb{S}_{p,+}^{n-1}$ satisfies
the equations
\begin{equation}
\mu x_{k}^{p-1}=\sum_{\left\{  k,i\right\}  \in E\left(  G\right)  }%
x_{i},\ \ \ \ k=1,\ldots,n, \label{eq2}%
\end{equation}
for $\mu=\lambda^{\left(  p\right)  }\left(  G\right)  ,$ then $x_{1}%
,\ldots,x_{n}$ are positive;

(b) If $p\geq2,$ there is a unique positive eigenvector $\mathbf{x}$ to
$\lambda^{\left(  p\right)  }\left(  G\right)  ;$

(c) If $p\geq2$ and a vector $\left[  x_{i}\right]  \in\mathbb{S}_{p,+}^{n-1}%
$\textbf{ }satisfies equations (\ref{eq2}) for some $\mu$, then $\mu
=\lambda^{\left(  p\right)  }\left(  G\right)  .$
\end{theorem}

\medskip

\subsection{\label{ET}Eigenvectors to $\lambda^{\left(  p\right)  }$ and even
transversals}

In this subsection we address the situation outlined in Fact \ref{mev}:
eigenvectors to $\lambda^{\left(  p\right)  }\left(  G\right)  $ of a
connected $r$-graph $G$ may have entries of different sign. It turns out that
this property is related to the existence of even transversals in $G$. Recall
that an even transversal in a graph $G$ is a nonempty set of vertices
intersecting each edge in an even number of vertices. A \emph{proper} even
transversal is a proper subset of $V\left(  G\right)  $.\medskip

\begin{theorem}
\label{thET}Let $r\geq2,$ $p>r-1,$ and $G\in\mathcal{G}^{r}\left(  n\right)
.$ If $G$ is connected, then $\lambda^{\left(  p\right)  }\left(  G\right)  $
has an eigenvector $\left[  x_{i}\right]  \in\mathbb{S}_{p}^{n-1}$ with
entries of different signs if and only if $G$ has a proper even transversal.
\end{theorem}

\begin{proof}
Let $\left[  x_{i}\right]  \in\mathbb{S}_{p}^{n-1}$ be an eigenvector to
$\lambda^{\left(  p\right)  }\left(  G\right)  $ with entries of different
sign. Note that $\lambda^{\left(  p\right)  }\left(  G\right)  =P_{G}\left(
\left[  x_{i}\right]  \right)  =P_{G}\left(  \left[  \left\vert x_{i}%
\right\vert \right]  \right)  =\lambda^{\left(  p\right)  }\left(  G\right)
.$ Hence $\left[  \left\vert x_{i}\right\vert \right]  $ is an eigenvector to
$\lambda^{\left(  p\right)  }\left(  G\right)  ,$ and by Theorem \ref{PFa}
$\left[  \left\vert x_{i}\right\vert \right]  >0$. Let $V^{-}$ be the set of
vertices with negative entries in $\left[  x_{i}\right]  $. We shall prove
that $V^{-}$ is a proper even transversal of $G.$ By the assumption $V^{-}%
\neq\varnothing$ and $V\left(  G\right)  \backslash V^{-}\neq\varnothing$.
Also every edge$\left\{  i_{1},\ldots,i_{r}\right\}  \in E\left(  G\right)  $
intersects $V^{-}$ in an even number of vertices, otherwise the product
$x_{i_{1}}\cdots x_{i_{r}}$ is negative and so $P_{G}\left(  \left[
x_{i}\right]  \right)  <P_{G}\left(  \left[  \left\vert x_{i}\right\vert
\right]  \right)  ,$ a contradiction.

Suppose now that $G$ has a proper even transversal $U$, and let $\left[
x_{i}\right]  \in\mathbb{S}_{p}^{n-1}$ be a positive eigenvector to
$\lambda^{\left(  p\right)  }\left(  G\right)  .$ Define $\left[
y_{i}\right]  \in\mathbb{S}_{p}^{n-1}$ by%
\[
y_{i}=\left\{
\begin{array}
[c]{ll}%
-x_{i}, & \text{if }i\in U,\\
x_{i}, & \text{if }i\notin U.
\end{array}
\right.
\]
Clearly $P_{G}\left(  \left[  x_{i}\right]  \right)  =P_{G}\left(  \left[
y_{i}\right]  \right)  ,$ and so $\left[  y_{i}\right]  $ is an eigenvector to
$\lambda^{\left(  p\right)  }\left(  G\right)  $ with entries of different signs.
\end{proof}

\section{\label{OPs}Relations of $\lambda^{\left(  p\right)  }$ and
$\lambda_{\min}^{\left(  p\right)  }$ to some graph operations}

We consider here only simple graph operations like union of graphs, graph
blow-up and star-like graphs. Some of the results are shaped after useful
spectral results for $2$-graphs.

\subsection{$\lambda^{\left(  p\right)  }$ and $\lambda_{\min}^{\left(
p\right)  }$ of blow-ups of graphs}

Given a graph $G\in\mathcal{G}^{r}\left(  n\right)  $ and positive integers
$k_{1},\ldots,k_{n}$, write $G\left(  k_{1},\ldots,k_{n}\right)  $ for the
graph obtained by replacing each vertex $v\in V\left(  G\right)  $ with a set
$U_{v}$ of size $x_{v}$ and each edge $\left\{  v_{1},\ldots,v_{r}\right\}
\in E\left(  G\right)  $ with a complete $r$-partite $r$-graph with vertex
classes $U_{v_{1}},\ldots,U_{v_{r}}.$ The graph $G\left(  k_{1},\ldots
,k_{h}\right)  $ is called a \emph{blow-up} of $G.$

Blow-ups are very useful in studying graphs, in particular for studying their
spectra. Let us note that if $G\in\mathcal{G}^{r}$ and $k\geq1,$ then
$\left\vert G\left(  k,\ldots,k\right)  \right\vert =k^{r}\left\vert
G\right\vert ;$ in the following proposition a similar property is proved for
$\lambda^{\left(  p\right)  }$ and $\lambda_{\min}^{\left(  p\right)  }$
also$.$

\begin{proposition}
\label{pro8}If $p\geq1,$ $k\geq1$ and $G\in\mathcal{G}^{r}\left(  n\right)  $,
then
\begin{align}
\lambda^{\left(  p\right)  }\left(  G\left(  k,\ldots,k\right)  \right)   &
=k^{r-r/p}\lambda^{\left(  p\right)  }\left(  G\right)  ,\nonumber\\
\lambda_{\min}^{\left(  p\right)  }\left(  G\left(  k,\ldots,k\right)
\right)   &  =k^{r-r/p}\lambda_{\min}^{\left(  p\right)  }\left(  G\right)
\label{blo1}%
\end{align}

\end{proposition}

\begin{proof}
We shall prove only (\ref{blo1}). By definition, $V\left(  G\left(
k,\ldots,k\right)  \right)  $ can be partitioned into $n$ disjoint sets
$U_{1},\ldots,U_{n}$ each consisting of $k$ vertices such that if $\left\{
i_{1},\ldots,i_{r}\right\}  \in E\left(  G\right)  ,$ then $\left\{
j_{1},\ldots,j_{r}\right\}  \in$ $E\left(  G\left(  k,\ldots,k\right)
\right)  $ for every $j_{1}\in U_{i_{1}},$ $j_{2}\in U_{i_{2}},\ldots,j_{r}\in
U_{i_{r}}.$ First we shall prove that
\begin{equation}
\lambda_{\min}^{\left(  p\right)  }\left(  G\left(  k,\ldots,k\right)
\right)  =\leq k^{r-r/p}\lambda^{\left(  p\right)  }\left(  G\right)  .
\label{in6}%
\end{equation}
Let $\left[  x_{i}\right]  \in\mathbb{S}_{p}^{n-1}$ be an eigenvector to
$\lambda_{\min}^{\left(  p\right)  }\left(  G\right)  .$ We define a new
eigenvector $\left[  y_{i}\right]  \in\mathbb{S}_{p}^{kn-1}$ as follows: for
each $i\in V\left(  G\left(  k,\ldots,k\right)  \right)  ,$ set $y_{i}%
:=k^{-1/p}x_{j},$ where $j$ is the unique value satisfying $i\in U_{j}.$
Clearly, $\left[  y_{i}\right]  \in\mathbb{S}_{p}^{nk-1},$ and therefore,%
\begin{align*}
\lambda_{\min}^{\left(  p\right)  }\left(  G\left(  k,\ldots,k\right)
\right)   &  \leq P_{G\left(  k,\ldots,k\right)  }\left(  \left[
y_{i}\right]  \right)  =r!\sum_{\left\{  i_{1},\ldots,i_{r}\right\}  \in
E\left(  G\right)  }\left(  \sum_{j\in U_{i_{1}}}x_{j}\right)  \cdots\left(
\sum_{j\in U_{i_{r}}}x_{j}\right) \\
&  =\frac{1}{k^{r/p}}k^{r}P_{G}\left(  \left[  x_{i}\right]  \right)
=k^{r-r/p}\lambda_{\min}^{\left(  p\right)  }\left(  G\right)  ,
\end{align*}
proving (\ref{in6}). To complete the proof of (\ref{in6}) we shall show that%
\begin{equation}
\lambda_{\min}^{\left(  p\right)  }\left(  G\left(  k,\ldots,k\right)
\right)  \leq k^{r-r/p}\lambda_{\min}^{\left(  p\right)  }\left(  G\right)  .
\label{in5}%
\end{equation}
Let $\left[  x_{i}\right]  \in\mathbb{S}_{p}^{nk-1}$\ be an eigenvector to
$\lambda_{\min}^{\left(  p\right)  }\left(  G\left(  k,\ldots,k\right)
\right)  .$ By definition,%
\begin{align*}
P_{G\left(  k,\ldots,k\right)  }\left(  \left[  x_{i}\right]  \right)   &
=r!\sum_{\left\{  i_{1},\ldots,i_{r}\right\}  \in E\left(  G\left(
k,\ldots,k\right)  \right)  }x_{i_{1}}\cdots x_{i_{r}}\\
&  =r!\sum_{\left\{  i_{1},\ldots,i_{r}\right\}  \in E\left(  G\right)
}\left(  \sum_{j\in U_{i_{1}}}x_{j}\right)  \cdots\left(  \sum_{j\in U_{i_{r}%
}}x_{j}\right)
\end{align*}
Next, Corollary \ref{corEX} implies that $x_{i}$ are the same within each
class $U_{j.}$ Now, setting for each $s\in\left[  n\right]  ,$ $y_{s}%
:=k^{1/p}x_{j},$ where $x_{j}\in U_{s},$ we get a vector $\mathbf{y}=\left[
y_{i}\right]  $ $\in\mathbb{S}_{p}^{n-1}.$ Also, $\mathbf{y}$ satisfies
\begin{align*}
\lambda_{\min}^{\left(  p\right)  }\left(  G\left(  k,\ldots,k\right)
\right)   &  =P_{G\left(  k,\ldots,k\right)  }\left(  \left[  x_{i}\right]
\right) \\
&  =r!\sum_{\left\{  i_{1},\ldots,i_{r}\right\}  \in E\left(  G\right)
}\left(  \sum_{j\in U_{i_{1}}}x_{j}\right)  \cdots\left(  \sum_{j\in U_{i_{r}%
}}x_{j}\right) \\
&  =r!k^{r-r/p}\sum_{\left\{  i_{1},\ldots,i_{r}\right\}  \in E\left(
G\right)  }y_{i_{1}}\cdots y_{i_{r}}=k^{r-r/p}P_{G}\left(  \mathbf{y}\right)
\\
&  \geq k^{r-r/p}\lambda_{\min}^{\left(  p\right)  }\left(  G\right)  .
\end{align*}
This completes the proof of (\ref{in5}); in view of (\ref{in6}), the proof of
(\ref{blo1}) is completed as well.
\end{proof}

\subsection{Weyl type inequalities and applications}

The following inequalities are shaped after Weyl's inequalities for Hermitian
matrices and have numerous applications.

\begin{proposition}
\label{pro_Weyl}If $G_{1}\in\mathcal{W}^{r}\left(  n\right)  ,$ $G_{2}%
\in\mathcal{W}^{r}\left(  n\right)  ,$ then
\[
\lambda^{\left(  p\right)  }\left(  G_{1}+G_{2}\right)  \leq\lambda^{\left(
p\right)  }\left(  G_{1}\right)  +\lambda^{\left(  p\right)  }\left(
G_{2}\right)  .
\]%
\begin{equation}
\lambda_{\min}^{\left(  p\right)  }\left(  G_{1}\right)  +\lambda_{\min
}^{\left(  p\right)  }\left(  G_{2}\right)  \leq\lambda_{\min}^{\left(
p\right)  }\left(  G_{1}+G_{2}\right)  \leq\lambda^{\left(  p\right)  }\left(
G_{1}\right)  +\lambda_{\min}^{\left(  p\right)  }\left(  G_{2}\right)  .
\label{Wey}%
\end{equation}

\end{proposition}

\begin{proof}
If $\left[  x_{i}\right]  \in\mathbb{S}_{p,+}^{n-1},$ $\left[  y_{i}\right]
\in\mathbb{S}_{p}^{n-1},$ and $\left[  z_{i}\right]  \in\mathbb{S}_{p}^{n-1}$
are eigenvectors to $\lambda^{\left(  p\right)  }\left(  G\right)  ,$
$\lambda_{\min}^{\left(  p\right)  }\left(  G\right)  $ and $\lambda_{\min
}^{\left(  p\right)  }\left(  G_{2}\right)  ,$ then
\begin{align*}
\lambda^{\left(  p\right)  }\left(  G_{1}+G_{2}\right)   &  =P_{G_{1}+G_{2}%
}\left(  \left[  x_{i}\right]  \right)  =P_{G_{1}}\left(  \left[
x_{i}\right]  \right)  +P_{G_{2}}\left(  \left[  x_{i}\right]  \right)
\leq\lambda^{\left(  p\right)  }\left(  G_{1}\right)  +\lambda^{\left(
p\right)  }\left(  G_{2}\right)  ,\\
\lambda_{\min}^{\left(  p\right)  }\left(  G_{1}+G_{2}\right)   &
=P_{G_{1}+G_{2}}\left(  \left[  y_{i}\right]  \right)  =P_{G_{1}}\left(
\left[  y_{i}\right]  \right)  +P_{G_{2}}\left(  \left[  y_{i}\right]
\right)  \geq\lambda_{\min}^{\left(  p\right)  }\left(  G_{1}\right)
+\lambda_{\min}^{\left(  p\right)  }\left(  G_{2}\right)  ,\\
\lambda_{\min}^{\left(  p\right)  }\left(  G_{2}\right)   &  =P_{G_{2}}\left(
\left[  z_{i}\right]  \right)  =P_{G_{1}+G_{2}}\left(  \left[  z_{i}\right]
\right)  -P_{G_{1}}\left(  \left[  z_{i}\right]  \right)  \geq\lambda_{\min
}^{\left(  p\right)  }\left(  G_{1}+G_{2}\right)  -\lambda^{\left(  p\right)
}\left(  G_{1}\right)  .
\end{align*}
This completes the proof of the proposition.
\end{proof}

Next, we shall deduce several useful applications, starting with perturbation
bounds on $\lambda^{\left(  p\right)  }\left(  G\right)  $ and $\lambda_{\min
}^{\left(  p\right)  }\left(  G\right)  .$

\begin{proposition}
\label{pro10}Let $p\geq1,$ $k\geq1$ and $G_{1}\in\mathcal{G}^{r}\left(
n\right)  ,$ $G_{2}\in\mathcal{G}^{r}\left(  n\right)  $. If $G_{1}$ and
$G_{2}$ differ in at most $k$ edges, then%
\begin{align}
\left\vert \lambda^{\left(  p\right)  }\left(  G_{1}\right)  -\lambda^{\left(
p\right)  }\left(  G_{2}\right)  \right\vert  &  \leq\left(  r!k\right)
^{1-1/p}\nonumber\\
\left\vert \lambda_{\min}^{\left(  p\right)  }\left(  G_{1}\right)
-\lambda_{\min}^{\left(  p\right)  }\left(  G_{2}\right)  \right\vert  &
\leq\left(  r!k\right)  ^{1-1/p} \label{perts}%
\end{align}

\end{proposition}

\begin{proof}
We shall prove only (\ref{perts}). Assume that $\lambda_{\min}^{\left(
p\right)  }\left(  G_{1}\right)  -\lambda_{\min}^{\left(  p\right)  }\left(
G_{2}\right)  \geq0$ and write $G_{1}\backslash G_{2},$ $G_{2}\backslash
G_{1},$ $G_{2}\cap G_{1},$ for the graphs with vertex set $\left[  n\right]  $
and edge sets $E\left(  G_{1}\right)  \backslash E\left(  G_{2}\right)  ,$
$E\left(  G_{2}\right)  \backslash E\left(  G_{1}\right)  ,$ $E\left(
G_{2}\right)  \cap E\left(  G_{1}\right)  .$ Now, inequalities (\ref{Wey})
imply that
\begin{align*}
\lambda_{\min}^{\left(  p\right)  }\left(  G_{1}\right)   &  \leq
\lambda^{\left(  p\right)  }\left(  G_{1}\backslash G_{2}\right)
+\lambda_{\min}^{\left(  p\right)  }\left(  G_{2}\cap G_{1}\right) \\
\lambda_{\min}^{\left(  p\right)  }\left(  G_{2}\right)   &  \geq\lambda
_{\min}^{\left(  p\right)  }\left(  G_{2}\cap G_{1}\right)  +\lambda_{\min
}^{\left(  p\right)  }\left(  G_{2}\backslash G_{1}\right)
\end{align*}
and so,%
\begin{align*}
\lambda_{\min}^{\left(  p\right)  }\left(  G_{1}\right)  -\lambda_{\min
}^{\left(  p\right)  }\left(  G_{2}\right)   &  \leq\lambda^{\left(  p\right)
}\left(  G_{1}\backslash G_{2}\right)  -\lambda_{\min}^{\left(  p\right)
}\left(  G_{2}\backslash G_{1}\right) \\
&  \leq\lambda^{\left(  p\right)  }\left(  G_{1}\backslash G_{2}\right)
+\lambda^{\left(  p\right)  }\left(  G_{2}\backslash G_{1}\right)  .
\end{align*}
Defining $G^{\prime}$ by $v\left(  G^{\prime}\right)  =\left[  n\right]  ,$
$E\left(  G^{\prime}\right)  =\left(  E\left(  G_{1}\right)  \backslash
E\left(  G_{2}\right)  \right)  \cup\left(  E\left(  G_{2}\right)  \backslash
E\left(  G_{1}\right)  \right)  ,$ we see that $\left\vert G^{\prime
}\right\vert \leq k,$ and so
\[
\lambda_{\min}^{\left(  p\right)  }\left(  G_{1}\right)  -\lambda_{\min
}^{\left(  p\right)  }\left(  G_{2}\right)  \leq\lambda^{\left(  p\right)
}\left(  G^{\prime}\right)  \leq\left(  r!k\right)  ^{1-1/p}.
\]
This completes the proof of (\ref{perts})
\end{proof}

The set $\mathcal{W}^{r}\left(  n\right)  $ is a complete metric space in any
$l^{q}$ norm, $1\leq q\leq\infty;$ Many graph parameters like $\left\vert
G\right\vert $ are continuous functions of $G\in\mathcal{W}^{r}\left(
n\right)  .$ Weyl's inequalities imply that for fixed $p\geq1,$ both
$\lambda^{\left(  p\right)  }\left(  G\right)  $ and $\lambda_{\min}^{\left(
p\right)  }\left(  G\right)  $ are also continuous functions of $G$.

\begin{proposition}
\label{pro_co}If $p\geq1,$ $G_{1}\in\mathcal{W}^{r}\left(  n\right)  ,$ and
$G_{2}\in\mathcal{W}^{r}\left(  n\right)  ,$ then
\begin{align*}
\left\vert \lambda^{\left(  p\right)  }\left(  G_{1}\right)  -\lambda^{\left(
p\right)  }\left(  G_{2}\right)  \right\vert  &  \leq\lambda^{\left(
p\right)  }\left(  \left\vert G_{1}-G_{2}\right\vert \right)  \leq\left\vert
G_{1}-G_{2}\right\vert _{p\left(  p-1\right)  }\\
\left\vert \lambda_{\min}^{\left(  p\right)  }\left(  G_{1}\right)
-\lambda_{\min}^{\left(  p\right)  }\left(  G_{2}\right)  \right\vert  &
\leq\lambda^{\left(  p\right)  }\left(  \left\vert G_{1}-G_{2}\right\vert
\right)  \leq\left\vert G_{1}-G_{2}\right\vert _{p/\left(  p-1\right)  }%
\end{align*}

In particular, if $p$ is fixed and $G\in\mathcal{W}^{r}\left(  n\right)  ,$
then $\lambda^{\left(  p\right)  }\left(  G\right)  $ \ and $\lambda_{\min
}^{\left(  p\right)  }\left(  G\right)  $ are continuous functions of $G$.
\end{proposition}

Note that Proposition \ref{pro_co} makes sense for $p=1$ with the proviso
$p/\left(  p-1\right)  =\infty.$

Another consequence from Weyl's inequalities are two Nordhaus-Stewart type
bounds about $\lambda^{\left(  p\right)  }$. obtained following the footprints
of Nosal \cite{Nos70}:

\begin{proposition}
Let $G\in\mathcal{G}^{r}\left(  n\right)  $ and $\overline{G}$ be its
complement. If $p\geq1,$ then%
\begin{equation}
\lambda^{\left(  p\right)  }\left(  G\right)  +\lambda^{\left(  p\right)
}\left(  \overline{G}\right)  \leq2^{1/p}\left(  n\right)  _{r}^{1-1/p}
\label{NS0}%
\end{equation}
and%
\begin{equation}
\lambda^{\left(  p\right)  }\left(  G\right)  +\lambda^{\left(  p\right)
}\left(  \overline{G}\right)  \geq\left(  n\right)  _{r}/n^{r/p} \label{NS1}%
\end{equation}
If equality holds in (\ref{NS1}), then $G$ is regular. If $p\geq r$, and
$G\ $is regular, then equality holds in (\ref{NS1}).
\end{proposition}

Note that the upper and lower bounds are close within a multiplicative factor
of $2^{1/p};$ however, unlike (\ref{NS1}), the upper bound (\ref{NS0}) seems
not too tight. This observation prompts the following Nordhaus-Stewart type problems.

\begin{problem}
If $p>1,$ find%
\[
\max_{G\in\mathcal{G}^{r}\left(  n\right)  }\lambda^{\left(  p\right)
}\left(  G\right)  +\lambda^{\left(  p\right)  }\left(  \overline{G}\right)
,
\]
and
\[
\min_{G\in\mathcal{G}^{r}\left(  n\right)  }\lambda_{\min}^{\left(  p\right)
}\left(  G\right)  +\lambda_{\min}^{\left(  p\right)  }\left(  \overline
{G}\right)  .
\]

\end{problem}

\bigskip

\subsection{Star-like $r$-graphs}

In this subsection we discuss $r$-graphs with certain intersection properties.
The possible variations are indeed numerous but we shall focus on two
constructions only. Our interest is motivated by certain extremal problems
discussed later.

Let $r\geq3$ and $G\in\mathcal{G}^{r-1}\left(  n-1\right)  .$ Choose a vertex
$v\notin V\left(  G\right)  $ and define the graph $G\vee K_{1}\in
\mathcal{G}^{r}\left(  n\right)  $ by
\[
V\left(  G\vee K_{1}\right)  :=V\left(  G\right)  \cup\left\{  v\right\}
,\text{ \ \ \ }E\left(  G\vee K_{1}\right)  :=\left\{  e\cup\left\{
v\right\}  :e\in E\left(  G\right)  \right\}  .
\]

\begin{proposition}
For every $p\geq1$ and every $G\in\mathcal{G}^{r-1},$
\begin{align*}
\lambda^{\left(  p\right)  }\left(  G\vee K_{1}\right)   &  =r^{1-r/p}\left(
r-1\right)  ^{\left(  r-1\right)  /p}\lambda^{\left(  p\right)  }\left(
G\right)  ,\\
\lambda^{\left(  p\right)  }\left(  G\vee K_{1}\right)   &  =-r^{1-r/p}\left(
r-1\right)  ^{\left(  r-1\right)  /p}\lambda^{\left(  p\right)  }\left(
G\right)  .
\end{align*}

\end{proposition}

\begin{proof}
Take a nonnegative eigenvector $\mathbf{x}=\left(  x_{0},\ldots,x_{n}\right)
\in\mathbb{S}_{p,+}^{n}$ to $\lambda^{\left(  p\right)  }\left(  G\vee
K_{1}\right)  ;$ suppose that $x_{2},\ldots,x_{n}$ are the entries
corresponding to vertices in $V\left(  G\right)  $ and $x_{1}$ is the entry
corresponding to $v.$
\begin{align*}
\lambda^{\left(  p\right)  }\left(  G\vee K_{1}\right)   &  =\max_{x_{1}%
^{p}+\cdots+x_{n}^{p}=1}P_{G\vee K_{1}}\left(  \mathbf{x}\right)
=r\max_{x_{1}^{p}+\cdots+x_{n}^{p}=1}x_{0}P_{G}\left(  \mathbf{x}^{\prime
}\right) \\
&  =r\max_{0\leq x_{1}\leq1}x_{1}\max_{x_{1}^{p}+\cdots+x_{n}^{p}=1}%
P_{G}\left(  \mathbf{x}^{\prime}\right)  =r\max_{0\leq x_{1}\leq1}x_{1}%
\lambda^{\left(  p\right)  }\left(  G\right)  \left(  1-x_{1}^{p}\right)
^{\left(  r-1\right)  /p}\\
&  =r\lambda^{\left(  p\right)  }\left(  G\right)  \max_{0\leq x_{1}\leq
1}x_{1}\left(  1-x_{1}^{p}\right)  ^{\left(  r-1\right)  /p}%
\end{align*}
Using calculus, we find that the maximum above is attained at $x_{1}=r^{-1/p}$
and the desired result follows
\end{proof}

In particular if $G$ is $K_{n}^{r-1},$ the complete $\left(  r-1\right)
$-graph of order $n,$ we obtain%
\[
\lambda^{\left(  p\right)  }\left(  K_{n}^{r-1}\vee K_{1}\right)
=r^{1-r/p}\left(  r-1\right)  ^{\left(  r-1\right)  /p}\left(  n\right)
_{r-1}n^{-\left(  r-1\right)  /p}.
\]

The above construction can be generalized as follows: Let $r\geq3$ and
$G\in\mathcal{G}^{r-1}\left(  n-t\right)  .$ Choose a set of $t$ vertices $T$
with $T\cap V\left(  G\right)  =\varnothing$ and define the graph $G\vee
tK_{1}\in\mathcal{G}^{r}\left(  n\right)  $ by
\[
V\left(  G\vee tK_{1}\right)  :=V\left(  G\right)  \cup T,\text{
\ \ \ }E\left(  G\vee tK_{1}\right)  :=\left\{  e\cup\left\{  v\right\}  :v\in
T,\text{ }e\in E\left(  G\right)  \right\}  .
\]
Exactly as in the previous proposition we obtain the following relations.

\begin{proposition}
Let $r\geq3,$ $t\geq1$ and $G\in\mathcal{G}^{r-1}$. For every $p\geq1,$
\begin{align*}
\lambda^{\left(  p\right)  }\left(  G\vee tK_{1}\right)   &  =t^{1-1/p}%
r^{1-r/p}\left(  r-1\right)  ^{\left(  r-1\right)  /p}\lambda^{\left(
p\right)  }\left(  G\right)  ,\\
\lambda_{\min}^{\left(  p\right)  }\left(  G\vee tK_{1}\right)   &
=-t^{1-1/p}r^{1-r/p}\left(  r-1\right)  ^{\left(  r-1\right)  /p}%
\lambda^{\left(  p\right)  }\left(  G\right)
\end{align*}

\end{proposition}

Here is another construction similar to the above: Let $r\geq3,$ $r>t\geq1$
and $G\in\mathcal{G}^{r-t}\left(  n-t\right)  .$ Choose a set of $t$ vertices
$T$ with $T\cap V\left(  G\right)  =\varnothing$ and define $G\vee K_{t}%
^{t}\in\mathcal{G}^{r}\left(  n\right)  $ by
\[
V\left(  G\vee K_{t}^{t}\right)  :=V\left(  G\right)  \cup T,\text{
\ \ \ }E\left(  G\vee K_{t}^{t}\right)  :=\left\{  e\cup T:\text{ }e\in
E\left(  G\right)  \right\}  .
\]
A graph with the structure of $G\vee K_{t}^{t}$ is called a $t$\emph{-star}.
The $t$-star $K_{n-t}^{r-t}\vee K_{t}^{t}$ is called a \emph{complete }%
$t$\emph{-star} of order $n$ and is denoted by $S_{t,n}^{r}.$

\begin{proposition}
Let $r\geq3,$ $r>t\geq1$ and $G\in\mathcal{G}^{r-t}.$ For every $p\geq1,$
\begin{align*}
\lambda^{\left(  p\right)  }\left(  G\vee K_{t}^{t}\right)   &  =\frac
{r!\left(  r-t\right)  ^{\left(  r-t\right)  /p}}{r^{r/p}\left(  r-t\right)
!}\lambda^{\left(  p\right)  }\left(  G\right)  ,\\
\lambda_{\min}^{\left(  p\right)  }\left(  G\vee K_{t}^{t}\right)   &
=-\frac{r!\left(  r-t\right)  ^{\left(  r-t\right)  /p}}{r^{r/p}\left(
r-t\right)  !}\lambda^{\left(  p\right)  }\left(  G\right)  .
\end{align*}
In particular,%
\begin{align}
\lambda^{\left(  p\right)  }\left(  S_{t,n}^{r}\right)   &  =\frac{\left(
r\right)  _{t}\left(  r-t\right)  ^{\left(  r-t\right)  /p}\left(  n-t\right)
_{r-t}}{r^{r/p}\left(  n-t\right)  ^{\left(  r-t\right)  /p}},\label{stal}\\
\lambda_{\min}^{\left(  p\right)  }\left(  S_{t,n}^{r}\right)   &
=-\frac{\left(  r\right)  _{t}\left(  r-t\right)  ^{\left(  r-t\right)
/p}\left(  n-t\right)  _{r-t}}{r^{r/p}\left(  n-t\right)  ^{\left(
r-t\right)  /p}},\nonumber
\end{align}
and eigenvectors to $\lambda^{\left(  p\right)  }\left(  S_{t,n}^{r}\right)  $
and $\lambda_{\min}^{\left(  p\right)  }\left(  S_{t,n}^{r}\right)  $ have
only nonzero entries.
\end{proposition}

Note that equation (\ref{stal}) has been proved in \cite{KLM13}, Lemma
13.\medskip

\section{\label{Props}More properties of $\lambda^{\left(  p\right)  }$}

In this section we present results on $\lambda^{\left(  p\right)  }\left(
G\right)  $ if $G$ is a graph with some special property. Our first goal is to
improve the bound $\lambda^{\left(  p\right)  }\left(  G\right)  \leq\left(
r!\left\vert G\right\vert \right)  ^{1-1/p}$ in (\ref{upb}), using extra
information about $G.$

\subsection{\label{CMP}$k$-partite and $k$-chromatic graphs}

If $G$ is a $k$-partite $2$-graph of order $n$, then Cvetkovi\'{c} showed that
$\lambda\left(  G\right)  \leq\left(  1-1/k\right)  n$ and Edwards and Elphick
\cite{EdEl83} improved that to $\lambda\left(  G\right)  \leq\sqrt{2\left(
1-1/k\right)  \left\vert G\right\vert }.$ The following theorems extend these
inequalities in several directions. First, using the proof of inequality
(\ref{ginu}) one can verify the following upper bounds for $k$-partite $r$-graphs.

\begin{theorem}
\label{th_kpart}Let $k>r,$ $p\geq1$ and $G\in\mathcal{G}^{r}.$ If $G$ is
$k$-partite, then
\begin{equation}
\lambda^{\left(  p\right)  }\left(  G\right)  \leq\left(  \left(  k\right)
_{r}/k^{r}\right)  ^{1/p}\left(  r!\left\vert G\right\vert \right)  ^{1-1/p}
\label{ukl}%
\end{equation}
If $p>1$ and $G$ has no isolated vertices, equality holds if and only if $G$
is a complete $k$-partite, with equal vertex classes. Furthermore, if $G$ is
of order $n,$ then
\[
\lambda^{\left(  p\right)  }\left(  G\right)  \leq\left(  \left(  k\right)
_{r}/k^{r}\right)  n^{r-r/p}.
\]
Equality holds if and only if $G$ is a complete $k$-partite graph with equal
vertex classes.
\end{theorem}

Note that inequality (\ref{ginu}) follows from this more general theorem
because $\left(  k\right)  _{r}/k^{r}\leq\left(  n\right)  _{r}/n^{r}<1$.

If $k=r$ Theorem the above inequalities become particularly simple, but more
cases of equality arise.

\begin{proposition}
\label{pro_rr}Let $p\geq1$ and $G\in\mathcal{G}^{r}.$ If $G$ is $r$-partite,
then
\begin{equation}
\lambda^{\left(  p\right)  }\left(  G\right)  \leq\left(  r!/r^{r/p}\right)
\left\vert G\right\vert ^{1-1/p}. \label{ukl1}%
\end{equation}
If $p>1,$ equality holds if and only if $G$ is a complete $r$-partite.
\end{proposition}

In particular, if $G\in\mathcal{G}^{r}$ is $r$-partite and $k_{1},\ldots
k_{r}$ are the sizes of its vertex classes, then
\[
\lambda^{\left(  p\right)  }\left(  G\right)  \leq\left(  r!/r^{r/p}\right)
\left(  k_{1}\cdots k_{r}\right)  ^{1-1/p},
\]
Equality holds in if and only if $G$ is a complete $r$-partite $r$-graph.

In the following proposition we deduce bounds on $\lambda^{\left(  p\right)
}$ of the Tur\'{a}n $2$-graph. A cruder form of these bounds has been given in
\cite{KLM13}, Lemma 13.

\begin{proposition}
If $T_{r}\left(  n\right)  $ is the Tur\'{a}n $2$-graph of order $n,$ then
\begin{equation}
\lambda^{\left(  1\right)  }\left(  T_{k}\left(  n\right)  \right)  =1-1/k,
\label{l1}%
\end{equation}
and for every $p>1,$%
\begin{equation}
2\left\vert T_{k}\left(  n\right)  \right\vert n^{-2/p}\leq\lambda^{\left(
p\right)  }\left(  T_{k}\left(  n\right)  \right)  \leq2\left\vert
T_{k}\left(  n\right)  \right\vert n^{-2/p}\left(  1+k/\left(  4pn^{2}\right)
\right)  \label{lp}%
\end{equation}

\end{proposition}

\begin{proof}
[Sketch of the proof]The equality (\ref{l1}) follows from (\ref{ukl}) and
(\ref{ukl1}). The lower bound in (\ref{lp}) follows by (\ref{ginl}). The upper
bound follows from (\ref{ukl}) and (\ref{ukl1}), using the fact that
$2\left\vert T_{k}\left(  n\right)  \right\vert \geq\left(  1-1/k\right)
n^{2}-k/4$ and Bernoulli's inequality.
\end{proof}

\medskip

Eigenvalues of $2$-graphs have a lot of fascinating relations with the
chromatic number and such seems to be the case with hypergraphs as well. We
state here a results similar to the above mentioned bound of Edwards and
Elphick. The proof method is described in \ref{Flats}.

\begin{theorem}
\label{Wchr}If $G\in\mathcal{G}^{r}\left(  n\right)  $ and $\chi\left(
G\right)  =k,$ then
\[
\lambda^{\left(  p\right)  }\left(  G\right)  \leq\left(  1-k^{-r+1}\right)
^{1/p}\left(  r!\left\vert G\right\vert \right)  ^{1-1/p}%
\]
and
\[
\lambda^{\left(  p\right)  }\left(  G\right)  \leq\left(  1-k^{-r+1}\right)
n^{r-r/p}.
\]

\end{theorem}

These bounds are essentially best possible as the shown by the complete
$k$-chromatic graph with with chromatic classes of sizes $\left\lfloor
n/k\right\rfloor $ and $\left\lceil n/k\right\rceil .$

\subsection{A coloring theme of Szekeres and Wilf}

One of the most appealing results in spectral graph theory is the inequality
$\chi\left(  G\right)  \leq\lambda\left(  G\right)  +1,$ proved for $2$-graphs
by Wilf in \cite{Wil67}. Somewhat later Szekeres and Wilf\cite{SzWi68} showed
that this inequality belongs to the study of a fundamental parameter called
\emph{graph degeneracy}. Similar results hold for hypergraphs as well, but we
need a few definitions first: a $\beta$\emph{-star} with vertex $v$ is a graph
such that the intersection of every two edges is $\left\{  v\right\}  .$ If
$G$ is a graph and $v\in V\left(  G\right)  ,$ the $\beta$\emph{-degree}%
\textbf{ }$d^{\beta}\left(  v\right)  $ of $v$\ is the size the maximum
$\beta$-star with vertex $v;$ $\delta^{\beta}\left(  G\right)  $ is the
smallest $\beta$-degree of\textbf{ }$G.$

Berge \cite{Ber87}, p.116, generalized the result of Wilf and Szekeres proving
that for every graph $G$%
\begin{equation}
\chi\left(  G\right)  \leq\max_{H\subset G}\delta^{\beta}\left(  H\right)  +1,
\label{Berbo}%
\end{equation}
which implies for $2$ graphs that $\chi\left(  G\right)  \leq\lambda\left(
G\right)  +1$. Moreover, Cooper and Dutle \cite{CoDu11} observed that for
every $G\in\mathcal{G}^{r}\left(  n\right)  ,$
\[
\chi\left(  G\right)  \leq\lambda\left(  G\right)  /\left(  r-1\right)  !+1;
\]
however, for $r\geq3,$ there is a certain incongruity in this bound, as the
left side never exceeds $n/\left(  r-1\right)  $ while almost surely
$\lambda\left(  G\right)  =\Theta\left(  n^{r-1}\right)  .$ We propose a tight
generalization of Wilf's bound of a different kind. Recall that the
$2$-section $G_{\left(  2\right)  }$ of a graph $G$ is a $2$-graph with
$V\left(  G_{\left(  2\right)  }\right)  =$ $V\left(  G\right)  $ and
$E\left(  G_{\left(  2\right)  }\right)  $ consisting of all pairs of vertices
that belong to an edge of $G.$ Clearly,%
\[
\lambda\left(  G_{\left(  2\right)  }\right)  \geq\delta\left(  G_{\left(
2\right)  }\right)  \geq\left(  r-1\right)  \delta^{\beta}\left(  G\right)  .
\]
and this, together with (\ref{Berbo}), gives the following generalization of
Wilf's bound.

\begin{proposition}
If $G\in\mathcal{G}^{r},$ then
\[
\chi\left(  G\right)  \leq\lambda\left(  G_{\left(  2\right)  }\right)
/\left(  r-1\right)  +1.
\]

\end{proposition}

It will be interesting to prove that for every $G\in\mathcal{G}^{r}\left(
n\right)  ,$
\[
\lambda\left(  G_{\left(  2\right)  }\right)  \leq\frac{\lambda\left(
G\right)  }{\left(  r-2\right)  !}%
\]
which would imply the results of Cooper and Dutle; also, this inequality
suggests a more general problem.

\begin{problem}
If $G\in\mathcal{G}^{r}\left(  n\right)  ,$ $2\leq k<r,$ and $p\geq1,$ find
tight upper and lower bounds on $\lambda^{\left(  p\right)  }\left(
G_{\left(  k\right)  }\right)  $?
\end{problem}

\subsection{$\lambda^{\left(  p\right)  }$ and vertex degrees}

For the largest eigenvalue of a $2$-graph there is a tremendous variety of
bounds using the degrees of $G.$ Unfortunately the situation with
$\lambda^{\left(  p\right)  }$ is more subtle even for $2$-graphs. First, we
saw in Theorem \ref{pro1} that the inequality $\lambda\left(  G\right)
\geq2\left\vert G\right\vert /n$ for $2$-graphs generalizes seamlessly for
$\lambda^{\left(  p\right)  }\left(  G\right)  $ of an $r$-graph $G$ and any
$p\geq1$, but the condition for equality becomes quite intricate, even for
$r=2;$ see the discussion in Subsections \ref{regs}, \ref{Cycs} and \ref{Lins}
for a number of special cases. In general, Problem \ref{Pro_reg} captures the
main difficulty of this topic.

Another cornerstone bound on $\lambda\left(  G\right)  $ for a $2$-graph $G$
with maximum degree $\Delta,$ is the inequality $\lambda\left(  G\right)
\leq\Delta.$ This inequality also generalizes to $r$-graphs, but not so directly.

\begin{proposition}
\label{pro3}Let $G\in\mathcal{W}^{r}\left(  n\right)  $ and $\Delta\left(
G\right)  =\Delta.$

(i) If $p\geq r,$ then
\begin{equation}
\lambda^{\left(  p\right)  }\left(  G\right)  \leq\frac{\left(  r-1\right)
!\Delta}{n^{r/p-1}}. \label{inmax}%
\end{equation}
If $p>r,$ equality holds if and only if $G$ is regular. If $p=r,$ equality
holds if and only if $G$ contains a $\Delta$-regular component.

(ii) If $G\in\mathcal{G}^{r}\left(  n\right)  $ and $1<p<r,$ then
\[
\lambda^{\left(  p\right)  }\left(  G\right)  <\left(  r-1\right)
!\Delta^{\left(  1-1/p\right)  /\left(  1-1/r\right)  }.
\]

(iii) For every $\Delta,$ there exists a $G\in\mathcal{G}^{r}\left(  n\right)
$ such that $\Delta\left(  G\right)  \leq\Delta$ and
\[
\lambda^{\left(  p\right)  }\left(  G\right)  >\left(  1-o\left(  1\right)
\right)  \left(  r-1\right)  !\Delta^{\left(  1-1/p\right)  /\left(
1-1/r\right)  }%
\]
whenever $1<p<r.$
\end{proposition}

\begin{proof}
Let $\left[  x_{i}\right]  \in\mathbb{S}_{p,+}^{n-1}$ be an eigenvector to
$\lambda^{\left(  p\right)  }\left(  G\right)  .$ Assume that $p\geq r$ and
let $x_{k}=\max\left\{  x_{1},\ldots,x_{n}\right\}  .$ The eigenequations for
$\lambda^{\left(  p\right)  }\left(  G\right)  $ and the vertex $k$ implies
that%
\[
\frac{\lambda^{\left(  p\right)  }\left(  G\right)  }{\left(  r-1\right)
!}x_{k}^{p-1}=\sum_{\left\{  k,i_{1},\ldots,i_{r-1}\right\}  \in E\left(
G\right)  }G\left(  \left\{  k,i_{1},\ldots,i_{r-1}\right\}  \right)
x_{i_{1}}\cdots x_{i_{r-1}}\leq\Delta x_{k}^{r-1}%
\]
Since $x_{k}\geq n^{-1/p}$ and $p\geq r,$ we find that%
\[
\frac{\lambda^{\left(  p\right)  }\left(  G\right)  }{\left(  r-1\right)
!}\leq\Delta x_{k}^{r-p}\leq\Delta\left(  n^{-1/p}\right)  ^{r-p}=\frac
{\Delta}{n^{r/p-1}},
\]
proving (\ref{inmax}). Now if $p>r$ and we have equality in (\ref{inmax}),
then $x_{k}=n^{-1/p}$ and so $x_{1}=\cdots=x_{n}=n^{-1/p}.$ Thus, equations
(\ref{eequ}) show that all degrees are equal to $\lambda^{\left(  p\right)
}\left(  G\right)  /\left(  r-1\right)  !=\Delta$, and $G$ is regular. On the
other hand,
\[
\frac{r\left\vert G\right\vert }{n^{r/p}}=\frac{1}{\left(  r-1\right)  !}%
P_{G}\left(  n^{-1/p}\mathbf{j}_{n}\right)  \leq\frac{\lambda^{\left(
p\right)  }\left(  G\right)  }{\left(  r-1\right)  !}\leq\frac{\Delta
}{n^{r/p-1}},
\]
and so if $G$ is regular, then $\lambda^{\left(  p\right)  }\left(  G\right)
/\left(  r-1\right)  !=\Delta n^{1-r/p},$ completing the proof of \emph{(i)
}for $p>r$. We leave the case of equality for $p=r$ to the reader.

To prove \emph{(ii)} let $s=p\left(  r-1\right)  /\left(  r-p\right)  ,$ and
note that $s>1.$ The PM inequality implies that
\begin{align*}
\frac{\lambda^{\left(  p\right)  }\left(  G\right)  }{\left(  r-1\right)
!}x_{k}^{p-1}  &  =\sum_{\left\{  k,i_{1},\ldots,i_{r-1}\right\}  \in E\left(
G\right)  }x_{i_{1}}\cdots x_{i_{r-1}}\leq\Delta^{1-1/s}\left(  \sum_{\left\{
k,i_{1},\ldots,i_{r-1}\right\}  \in E\left(  G\right)  }x_{i_{1}}^{s}\cdots
x_{i_{r-1}}^{s}\right)  ^{1/s}\\
&  \leq\Delta^{1-1/s}\left(  \sum_{\left\{  k,i_{1},\ldots,i_{r-1}\right\}
\in E\left(  G\right)  }x_{i_{1}}^{p}\cdots x_{i_{r-1}}^{p}x^{\left(
r-1\right)  \left(  s-p\right)  }\right)  ^{1/s}\\
&  =\Delta^{1-1/s}\left(  \sum_{\left\{  k,i_{1},\ldots,i_{r-1}\right\}  \in
E\left(  G\right)  }x_{i_{1}}^{p}\cdots x_{i_{r-1}}^{p}\right)  ^{1/s}x^{p-1}.
\end{align*}
Hence,%
\[
\frac{\lambda^{\left(  p\right)  }\left(  G\right)  }{\left(  r-1\right)
!}\leq\Delta^{1-1/s}\left(  \sum_{\left\{  k,i_{1},\ldots,i_{r-1}\right\}  \in
E\left(  G\right)  }x_{i_{1}}^{p}\cdots x_{i_{r-1}}^{p}\right)  ^{1/s}.
\]
Maclaurin's inequality and the fact that $x_{1}^{p}+\cdots+x_{p}^{p}=1$ imply
that
\[
\frac{\lambda^{\left(  p\right)  }\left(  G\right)  }{\left(  r-1\right)
!}<\Delta^{1-1/s}=\Delta^{\left(  1-1/p\right)  /\left(  1-1/r\right)  },
\]
completing the proof of \emph{(ii).}

To prove \emph{(iii)}, let $k$ be the maximal integer such that
\[
\binom{k-1}{r-1}\leq\Delta
\]
and let $G$ be union of disjoint $K_{k}^{r}.$ Now, Propositions \ref{pro_com}
and \ref{pro_sub}, together with an easy calculation, give the result.
\end{proof}

For a\ $2$-graph $G$ it is known also that $\lambda\left(  G\right)  \geq
\sqrt{\Delta\left(  G\right)  }.$ This bound also can be extended to
$r$-graphs as follows.

\begin{proposition}
If $p\geq1,$ $G\in\mathcal{G}^{r}\left(  n\right)  $ and $\Delta\left(
G\right)  =\Delta$, then%
\[
\lambda^{\left(  p\right)  }\left(  G\right)  \geq\left(  r!/r^{r/p}\right)
\Delta^{1-\left(  r-1\right)  /p}.
\]
If $G$ is the $\beta$-star $S_{\Delta}^{r},$ then equality holds.
\end{proposition}

\begin{proof}
[Sketch of a proof]We can suppose that $G$ has precisely $\Delta$ edges all
sharing a common vertex $u$. Let $n=V\left(  G\right)  $ and note that
$\Delta\left(  r-1\right)  \geq\left(  n-1\right)  .$ Construct an $n$-vector
$\left[  y_{i}\right]  $ by letting $y_{u}=r^{-1/p}$ and $y_{i}=\left(
\left(  r-1\right)  /r\left(  n-1\right)  \right)  ^{1/p}$ for the remaining
entries. Clearly $\left[  y_{i}\right]  \in\mathbb{S}_{p,+}^{n-1}$ and%
\[
\lambda^{\left(  p\right)  }\left(  G\right)  \geq P_{G}\left(  \left[
y_{i}\right]  \right)  =\frac{r!\Delta\left(  r-1\right)  ^{\left(
r-1\right)  /p}}{r^{r/p}\left(  n-1\right)  ^{\left(  r-1\right)  /p}}%
=\frac{r!\Delta}{r^{r/p}\Delta^{\left(  r-1\right)  /p}}=\left(
r!/r^{r/p}\right)  \Delta^{1-\left(  r-1\right)  /p}=\lambda^{\left(
p\right)  }\left(  S_{\Delta}^{r}\right)  ,
\]
completing the proof.
\end{proof}

A very useful bound for $2$-graphs is the inequality of Hofmeister
\cite{Hof88}%
\[
\lambda\left(  G\right)  \geq\left(  \frac{1}{n}\sum d^{2}\left(  u\right)
\right)  ^{1/2}.
\]
We have no clues how this inequality can be generalized to $r$-graphs, but we
shall outline a limitation to possible generalizations. In the concluding
section we shall return to this topic.

\begin{proposition}
If $r\geq2$ and $\varepsilon>0,$ there is a $G\in\mathcal{G}^{r}\left(
n\right)  $ such that
\[
\lambda\left(  G\right)  <\left(  r-1\right)  !\left(  \frac{1}{n}\sum
d^{r/\left(  r-1\right)  +\varepsilon}\left(  u\right)  \right)  ^{1/\left(
r/\left(  r-1\right)  +\varepsilon\right)  }.
\]

\end{proposition}

\begin{proof}
Take $G$ to be the complete $r$-partite $r$-graph with vertex classes
$V_{1},\ldots,V_{r},$ where $\left\vert V_{1}\right\vert =1,$ and $\left\vert
V_{2}\right\vert =\cdots=\left\vert V_{r}\right\vert =k.$ Clearly, $v\left(
G\right)  =k\left(  r-1\right)  +1$ and
\[
\frac{1}{k\left(  r-1\right)  +1}\sum_{u\in V\left(  G\right)  }d^{r/\left(
r-1\right)  +\varepsilon}\left(  u\right)  >\frac{1}{k\left(  r-1\right)
+1}\left(  k^{r-1}\right)  ^{r/\left(  r-1\right)  +\varepsilon}=\frac
{1}{k\left(  r-1\right)  +1}k^{r+\varepsilon\left(  r-1\right)  }.
\]
On the other hand,
\begin{align*}
\left(  \frac{\lambda\left(  G\right)  }{\left(  r-1\right)  !}\right)
^{r/\left(  r-1\right)  +\varepsilon}  &  =\left(  r\left\vert G\right\vert
^{1-1/r}\right)  ^{r/\left(  r-1\right)  +\varepsilon}=\left(  r\left\vert
k^{r-1}\right\vert ^{1-1/r}\right)  ^{r/\left(  r-1\right)  +\varepsilon}\\
&  =r^{r/\left(  r-1\right)  +\varepsilon}k^{r-1+\varepsilon\left(
r-1\right)  ^{2}/r}.
\end{align*}
and a short calculation gives the desired inequality for $k$ sufficiently large.
\end{proof}

\subsection{$\lambda^{\left(  p\right)  }$ and set degrees}

For hypergraphs the concept of degree can be extended from vertices to sets,
and these set degrees are at least as important as vertex degree for
$2$-graphs. Thus, given a graph $G\ $and a set $U\subset V\left(  G\right)  ,$
the \emph{set degree }$d\left(  U\right)  $ of $U$ is defined as
\[
d\left(  U\right)  =\sum_{e\in E\left(  G\right)  ,\text{ }U\subset e}G\left(
e\right)  .
\]
$G$ is called $k$\emph{-set regular} if the degrees of all $k$-subsets of
$V\left(  G\right)  $ are equal. Note that if $G\in\mathcal{W}^{r}$ is $k$-set
regular, then it is $l$-set regular for each $l\in\left[  k\right]  ,$ and for
any $k$-set $U\subset V\left(  G\right)  $, we have%
\begin{equation}
d\left(  U\right)  =\left\vert G\right\vert \binom{r}{k}/\binom{n}%
{k}=\left\vert G\right\vert \left(  n\right)  _{k}/\left(  r\right)  _{k}.
\label{deeq}%
\end{equation}
In the same vein, let us also define $\Delta_{k}\left(  G\right)
=\max\left\{  d\left(  U\right)  :U\subset V\left(  G\right)  ,\text{
}\left\vert U\right\vert =k\right\}  .$

It turns out that $k$-set regularity goes quite well with $\lambda^{\left(
p\right)  }$ if $k\geq2.$

\begin{theorem}
\label{th_Sdeg}Let $r>k\geq2.$ If a graph $G\in\mathcal{W}^{r}\left(
n\right)  $ is $k$-set regular, then
\[
\lambda^{\left(  p\right)  }\left(  G\right)  =r!\left\vert G\right\vert
/n^{r/p}%
\]
and $\mathbf{j}_{n}$ is and eigenvector to $\lambda^{\left(  p\right)
}\left(  G\right)  $.
\end{theorem}

\begin{proof}
[Sketch of a proof]We know that $\lambda^{\left(  p\right)  }\left(  G\right)
\geq r!\left\vert G\right\vert /n^{r/p}$, so we shall prove the opposite
inequality. Taking an eigenvector $\left[  x_{i}\right]  \in\mathbb{S}%
_{p,+}^{n-1},$ Maclaurin's and the PM inequalities imply that
\begin{align*}
\lambda^{\left(  p\right)  }\left(  G\right)   &  =r!\sum_{\left\{
i_{1},\ldots,i_{r}\right\}  \in E\left(  G\right)  }G\left(  \left\{
i_{1},\ldots,i_{r}\right\}  \right)  x_{i_{1}}\cdots x_{i_{r}}\\
&  \leq r!\sum_{\left\{  i_{1},\ldots,i_{r}\right\}  \in E\left(  G\right)
}G\left(  \left\{  i_{1},\ldots,i_{r}\right\}  \right)  \left(  \mathbf{S}%
_{k}\left(  x_{i_{1}}^{p},\ldots,x_{i_{r}}^{p}\right)  /\binom{r}{k}\right)
^{r/\left(  kp\right)  }\\
&  \leq r!\left\vert G\right\vert \left(  \frac{1}{\left\vert G\right\vert
}\binom{r}{k}^{-1}\sum_{\left\{  i_{1},\ldots,i_{r}\right\}  \in E\left(
G\right)  }G\left(  \left\{  i_{1},\ldots,i_{r}\right\}  \right)
\mathbf{S}_{k}\left(  x_{i_{1}}^{p},\ldots,x_{i_{r}}^{p}\right)  \right)
^{r/\left(  kp\right)  }\\
&  =r!\left\vert G\right\vert ^{1-r/\left(  kp\right)  }\left(  \frac
{1}{\left\vert G\right\vert }\binom{r}{k}^{-1}\sum_{\left\{  i_{1}%
,\ldots,i_{k}\right\}  \in V^{\left(  r\right)  }}d\left(  \left\{
i_{1},\ldots,i_{k}\right\}  \right)  x_{i_{1}}^{p}\cdots x_{i_{k}}^{p}\right)
^{r/\left(  kp\right)  }.
\end{align*}
Now, (\ref{deeq}), and again Maclaurin's inequality give%
\begin{align*}
\lambda^{\left(  p\right)  }\left(  G\right)   &  =r!\left\vert G\right\vert
\left(  \binom{n}{k}^{-1}\mathbf{S}_{k}\left(  x_{1}^{p},\ldots,x_{n}%
^{p}\right)  \right)  ^{r/\left(  kp\right)  }\leq r!\left\vert G\right\vert
\left(  \frac{x_{1}^{p}+\cdots+x_{n}^{p}}{n}\right)  ^{r/p}\\
&  =r!\left\vert G\right\vert /n^{r/p}%
\end{align*}
as claimed.
\end{proof}

Note that for $k=r-1,$ $p=2$ and $G\in\mathcal{G}^{r}\left(  n\right)  $
Theorem \ref{th_Sdeg} has been proved by Friedman and Wigderson \cite{FrWi95}
by a different approach. Our proof can be used also to get the following
flexible form of Theorem \ref{th_Sdeg}.

\begin{theorem}
\label{th_Sdeg1}Let $r>k\geq2$ and $p\geq r/k.$ If $G\in\mathcal{W}^{r}\left(
n\right)  $ and $\Delta_{k}=\Delta_{k}\left(  G\right)  ,$ then
\[
\frac{r!\left\vert G\right\vert }{n^{r/p}}\leq\lambda^{\left(  p\right)
}\left(  G\right)  \leq\frac{r!\left\vert G\right\vert }{n^{r/p}}\left(
\frac{\left(  n\right)  _{k}\Delta_{k}}{\left(  r\right)  _{k}\left\vert
G\right\vert }\right)  ^{r/\left(  kp\right)  }%
\]

\end{theorem}

In turn, Theorem \ref{th_Sdeg1} can be used to estimate the largest eigenvalue
of random $r$-graphs, see, e.g., Section \ref{Rands}.

\subsection{\label{Lins}$k$-linear graphs and Steiner systems}

If $k\geq1,$ an $r$-graph is called $k$\emph{-linear} if every two edges share
at most $k$ vertices; for short, $1$-linear graphs are called \emph{linear}.
Clearly all $2$-graphs are linear, so the concept makes sense only for
hypergraphs. In fact, linearity is related to Steiner systems; recall that a
\emph{Steiner }$\left(  k,r,n\right)  $\emph{-system} is a graph in
$\mathcal{G}^{r}\left(  n\right)  $ such that every set of $k$ vertices is
contained in exactly one edge.

\begin{theorem}
\label{th_lin}Let $1\leq k\leq r-2$ and $G\in\mathcal{G}^{r}\left(  n\right)
.$ If $G$ is $k$-linear, then%
\begin{equation}
\lambda^{\left(  r/\left(  k+1\right)  \right)  }\left(  G\right)  \leq
r!\binom{r}{k+1}^{-1}\binom{n}{k+1}/n^{k+1}. \label{lin_b}%
\end{equation}
Equality holds if and only if $G$ is a Steiner $\left(  k+1,r,n\right)  $-system.
\end{theorem}

\begin{proof}
Let $\left[  x_{i}\right]  \in\mathbb{S}_{r/\left(  k+1\right)  ,+}^{n-1}$ be
an eigenvector to $\lambda^{\left(  r/\left(  k+1\right)  \right)  }\left(
G\right)  .$ If $\left\{  i_{1},\ldots i_{r}\right\}  \in E\left(  G\right)
,$ by the AM-GM inequality we have
\[
x_{i_{1}}\cdots x_{i_{r}}\leq\mathbf{S}_{k+1}\left(  \left(  x_{i_{1}%
}^{r/\left(  k+1\right)  },\ldots,x_{i_{r}}^{r/\left(  k+1\right)  }\right)
\right)  /\binom{r}{k+1};
\]
hence
\[
\sum_{\left\{  i_{1},\ldots,i_{r}\right\}  \in E\left(  G\right)  }x_{i_{1}%
}\cdots x_{i_{r}}\leq\binom{r}{k+1}^{-1}\sum_{\left\{  i_{1},\ldots
,i_{r}\right\}  \in E\left(  G\right)  }\mathbf{S}_{k+1}\left(  \left(
x_{i_{1}}^{r/\left(  k+1\right)  },\ldots,x_{i_{r}}^{r/\left(  k+1\right)
}\right)  \right)
\]
In the right side we have a sum of monomials of the type $x_{i_{1}}^{r/\left(
k+1\right)  }\cdots x_{i_{k+1}}^{r/\left(  k+1\right)  }$ where $\left\{
i_{1},\ldots,i_{k+1}\right\}  $ is a $\left(  k+1\right)  $-subset of some
edge of $G.$ Since every $\left(  k+1\right)  $-subset $\left\{  i_{1}%
,\ldots,i_{k+1}\right\}  $ belongs to at most one edge, Maclaurin's inequality
implies that
\begin{align*}
\lambda^{\left(  r/\left(  k+1\right)  \right)  }\left(  G\right)   &  \leq
r!\mathbf{S}_{k+1}\left(  \left(  x_{1}^{r/\left(  k+1\right)  },\ldots
,x_{n}^{r/\left(  k+1\right)  }\right)  \right) \\
&  \leq r!\binom{r}{k+1}^{-1}\binom{n}{k+1}\left(  \frac{1}{n}\mathbf{S}%
_{1}\left(  \left(  x_{1}^{r/\left(  k+1\right)  },\ldots,x_{n}^{r/\left(
k+1\right)  }\right)  \right)  \right)  ^{k+1}\\
&  =r!\binom{r}{k+1}^{-1}\binom{n}{k+1}/n^{k+1}.
\end{align*}
If equality holds in (\ref{lin_b}), then the condition for equality in
Maclaurin's inequality implies that $\left[  x_{i}\right]  =n^{-\left(
k+1\right)  /r}\mathbf{j}_{n};$ thus, every $\left(  k+1\right)  $-set of
$V\left(  G\right)  $ is contained in some edge, implying that $G$ is a
Steiner $\left(  k+1,r,n\right)  $-system. Conversely, if $G$ is a Steiner
$\left(  k+1,r,n\right)  $-system, taking $\left[  x_{i}\right]  =n^{-\left(
k+1\right)  /r}\mathbf{j}_{n},$ we see that
\[
P_{G}\left(  \left[  x_{i}\right]  \right)  =r!\binom{n}{k+1}/n^{k+1}%
\]
and so equality holds in (\ref{lin_b}).
\end{proof}

Since $\lambda^{\left(  p\right)  }\left(  G\right)  $ is increasing in $p,$
we obtain the following more applicable bound.

\begin{proposition}
Let $1\leq k\leq r-2,$ $1\leq p\leq r/\left(  k+1\right)  ,$ and
$G\in\mathcal{G}^{r}\left(  n\right)  .$\ If $G$ is $k$-linear, then%
\[
\lambda^{\left(  p\right)  }\left(  G\right)  \leq r!/\left(  r\right)
_{k+1}.
\]

\end{proposition}

Let us make also an easy observation.

\begin{proposition}
Let $1\leq k\leq r-2$ and $G\in\mathcal{G}^{r}\left(  n\right)  .$ If $G\ $is
$k$-linear and $1\leq p<r/\left(  k+1\right)  $, then the vector
$n^{-1/p}\mathbf{j}_{n}$ is not an eigenvector to $\lambda^{\left(  p\right)
}\left(  G\right)  .$ If $G$ is a Steiner $\left(  k+1,r,n\right)  $-system,
then $n^{-\left(  k+1\right)  /r}\mathbf{j}_{n}$ is an eigenvector to
$\lambda^{\left(  r/\left(  k+1\right)  \right)  }\left(  G\right)  $.
\end{proposition}

\begin{proof}
Every $\left(  k+1\right)  $-set of vertices belongs to at most one edge;
therefore,
\[
\binom{n}{k+1}\geq\left\vert G\right\vert \binom{r}{k+1}.
\]
Hence, if $p<r/\left(  k+1\right)  ,$ then
\[
\frac{r!\left\vert G\right\vert }{n^{r/p}}\leq\frac{r!\left(  n\right)
_{k+1}}{n^{r/p}\left(  r\right)  _{k+1}}=o\left(  p\right)  .
\]
The second statement is immediate from Theorem \ref{th_lin}.
\end{proof}

\begin{question}
Is is true that if $G\in\mathcal{G}^{r}\left(  n\right)  $ and $G$ is
$k$-linear and $n^{-\left(  k+1\right)  /r}\mathbf{j}_{n}$ is an eigenvector
to $\lambda^{\left(  r/\left(  k+1\right)  \right)  }\left(  G\right)  $, then
$G$ is a \ Steiner $\left(  k+1,r,n\right)  $-system?
\end{question}

\subsection{Bounds on the entries of a $\lambda^{\left(  p\right)  }$
eigenvector}

Information about the entries of eigenvector can be quite useful in
calculations. Papendieck and Recht \cite{PaRe00} showed that if $G\in
\mathcal{G}_{2}\left(  n\right)  $ and $\left[  x_{i}\right]  \in
\mathbb{S}_{2}^{n-1}$ is an eigenvector to $\lambda\left(  G\right)  ,$ then
$x_{k}^{2}\leq1/2$ for any $k\in V\left(  G\right)  .$ This bound easily
extends to $r$-graphs.

\begin{proposition}
Let $G\in\mathcal{G}^{r}\left(  n\right)  $ and $\left[  x_{i}\right]
\in\mathbb{S}_{p}^{n-1}.$ If $\left[  x_{i}\right]  $ is an eigenvector to
$\lambda^{\left(  p\right)  }\left(  G\right)  ,$ then $\left\vert
x_{k}\right\vert ^{p}\leq1/r$ for any $k\in V\left(  G\right)  .$ If $G$ is a
$\beta$-star, then equality holds.
\end{proposition}

In fact, the result easily generalizes to star-like subgraphs. This seems new
even for $2$-graphs.

\begin{proposition}
If $G\in\mathcal{G}^{r}$ and $U\subset V\left(  G\right)  $ is such that
$\left\vert e\cap U\right\vert \leq1$ for every $e\in E\left(  G\right)  ,$
then
\[
\sum_{k\in U}\left\vert x_{k}\right\vert ^{p}\leq1/r.
\]
If $G$ is a star-like graph of the type $K_{s}\vee tK_{r-1}^{r-1},$ then
equality holds above.
\end{proposition}

It would be interesting to determine all cases of equality in the above two
propositions.\medskip

A useful result in spectral extremal theory for $2$-graphs is the following
bound from \cite{Nik09}:

\emph{Let }$G$\emph{ be an }$2$\emph{-graph with minimum degree }$\delta
,$\emph{ and let }$\mathbf{x}=\left(  x_{1},\ldots,x_{n}\right)  $\emph{ be a
nonnegative eigenvector to }$\lambda\left(  G\right)  $\emph{ with
}$\left\vert \mathbf{x}\right\vert _{2}=1.$\emph{ If }$x=\min\left\{
x_{1},\ldots,x_{n}\right\}  ,$\emph{ then}%
\begin{equation}
x^{2}\left(  \lambda\left(  G\right)  ^{2}+\delta n-\delta^{2}\right)
\leq\delta\label{minx}%
\end{equation}

The bound (\ref{minx}) is exact for many graphs, and it has been crucial in
proving upper bounds on $\lambda\left(  G\right)  $ by induction on the number
of vertices of $G.$ Similar bounds for hypergraphs are useful as well. Below
we state and prove such a result; despite its awkward form, for $r=p=2$ it
yields precisely (\ref{minx}).

\begin{theorem}
\label{le1}Let $1\leq p\leq r,$ $G\in\mathcal{G}^{r}\left(  n\right)  ,$
$\delta\left(  G\right)  =$ $\delta,$ $\lambda^{\left(  p\right)  }\left(
G\right)  =\lambda,$ and $\left[  x_{i}\right]  \in\mathbb{S}_{p}^{n-1}.$ If
$\left[  x_{i}\right]  $ is an eigenvector to $\lambda^{\left(  p\right)
}\left(  G\right)  ,$ then the value $\sigma:=\min\left\{  \left\vert
x_{1}\right\vert ^{p},\ldots,\left\vert x_{n}\right\vert ^{p}\right\}  $
satisfies%
\begin{equation}
\left(  \left(  \frac{\lambda n^{r/p-1}}{\left(  r-1\right)  !}\right)
^{p}-\delta^{p}\right)  \sigma^{r-1}\leq\binom{n-1}{r-1}\delta^{p-1}\left(
\frac{\left(  1-\sigma\right)  ^{r-1}}{\left(  n-1\right)  ^{r-1}}%
-\sigma^{r-1}\right)  . \label{minxr}%
\end{equation}

\end{theorem}

\begin{proof}
Set for short $V=V\left(  G\right)  $ and let $k\in V$ be a vertex of degree
$\delta$. Since $\left[  \left\vert x_{i}\right\vert \right]  $ is also an
eigenvector to $\lambda^{\left(  p\right)  }\left(  G\right)  $ we can assume
that $\left[  x_{i}\right]  \geq0.$ The eigenequation for $\lambda^{\left(
p\right)  }\left(  G\right)  $ and the vertex $k$ implies that%
\[
\lambda\sigma^{1-1/p}\leq\lambda x_{k}^{p-1}=\left(  r-1\right)
!\sum_{\left\{  k,i_{1},\ldots,i_{r-1}\right\}  \in E\left(  G\right)
}x_{i_{1}}\ldots x_{i_{r-1}}.
\]
Now, dividing by $\left(  r-1\right)  !$ and applying the PM inequality to the
right side, we find that
\begin{equation}
\left(  \frac{\lambda\sigma^{1-1/p}}{\left(  r-1\right)  !}\right)  ^{p}%
\leq\delta^{p-1}\sum_{\left\{  k,i_{1},\ldots,i_{r-1}\right\}  \in E\left(
G\right)  }x_{i_{1}}^{p}\ldots x_{i_{r-1}}^{p}. \label{in4}%
\end{equation}
Our next goal is to bound the quantity $A=\sum_{\left\{  k,i_{1}%
,\ldots,i_{r-1}\right\}  \in E\left(  G\right)  }x_{i_{1}}^{p}\ldots
x_{i_{r-1}}^{p}$ from above$.$ First, let $X_{k}=\left(  V\backslash\left\{
v_{k}\right\}  \right)  ^{\left(  r-1\right)  }$ and note that
\begin{align}
&  A=\sum_{\left\{  k,i_{1},\ldots,i_{r-1}\right\}  \in E\left(  G\right)
}x_{i_{1}}^{p}\ldots x_{i_{r-1}}^{p}\nonumber\\
&  =\sum_{\left\{  i_{1},\ldots,i_{r-1}\right\}  \in X_{k}}x_{i_{1}}^{p}\ldots
x_{i_{r-1}}^{p}-\left(  \sum_{\left\{  i_{1},\ldots,i_{r-1}\right\}  \in
X_{k}\text{ and }\left\{  k,i_{1},\ldots,i_{r-1}\right\}  \notin E\left(
G\right)  }x_{i_{1}}^{p}\ldots x_{i_{r-1}}^{p}\right) \nonumber\\
&  \leq\sum_{\left\{  i_{1},\ldots,i_{r-1}\right\}  \in X_{k}}x_{i_{1}}%
^{p}\ldots x_{i_{r-1}}^{p}-\left(  \sum_{\left\{  i_{1},\ldots,i_{r-1}%
\right\}  \in X_{k}\text{ and }\left\{  k,i_{1},\ldots,i_{r-1}\right\}  \notin
E\left(  G\right)  }\sigma^{r-1}\right) \nonumber\\
&  =\sum_{\left\{  i_{1},\ldots,i_{r-1}\right\}  \in X_{k}}x_{i_{1}}^{p}\ldots
x_{i_{r-1}}^{p}-\left(  \binom{n-1}{r-1}-\delta\right)  \sigma^{r-1}.
\label{in7}%
\end{align}
Next, applying Maclaurin's inequality for the $\left(  r-1\right)  $'th
symmetric function of the variables $x_{i}^{p},$ $i\in V\backslash\left\{
v_{k}\right\}  $, we find that%
\[
\frac{1}{\binom{n-1}{r-1}}\sum_{\left\{  i_{1},\ldots,i_{r-1}\right\}  \in
X_{k}}x_{i_{1}}^{p}\ldots x_{i_{r-1}}^{p}\leq\left(  \frac{1}{n-1}\sum_{i\in
V\backslash\left\{  k\right\}  }x_{i}^{p}\right)  ^{r-1}=\frac{1}{\left(
n-1\right)  ^{r-1}}\left(  1-\sigma\right)  ^{r-1}.
\]
Hence, replacing in (\ref{in7}), we obtain the desired bound on $A:$
\[
A=\sum_{\left\{  k,i_{1},\ldots,i_{r-1}\right\}  \in E\left(  G\right)
}x_{i_{1}}^{p}\ldots x_{i_{r-1}}^{p}\leq\binom{n-1}{r-1}\frac{\left(
1-\sigma\right)  ^{r-1}}{\left(  n-1\right)  ^{r-1}}-\left(  \binom{n-1}%
{r-1}-\delta\right)  \sigma^{r-1}.
\]
Returning back to (\ref{in4}), we see that
\[
\left(  \frac{\lambda}{\left(  r-1\right)  !}\right)  ^{p}\sigma^{p-1}%
\leq\binom{n-1}{r-1}\delta^{p-1}\left(  \frac{\left(  1-\sigma\right)  ^{r-1}%
}{\left(  n-1\right)  ^{r-1}}-\sigma^{r-1}\right)  +\delta^{p}\sigma^{r-1}.
\]
Since $p\leq r$ and $\sigma\leq1/n,$we see that%
\[
\frac{\lambda n^{r/p-1}}{\left(  r-1\right)  !}\sigma^{r/p-1}\leq\frac
{\lambda}{\left(  r-1\right)  !}\sigma^{1-1/p}.
\]
Hence,%
\begin{align*}
\left(  \frac{\lambda n^{r/p-1}}{\left(  r-1\right)  !}\right)  ^{p}%
\sigma^{r-1}  &  \leq\left(  \frac{\lambda}{\left(  r-1\right)  !}\right)
^{p}\sigma^{p-1}\\
&  \leq\binom{n-1}{r-1}\delta^{p-1}\left(  \frac{\left(  1-\sigma\right)
^{r-1}}{\left(  n-1\right)  ^{r-1}}-\sigma^{r-1}\right)  +\delta^{p}%
\sigma^{r-1},
\end{align*}
and so,%
\[
\left(  \left(  \frac{\lambda n^{r/p-1}}{\left(  r-1\right)  !}\right)
^{p}-\delta^{p}\right)  \sigma^{r-1}\leq\binom{n-1}{r-1}\delta^{p-1}\left(
\frac{\left(  1-\sigma\right)  ^{r-1}}{\left(  n-1\right)  ^{r-1}}%
-\sigma^{r-1}\right)  ,
\]
completing the proof of Theorem \ref{le1}.
\end{proof}

A weaker but handier form of Theorem \ref{le1} can be obtained by first
proving that%
\[
\left(  \frac{\lambda n^{r/p-1}}{\left(  r-1\right)  !}\right)  ^{p}%
-\delta^{p}\geq\left(  \frac{\lambda n^{r/p-1}}{\left(  r-1\right)  !}%
-\delta\right)  p\delta^{p-1}%
\]
using Bernoulli's inequality, and then rearranging (\ref{minxr}) to get the
following corollary.

\begin{corollary}
Under the assumptions of Theorem \ref{le1} we have%
\[
\frac{\lambda n^{r/p-1}}{\left(  r-1\right)  !}-\delta\leq\binom{n-1}%
{r-1}\left(  \frac{\left(  1/\sigma-1\right)  ^{r-1}}{\left(  n-1\right)
^{r-1}}-1\right)  .
\]

\end{corollary}

Finally using elements of the proof of Theorem \ref{le1}, we obtain the
following simple bounds.

\begin{proposition}
Let $G\in\mathcal{G}^{r}\left(  n\right)  $ and $\left[  x_{i}\right]
\in\mathbb{S}_{p}^{n-1}.$ If $\left[  x_{i}\right]  $ is an eigenvector to
$\lambda^{\left(  p\right)  }\left(  G\right)  ,$ then for every $k\in
V\left(  G\right)  ,$%
\[
\left\vert x_{k}\right\vert ^{p}\leq\left(  r-1\right)  !d\left(  k\right)
/\left(  \lambda^{\left(  p\right)  }\left(  G\right)  \right)  ^{p/\left(
p-1\right)  }.
\]

\end{proposition}

\begin{proof}
Since $\left[  \left\vert x_{i}\right\vert \right]  $ is also an eigenvector
to $\lambda^{\left(  p\right)  }\left(  G\right)  $ we can assume that
$\left[  x_{i}\right]  \geq0.$ The eigenequation for $\lambda^{\left(
p\right)  }\left(  G\right)  $ and the vertex $k$ implies that
\[
\frac{\lambda x_{k}^{p-1}}{\left(  r-1\right)  !}=\sum_{\left\{
k,i_{1},\ldots,i_{r-1}\right\}  \in E\left(  G\right)  }x_{i_{1}}\ldots
x_{i_{r-1}}.
\]
Applying the PM inequality to the right side we find that
\[
\left(  \frac{\lambda x_{k}^{p-1}}{\left(  r-1\right)  !}\right)  ^{p}\leq
d\left(  k\right)  ^{p-1}\sum_{\left\{  k,i_{1},\ldots,i_{r-1}\right\}  \in
E\left(  G\right)  }x_{i_{1}}^{p}\ldots x_{i_{r-1}}^{p}\leq d\left(  k\right)
^{p-1}\binom{n-1}{r-1}\left(  \frac{1}{n-1}\right)  ^{r-1}%
\]
and the assertion follows by simple algebra.
\end{proof}

\medskip

\section{\label{Props_m}More properties of $\lambda_{\min}^{\left(  p\right)
}$}

The study of $\lambda_{\min}^{\left(  p\right)  }$ is considerably harder than
of $\lambda^{\left(  p\right)  };$ e.g., for even $r\geq4$ we do not know
$\lambda_{\min}^{\left(  p\right)  }$ of the complete $r$-graph of order $n.$
Since $\lambda_{\min}^{\left(  p\right)  }\left(  G\right)  =-\lambda^{\left(
p\right)  }\left(  G\right)  $ if $r$ is odd and $G\in\mathcal{G}^{r}\left(
n\right)  ,$ in this section we shall assume that $r$ is even. However,
bipartite $2$-graphs show that $\lambda_{\min}^{\left(  p\right)  }\left(
G\right)  =-\lambda^{\left(  p\right)  }\left(  G\right)  $ can hold for even
$r$ as well; this interesting situation is fully investigated below, in
subsection \ref{OTs}, where symmetry of eigenvalues is explored in general and
a question of Pearson and Zhang is answered.

\subsection{Lower bounds on $\lambda_{\min}$ in terms of order and size}

The first question that one may ask about $\lambda_{\min}$ is: how small can
be $\lambda_{\min}\left(  G\right)  $ of an $r$-graph $G\ $of order $n.$ For
$2$-graphs there are several well-known bounds, like $\lambda_{\min}\left(
G\right)  \geq-\sqrt{\left\vert G\right\vert }$ and $\lambda_{\min}\left(
G\right)  >-n/2.$ The purpose of this subsection is to extend these two bounds
to $r$-graphs if $r$ is even.

\begin{theorem}
\label{th_LOB}If $r$ is even, $p\geq1,$ and $G\in\mathcal{G}^{r}\left(
n\right)  $, then%
\[
\lambda_{\min}^{p}\left(  G\right)  \geq-\left(  r!\left\vert G\right\vert
\right)  ^{1-1/p}/2^{1/p}.
\]

\end{theorem}

For the proof of this theorem we need a simple analytical bound stated as follows.

\begin{proposition}
\label{pre_pro}If $r$ is even, then%
\[
\sum_{s=0}^{r/2-1}\frac{\left(  1-x\right)  ^{2s+1}\left(  1+x\right)
^{r-2s+1}}{\left(  2s+1\right)  !\left(  r-2s-1\right)  !}\leq\frac{2^{r-1}%
}{r!}%
\]
and therefore if $0\leq a\leq1,$ then
\[
r!\sum_{s=0}^{r/2-1}\frac{a^{2s+1}\left(  1-a\right)  ^{r-2s+1}}{\left(
2s+1\right)  !\left(  r-2s-1\right)  !}\leq\frac{1}{2}%
\]

\end{proposition}

\begin{proof}
[\textbf{Proof of Theorem \ref{th_LOB}}]Let $\left[  x_{i}\right]
\in\mathbb{S}_{p}^{n-1}$ be an eigenvector to $\lambda_{\min}\left(  G\right)
,$ let $V_{1}$ be the set of vertices $v$ for which $x_{v}<0$ and let
$V_{2}=V\left(  G\right)  \backslash V_{1}$. Let $G^{\prime}$ be a subgraph of
$G\ $of order $n$ such that $e\in E\left(  G^{\prime}\right)  $ whenever
$\left\vert e\cap V_{1}\right\vert $ is odd. Note that if $\left\{
i_{1},\ldots,i_{r}\right\}  \in E\left(  G^{\prime}\right)  ,$ then $x_{i_{1}%
}\cdots x_{i_{r}}\leq0;$ conversely if $\left\{  i_{1},\ldots,i_{r}\right\}
\in E\left(  G\right)  \backslash E\left(  G^{\prime}\right)  ,$ then
$x_{i_{1}}\cdots x_{i_{r}}\geq0.$ We conclude that%
\[
\lambda_{\min}\left(  G^{\prime}\right)  \leq P_{G^{\prime}}\left(
\mathbf{x}\right)  \leq P_{G}\left(  \mathbf{x}\right)  =\lambda_{\min}\left(
G\right)  ,
\]
The PM inequality implies that
\[
\lambda_{\min}\left(  G\right)  \geq P_{G^{\prime}}\left(  \mathbf{x}\right)
\geq-r!\left\vert G^{\prime}\right\vert ^{1-1/p}\left(  \sum_{\left\{
i_{1},\ldots,i_{r}\right\}  \in E\left(  G^{\prime}\right)  }\left\vert
x_{i_{1}}\right\vert ^{p}\cdots\left\vert x_{i_{r}}\right\vert ^{p}\right)
^{1/p}.
\]
Our next purpose is find an upper bound on the value%
\[
\sum_{\left\{  i_{1},\ldots,i_{r}\right\}  \in E\left(  G^{\prime}\right)
}\left\vert x_{i_{1}}\right\vert ^{p}\cdots\left\vert x_{i_{r}}\right\vert
^{p}.
\]
Set first $\left\vert V_{1}\right\vert =k,$ let $\mathbf{y}=\left(
y_{1},\ldots,y_{k}\right)  $ and $\mathbf{z}=\left(  z_{1},\ldots
,z_{n-k}\right)  $ be the restrictions of $\left(  x_{1}^{p},\ldots,x_{n}%
^{p}\right)  $ to $V_{1}$ and to $V_{2}$. Let $G^{\prime\prime}$ be the
$r$-graph such that $V\left(  G^{\prime\prime}\right)  =V\left(  G\right)  $
and $E\left(  G^{\prime\prime}\right)  $ is the set of all $r$-subsets
$e\subset V\left(  G\right)  $ for which $\left\vert e\cap V_{1}\right\vert $
is odd. Clearly $G^{\prime}$ is a subgraph of $G^{\prime\prime}\ $and so%
\begin{align*}
\sum_{\left\{  i_{1},\ldots,i_{r}\right\}  \in E\left(  G^{\prime}\right)
}\left\vert x_{i_{1}}\right\vert ^{p}\cdots\left\vert x_{i_{r}}\right\vert
^{p}  &  \leq\sum_{\left\{  i_{1},\ldots,i_{r}\right\}  \in E\left(
G^{\prime\prime}\right)  }\left\vert x_{i_{1}}\right\vert ^{p}\cdots\left\vert
x_{i_{r}}\right\vert ^{p},\\
&  =\mathbf{S}_{1}\left(  \mathbf{y}\right)  \mathbf{S}_{r-1}\left(
\mathbf{z}\right)  +\mathbf{S}_{3}\left(  \mathbf{y}\right)  \mathbf{S}%
_{r-3}\left(  \mathbf{z}\right)  +\cdots+\mathbf{S}_{r-1}\left(
\mathbf{y}\right)  \mathbf{S}_{1}\left(  \mathbf{z}\right)  ,
\end{align*}
where $\mathbf{S}_{k}\left(  \mathbf{w}\right)  $ is the $k$'th symmetric
function of the entries of a vector $\mathbf{w}.$ Maclaurin's inequality
implies that that for every $k=1,3,\ldots,r-1,$
\[
\mathbf{S}_{k}\left(  \mathbf{y}\right)  \mathbf{S}_{r-k}\left(
\mathbf{z}\right)  \leq\binom{p}{k}\left(  \frac{a}{p}\right)  ^{k}\binom
{n-p}{r-k}\left(  \frac{1-a}{n-p}\right)  ^{r-k}\leq\frac{a^{k}\left(
1-a\right)  ^{r-k}}{k!\left(  r-k\right)  !}.
\]
where $a=\sum_{i\in V_{1}}y_{i}$ and $\sum_{i\in V_{2}}z_{i}=1-a.$ Proposition
\ref{pre_pro} implies that
\[
\sum_{s=0}^{r/2-1}\frac{a^{2s+1}\left(  1-a\right)  ^{r-2s+1}}{\left(
2s+1\right)  !\left(  r-2s-1\right)  !}\leq\frac{1}{2r!},
\]
and so,
\begin{align*}
\lambda_{\min}\left(  G\right)   &  \geq-r!\left\vert G^{\prime}\right\vert
^{1-1/p}\left(  \sum_{\left\{  i_{1},\ldots,i_{r}\right\}  \in E\left(
G^{\prime}\right)  }\left\vert x_{i_{1}}\right\vert ^{p}\cdots\left\vert
x_{i_{r}}\right\vert ^{p}\right)  ^{1/p}\geq-r!\frac{\left\vert G\right\vert
^{1-1/p}}{\left(  2r!\right)  ^{1/p}}\\
&  =-\frac{\left(  r!\left\vert G\right\vert \right)  ^{1-1/p}}{2^{1/p}},
\end{align*}
as required.
\end{proof}

Following the proof of the previous theorem, one can show that
\[
r!\left\vert G^{\prime\prime}\right\vert \leq n^{r}/2,
\]
obtaining the following absolute bound.

\begin{theorem}
If $r$ is even and $G\in\mathcal{G}^{r}\left(  n\right)  $, then
$\lambda_{\min}^{\left(  p\right)  }\left(  G\right)  \geq-n^{r-r/p}/2.$
\end{theorem}

To show that the last theorem is essentially tight let us give an example.

\begin{proposition}
If $r$ is even, there exists $G\in\mathcal{G}^{r}\left(  n\right)  $ such
that
\[
\lambda_{\min}^{\left(  p\right)  }\left(  G\right)  \leq-n^{r-r/p}%
/2+r^{2}n^{r-1-r/p}.
\]

\end{proposition}

\begin{proof}
Let $V\left(  G\right)  :=\left[  n\right]  $ and let $V_{1}=\left[
\left\lfloor n/2\right\rfloor \right]  ,$ $V_{2}=\left[  n\right]  \backslash
V_{1}.$ Let $E\left(  G\right)  $ be the set of all $r$-subsets of $\left[
n\right]  $ that intersect $V_{1}$ in an odd number of vertices. We claim that
$G$ satisfies the requirement. Indeed, define a vector $\mathbf{x}=\left[
x_{i}\right]  \in\mathbb{S}_{r}^{n-1}$ by%
\[
x_{i}:=\left\{
\begin{array}
[c]{cc}%
-n^{-1/p} & \text{if }i\in V_{1}\\
n^{-1/p} & \text{if }i\in V_{2}%
\end{array}
\right.  .
\]
Clearly,
\[
\lambda_{\min}\left(  G\right)  \leq P_{G}\left(  \mathbf{x}\right)
=r!\sum_{\left\{  i_{1},\ldots,i_{r}\right\}  \in E\left(  G\right)  }%
x_{i_{1}}\cdots x_{i_{r}}=\frac{r!}{n^{r/p}}\sum_{\left\{  i_{1},\ldots
,i_{r}\right\}  \in E\left(  G\right)  }-1=-\frac{r!}{n^{r/p}}\left\vert
G\right\vert .
\]
To complete the proof note that
\begin{align*}
r!\left\vert G\right\vert  &  =r!\left(  \binom{\left\vert V_{1}\right\vert
}{1}\binom{\left\vert V_{2}\right\vert }{r-1}+\binom{\left\vert V_{1}%
\right\vert }{3}\binom{\left\vert V_{2}\right\vert }{r-3}+\cdots
+\binom{\left\vert V_{1}\right\vert }{r-1}\binom{\left\vert V_{2}\right\vert
}{1}\right) \\
&  \geq\left(  \frac{n}{2}-r\right)  ^{r}\left(  \frac{r!}{1!\left(
r-1\right)  !}+\frac{r!}{3!\left(  r-3\right)  !}+\cdots+\frac{r!}{\left(
r-1\right)  !1!}\right) \\
&  >\left(  \left(  \frac{n}{2}\right)  ^{r}-r^{2}\left(  \frac{n}{2}\right)
^{r-1}\right)  2^{r-1}=\frac{n^{r}}{2}-r^{2}n^{r-1}%
\end{align*}
implying the required bound.
\end{proof}

\subsection{\label{OTs}Odd transversals and symmetry of the algebraic
spectrum}

One of the best-known theorems in spectral graph theory is the following one:
\emph{If a }$2$\emph{-graph }$G$\emph{ is bipartite, then its spectrum is
symmetric with respect to }$0.$\emph{ If }$G$\emph{ is connected and }%
$\lambda\left(  G\right)  =-\lambda_{\min}\left(  G\right)  ,$\emph{ then }%
$G$\emph{ is bipartite.\medskip}

Not surprisingly there have been attempts to generalize this statement for
hypergraphs, but they seem overly algebraic and not too convincing. We offer
here another, rather natural generalization, replacing \textquotedblleft
bipartite\textquotedblright\ by \textquotedblleft having odd
transversal\textquotedblright. Note that \textquotedblleft having odd
transversal\textquotedblright\ is a monotone graph property, inherited by
subgraphs, just like subgraphs of bipartite $2$-graphs are also bipartite.

Our main interest is in $\lambda^{\left(  p\right)  }\left(  G\right)  $ \ and
$\lambda_{\min}^{\left(  p\right)  }\left(  G\right)  ,$ but at the end of the
subsection we outline a general statement about other possible eigenvalues.
Here is our first theorem.

\begin{theorem}
\label{thOT}If $G\in\mathcal{W}^{r}$ and $G$ has an odd transversal, then
$\lambda^{\left(  p\right)  }\left(  G\right)  =-\lambda_{\min}^{\left(
p\right)  }\left(  G\right)  $ for every $p\geq1$.
\end{theorem}

\begin{proof}
Let $U\subset V\left(  G\right)  $ be an odd transversal of $G.$ Let $\left[
x_{i}\right]  \in\mathbb{S}_{p,+}^{n-1}$ be an eigenvector to $\lambda
^{\left(  p\right)  }\left(  G\right)  ;$ negate $x_{i}$ whenever $i\in U,$
and write $\mathbf{y}$ for the resulting vector. Clearly $\mathbf{y}%
\in\mathbb{S}_{p}^{n-1}$ and
\[
P_{G}\left(  \mathbf{y}\right)  =-P_{G}\left(  \left[  x_{i}\right]  \right)
=-\lambda^{\left(  p\right)  }\left(  G\right)  ;
\]
hence, in view of Proposition \ref{pro_ls}, $P_{G}\left(  \mathbf{y}\right)
=\lambda_{\min}^{\left(  p\right)  }\left(  G\right)  =-\lambda^{\left(
p\right)  }\left(  G\right)  ,$ completing the proof.
\end{proof}

To elucidate the picture let us state an immediate corollary from Theorem
\ref{thOT}

\begin{corollary}
If $G$ is an $r$-partite $r$-graph, then $\lambda^{\left(  p\right)  }\left(
G\right)  =-\lambda_{\min}^{\left(  p\right)  }\left(  G\right)  $ for every
$p\geq1.$
\end{corollary}

A converse of Theorem \ref{thOT} can be proved only if $\lambda^{\left(
p\right)  }\left(  G\right)  $ has a positive eigenvector, so we give the
following statement.

\begin{theorem}
\label{th_OT1}Let $G\in\mathcal{W}^{r}.$ If $G$ is connected, and
$\lambda^{\left(  p\right)  }\left(  G\right)  =-\lambda_{\min}^{\left(
p\right)  }\left(  G\right)  $ for some $p>r-1,$ then $G$ has an odd transversal.
\end{theorem}

\begin{proof}
Let $\lambda^{\left(  p\right)  }\left(  G\right)  =-\lambda_{\min}^{\left(
p\right)  }\left(  G\right)  $ and let $\left[  x_{i}\right]  \in
\mathbb{S}_{p}^{n-1}$ be an eigenvector to $\lambda_{\min}^{\left(  p\right)
}\left(  G\right)  .$ Clearly, $\left[  \left\vert x_{i}\right\vert \right]
\in\mathbb{S}_{p,+}^{n-1}$ and%
\[
\lambda^{\left(  p\right)  }\left(  G\right)  =-\lambda_{\min}^{\left(
p\right)  }\left(  G\right)  =-P_{G}\left(  \left[  x_{i}\right]  \right)
\leq P_{G}\left(  \left[  \left\vert x_{i}\right\vert \right]  \right)
\leq\lambda^{\left(  p\right)  }\left(  G\right)  .
\]
Hence $\left[  \left\vert x_{i}\right\vert \right]  $ is a nonnegative
eigenvector to $\lambda^{\left(  p\right)  }\left(  G\right)  ,$ and by
Theorem \ref{PFa0}, $\left[  \left\vert x_{i}\right\vert \right]  >0.$ Also
\[
-x_{i_{1}}\cdots x_{i_{r}}=\left\vert x_{i_{1}}\right\vert \cdots\left\vert
x_{i_{r}}\right\vert >0
\]
for every edge $\left\{  i_{1},\ldots,i_{r}\right\}  \in E\left(  G\right)  $.
Therefore, the set of vertices with negative entries in $\left[  x_{i}\right]
$ is an odd transversal of $G,$ completing the proof.
\end{proof}

Clearly, using Theorem \ref{PFa}, the above theorem can be generalized to any
$p>1,$ but the requirement for connectivity of $G\ $should be replaced by
\textquotedblleft$k$-tightness\textquotedblright\ for an appropriate $k.$

Moreover, Theorem \ref{thOT} proves also the symmetry of \textquotedblleft
general eigenvalues\textquotedblright, i.e. the solutions of the equations
(\ref{eequ}). More precisely, the following theorem holds.

\begin{theorem}
\label{thOTg}Let $G\in\mathcal{W}^{r}\left(  n\right)  .$ If $G$ has an odd
transversal, $p>1$ and $\lambda$ is a complex number satisfying
\begin{equation}
\lambda x_{k}\left\vert x_{k}\right\vert ^{p-2}=\frac{1}{r}\frac{\partial
P_{G}\left(  \left[  x_{i}\right]  \right)  }{\partial x_{k}},\text{
\ \ \ \ }k=1,\ldots,n, \label{eq3}%
\end{equation}
for some $\left[  x_{i}\right]  \in\mathbb{C}^{n}$ with $\left\vert \left[
x_{i}\right]  \right\vert _{p}=1$, then $-\lambda$ also satisfies (\ref{eq3})
for some $\left[  y_{i}\right]  \in\mathbb{C}^{n}$ with $\left\vert \left[
y_{i}\right]  \right\vert _{p}=1$.
\end{theorem}

From Theorem \ref{th_OT1} we see that the converse of Theorem \ref{thOTg} also
holds for connected graphs. In a similar vein, Pearson and Zhang asked in
\cite{PeZh12}, Question 4.10, what conditions would guarantee that the set of
algebraic eigenvalues determined by equations (\ref{alg}) is symmetric with
respect to the $0.$ We answer this question in the following two statements.

\begin{theorem}
If $r$ is odd and $G\in\mathcal{W}^{r}\left(  n\right)  ,$ then $-\lambda
\left(  G\right)  $ never satisfies the equations%
\begin{equation}
-\lambda\left(  G\right)  x_{k}^{r-1}=\frac{1}{r}\frac{\partial P\left(
\left[  x_{i}\right]  \right)  }{\partial x_{k}},\text{ \ \ }k=1,\ldots,n,
\label{Qieq}%
\end{equation}
for a nonzero vector $\left[  x_{i}\right]  \in\mathbb{C}^{n}.$
\end{theorem}

\begin{proof}
Assume for a contradiction that $\left[  x_{i}\right]  \in\mathbb{C}^{n}$ is a
nonzero vector satisfying (\ref{Qieq}). Clearly we can assume that $\left\vert
\left[  x_{i}\right]  \right\vert _{r}=1$. As $\left[  \left\vert
x_{i}\right\vert \right]  $ is an eigenvector to $\lambda\left(  G\right)  ,$
we see that%
\[
\lambda\left(  G\right)  \left(  \left\vert x_{1}\right\vert ^{r}%
+\cdots+\left\vert x_{n}\right\vert ^{r}\right)  =\frac{1}{r}\sum\left\vert
x_{1}\frac{\partial P\left(  \left[  x_{i}\right]  \right)  }{\partial x_{k}%
}\right\vert =P\left(  \left[  \left\vert x_{i}\right\vert \right]  \right)
.
\]
This implies that for each $k\in\left[  n\right]  ,$
\[
\left\vert \sum_{\left\{  k,i_{1},\ldots,i_{r-1}\right\}  \in E\left(
G\right)  }G\left(  \left\{  k,\ldots,i_{r}\right\}  \right)  x_{k}x_{i_{1}%
}\cdots x_{i_{r-1}}\right\vert =\sum_{\left\{  k,i_{1},\ldots,i_{r-1}\right\}
\in E\left(  G\right)  }G\left(  \left\{  k,\ldots,i_{r}\right\}  \right)
\left\vert x_{k}\right\vert \left\vert x_{i_{1}}\right\vert \cdots\left\vert
x_{i_{r-1}}\right\vert
\]
and so for each $k\in\left[  n\right]  ,$ the value $\arg x_{i_{1}}\cdots
x_{i_{r-1}}$ is the same for each edge $\left\{  k,i_{1},\ldots,i_{r-1}%
\right\}  .$ Hence, (\ref{Qieq}) implies that
\[
\pi+r\arg x_{i_{1}}=\arg x_{i_{1}}\cdots x_{i_{r}}\text{ }\left(
\operatorname{mod}2\pi\right)
\]
for every edge $\left\{  i_{1},\ldots,i_{r}\right\}  .$ By symmetry we obtain
\[
r\pi+r\arg x_{i_{1}}+\cdots+r\arg x_{i_{r}}=r\arg x_{i_{1}}\cdots x_{i_{r}%
}\text{ }\left(  \operatorname{mod}2\pi\right)  =r\arg x_{i_{1}}+\cdots+r\arg
x_{i_{r}}\text{ }\left(  \operatorname{mod}2\pi\right)  ,
\]
and so $r\pi=0$ $\left(  \operatorname{mod}2\pi\right)  ,$ a contradiction, as
$r$ is odd.
\end{proof}

So the algebraic spectrum of an $r$-graph can be symmetric only for even $r.$
The following proposition gives a sufficient condition for symmetry if $r$ is even.

\begin{proposition}
Let $r$ be even and $G\in\mathcal{W}^{r}\left(  n\right)  .$ If $G$ has an odd
transversal and $\lambda$ satisfies the equations%
\[
\lambda x_{k}^{r-1}=\frac{1}{r}\frac{\partial P\left(  \left[  x_{i}\right]
\right)  }{\partial x_{k}},\ \ k=1,\ldots,n,
\]
for some nonzero vector $\left[  x_{i}\right]  \in\mathbb{C}^{n},$ then
$-\lambda$ satisfies the same equations for some $\left[  x_{i}\right]
\in\mathbb{C}^{n}.$
\end{proposition}

Finally, from Theorem \ref{th_OT1} we already know that if $G$ is connected,
and $-\lambda\left(  G\right)  $ satisfies the equations (\ref{Qieq}), then
$G\ $has an odd transversal. This completely answers the question of Pearson
and Zhang.\medskip

\section{\label{Exts}Spectral extremal hypergraph theory}

The statements of Proposition and Theorem discuss problems of the following
type: \emph{how large can be }$\lambda^{\left(  p\right)  }\left(  G\right)
$\emph{ of an }$r$\emph{-graph }$G\ $\emph{with some particular property.}
Such problems belong to extremal graph theory and have been extensively
studied for $2$-graphs. Past experience shows that the variations of these
problems are practically infinite even for $2$-graphs. Since extremal problems
for hypergraphs are overwhelmingly diverse and hard, we shall base the study
of the extrema of $\lambda^{\left(  p\right)  }$ and $\lambda_{\min}^{\left(
p\right)  }$ on the following two principles: first, we shall focus on
hereditary properties of graphs; second, we shall seek asymptotic solutions
foremost and shall deduce exact ones only afterwards, whenever possible.

\subsection{Hereditary properties of hypergraphs}

A \emph{property} of graphs is a family of $r$-graphs closed under
isomorphisms. A property is called \emph{monotone}\textbf{ }if it is closed
under taking subgraphs, and \emph{hereditary},\textbf{ }if it is closed under
taking induced subgraphs. For example, given a set of $r$-graphs
$\mathcal{F},$ the family of all $r$-graphs that do not contain any
$F\in\mathcal{F}$ as a subgraph is a monotone property, denoted by $Mon\left(
\mathcal{F}\right)  .$ Likewise, the family of all $r$-graphs that do not
contain any $F\in\mathcal{F}$ as an induced subgraph is a hereditary property,
denoted as $Her\left(  \mathcal{F}\right)  .$ When $\mathcal{F}$ consists of a
single graph $F,$ we shall write $Mon\left(  F\right)  $\ and $Her\left(
F\right)  $ instead of $Mon\left(  \left\{  F\right\}  \right)  $ and
$Her\left(  \left\{  F\right\}  \right)  .$ Given a property $\mathcal{P},$ we
write $\mathcal{P}_{n}$ for the set of all graphs in $\mathcal{P}$ of order
$n$.

The typical extremal hypergraph problem is the following one: \emph{Given a
hereditary property }$\mathcal{P}$\emph{ of }$r$\emph{-graphs, find}
\begin{equation}
ex\left(  \mathcal{P},n\right)  :=\max_{G\in\mathcal{P}_{n}}\left\vert
G\right\vert . \label{exdef}%
\end{equation}
If $r=2$ and $\mathcal{P}$ is a monotone property, asymptotic solutions are
given by the theorem of Erd\H{o}s and Stone; for a general hereditary property
$\mathcal{P}$ an asymptotic solution was given in \cite{NikC}. For $r\geq3$
the problem has turned out to be generally hard and is solved only for very
few properties $\mathcal{P};$ see \cite{Kee11} for an up-to-date discussion.

An easier asymptotic version of the same problem arises from the following
fact, established by Katona, Nemetz and Simonovits \cite{KNS64}:\emph{ If }%
$P$\emph{ is a hereditary property of }$r$\emph{-graphs, then the sequence }%
\[
\left\{  ex\left(  \mathcal{P},n\right)  /\binom{n}{r}\right\}  _{n=1}%
^{\infty}%
\]
\emph{is nonincreasing and so the limit }%
\[
\pi\left(  \mathcal{P}\right)  :=\lim_{n\rightarrow\infty}ex\left(
\mathcal{P},n\right)  /\binom{n}{r}%
\]
\emph{always exists.\medskip}

One of the most appealing features of $\lambda^{\left(  p\right)  }$ is that
under the same umbrella it covers three graph parameters, all important in
extrema problems - the graph Lagrangian, the largest eigenvalue and the number
of edges. So, in analogy to (\ref{exdef}), given a hereditary property
$\mathcal{P}$ of $r$-graphs, we set
\[
\lambda^{\left(  p\right)  }\left(  \mathcal{P},n\right)  :=\max
_{G\in\mathcal{P}_{n}}\lambda^{\left(  p\right)  }\left(  G\right)  .
\]

Let us begin with a theorem about $\lambda^{\left(  p\right)  }\left(
G\right)  $, which is similar to the above mentioned result of Katona, Nemetz
and Simonovits.

\begin{theorem}
\label{th1}Let $p\geq1.$ If $\mathcal{P}$ is a hereditary property of
$r$-graphs, then the limit
\begin{equation}
\lambda^{\left(  p\right)  }\left(  \mathcal{P}\right)  :=\lim_{n\rightarrow
\infty}\lambda^{\left(  p\right)  }\left(  \mathcal{P},n\right)  n^{r/p-r}
\label{exlima}%
\end{equation}
exists. If $p=1,$ then $\lambda^{\left(  1\right)  }\left(  \mathcal{P}%
,n\right)  $ is nondecreasing, and so
\begin{equation}
\lambda^{\left(  1\right)  }\left(  \mathcal{P},n\right)  \leq\lambda^{\left(
1\right)  }\left(  \mathcal{P}\right)  . \label{bnd1}%
\end{equation}
If $p>1,$ then $\lambda^{\left(  p\right)  }\left(  \mathcal{P}\right)  $
satisfies%
\begin{equation}
\lambda^{\left(  p\right)  }\left(  \mathcal{P}\right)  \leq\lambda^{\left(
p\right)  }\left(  \mathcal{P},n\right)  n^{r/p}/\left(  n\right)  _{r}.
\label{bndsa}%
\end{equation}

\end{theorem}

For a proof of Theorem \ref{th1} we refer the reader to \cite{NikA}; it is
very similar to the proof of Theorem \ref{th2} below. Before exploring some of
the consequences of Theorem \ref{th1}, we shall extend the above setup to
$\lambda_{\min}^{\left(  p\right)  }$ as well. Thus, if $\mathcal{P}$ is a
hereditary property of $r$-graphs, we define
\[
\lambda_{\min}^{\left(  p\right)  }\left(  \mathcal{P},n\right)  :=\min
_{G\in\mathcal{P}_{n}}\lambda_{\min}^{\left(  p\right)  }\left(  G\right)  .
\]
and prove the following statement:

\begin{theorem}
\label{th2}Let $p\geq1.$ If $\mathcal{P}$ is a hereditary property of
$r$-graphs, then the limit
\begin{equation}
\lambda_{\min}^{\left(  p\right)  }\left(  \mathcal{P}\right)  :=\lim
_{n\rightarrow\infty}\lambda_{\min}^{\left(  p\right)  }\left(  \mathcal{P}%
,n\right)  n^{r/p-r} \label{exlimlp}%
\end{equation}
exists. If $p=1,$ then $\lambda^{\left(  1\right)  }\left(  \mathcal{P}%
,n\right)  $ is nonincreasing, and so
\begin{equation}
\lambda_{\min}^{\left(  1\right)  }\left(  \mathcal{P}\right)  \leq
\lambda_{\min}^{\left(  1\right)  }\left(  \mathcal{P},n\right)  .
\label{bnd2}%
\end{equation}
If $p>1,$ then $\lambda^{\left(  p\right)  }\left(  \mathcal{P}\right)  $
satisfies%
\begin{equation}
\lambda_{\min}^{\left(  p\right)  }\left(  \mathcal{P}\right)  \geq
\lambda_{\min}^{\left(  p\right)  }\left(  \mathcal{P},n\right)
n^{r/p}/\left(  n\right)  _{r}. \label{bndlp}%
\end{equation}

\end{theorem}

\begin{proof}
Set for short $\lambda_{n}^{\left(  p\right)  }=\lambda_{\min}^{\left(
p\right)  }\left(  \mathcal{P},n\right)  .$ Let $G\in\mathcal{P}_{n}$ be such
that $\lambda_{n}^{\left(  p\right)  }=\lambda_{\min}^{\left(  p\right)
}\left(  G\right)  $ and let $\mathbf{x}=\left[  x_{i}\right]  \in
\mathbb{S}_{p}^{n-1}$ be an eigenvector to $\lambda_{\min}^{\left(  p\right)
}\left(  G\right)  .$ If $p=1,$ we obviously have $\lambda_{n}^{\left(
1\right)  }\leq\lambda_{n-1}^{\left(  1\right)  }.$ and in view of
\[
\lambda_{n}^{\left(  1\right)  }=P_{G}\left(  \mathbf{x}\right)  \geq
-r!\sum_{1\leq i_{1}<\cdots<i_{r}\leq n}x_{i_{1}}\ldots x_{i_{r}}>-\left(
x_{1}+\cdots+x_{n}\right)  ^{r}=-1,
\]
the sequence $\left\{  \lambda_{n}^{\left(  1\right)  }\right\}
_{n=1}^{\infty}$ is converging to some $\lambda$. We have%
\[
\lambda=\lim_{n\rightarrow\infty}\lambda_{n}^{\left(  1\right)  }%
n^{r-r}=\lambda_{\min}^{\left(  1\right)  }\left(  \mathcal{P}\right)  ,
\]
proving (\ref{bnd2}).

Suppose now that $p>1.$ Since $\left\vert \mathbf{x}\right\vert _{p}=1,$ there
is a vertex $k\in V\left(  G\right)  $ such that $\left\vert x_{k}\right\vert
^{p}\leq1/n.$ Write $G-k$ for the $r$-graph obtained from $G$ by omitting the
vertex $k,$ and let $\mathbf{x}^{\prime}$ be the $\left(  n-1\right)  $-vector
obtained from $\mathbf{x}$ by omitting the entry $x_{k}.$ Now the
eigenequation for $\lambda_{\min}^{\left(  p\right)  }\left(  G\right)  $ and
the vertex $k$ implies that%
\begin{align*}
P_{G-k}\left(  \mathbf{x}^{\prime}\right)   &  =P_{G}\left(  \mathbf{x}%
\right)  -r!x_{k}\sum_{\left\{  k,i_{1},\ldots,i_{r-1}\right\}  \in E\left(
G\right)  }x_{i_{1}}\ldots x_{i_{r-1}}=\\
&  =\lambda_{\min}^{\left(  p\right)  }\left(  G\right)  -rx_{k}\left(
\lambda_{\min}^{\left(  p\right)  }\left(  G\right)  x_{k}\left\vert
x_{k}\right\vert ^{p-2}\right)  =\lambda_{n}^{\left(  p\right)  }\left(
1-r\left\vert x_{k}\right\vert ^{p}\right)
\end{align*}
Since $\mathcal{P}$ is a hereditary property, $G-k\in\mathcal{P}_{n-1},$ and
therefore,%
\[
P_{G-k}\left(  \mathbf{x}^{\prime}\right)  \geq\lambda_{\min}^{\left(
p\right)  }\left(  G-k\right)  \left\vert \mathbf{x}^{\prime}\right\vert
_{p}^{r}=\lambda_{\min}^{\left(  p\right)  }\left(  G-k\right)  \left(
1-\left\vert x_{k}\right\vert ^{p}\right)  ^{r/p}\geq\lambda_{n-1}^{\left(
p\right)  }\left(  1-\left\vert x_{k}\right\vert ^{p}\right)  ^{r/p}.
\]
Thus, we obtain%
\begin{equation}
\lambda_{n}^{\left(  p\right)  }\geq\lambda_{n-1}^{\left(  p\right)  }%
\frac{\left(  1-\left\vert x_{k}\right\vert ^{p}\right)  ^{r/p}}{\left(
1-r\left\vert x_{k}\right\vert ^{p}\right)  }. \label{in3}%
\end{equation}
Note that the function
\[
f\left(  y\right)  :=\frac{\left(  1-y\right)  ^{r/p}}{1-ry}%
\]
is nondecreasing in $y$ for $0\leq y\leq1/n$ and $n$ sufficiently large.
Indeed,%
\begin{align*}
\frac{d}{dy}f\left(  y\right)   &  =\frac{-\frac{r}{p}\left(  1-y\right)
^{r/p-1}\left(  1-ry\right)  +r\left(  1-y\right)  ^{r/p}}{\left(
1-ry\right)  ^{2}}\\
&  =\left(  -\frac{1}{p}\left(  1-ry\right)  +\left(  1-y\right)  \right)
\frac{r\left(  1-y\right)  ^{r/p-1}}{\left(  1-ry\right)  ^{2}}\\
&  =\left(  -\left(  \frac{1}{p}-1\right)  +\left(  \frac{r}{p}-1\right)
y\right)  \frac{r\left(  1-y\right)  ^{r/p-1}}{\left(  1-ry\right)  ^{2}}\geq0
\end{align*}
Here we use the fact that $1/p-1>0$ and that $\left(  r/p-1\right)  y$ tends
to $0$ when $n$ $\rightarrow\infty.$

Hence, in view of (\ref{in3}) and $\lambda_{n-1}^{\left(  p\right)  }<0$, we
find that for $n$ large enough,
\[
\lambda_{n}^{\left(  p\right)  }\geq\lambda_{n-1}^{\left(  p\right)  }f\left(
\left\vert x_{k}\right\vert ^{p}\right)  \geq\lambda_{n-1}^{\left(  p\right)
}f\left(  \frac{1}{n}\right)  =\lambda_{n-1}^{\left(  p\right)  }%
\frac{n\left(  1-1/n\right)  ^{r/p}}{\left(  n-r\right)  },
\]
and so,
\[
\frac{\lambda_{n}^{\left(  p\right)  }n^{r/p}}{n\left(  n-1\right)
\cdots\left(  n-r+1\right)  }\geq\frac{\lambda_{n}^{\left(  p\right)  }\left(
n-1\right)  ^{r/p}}{\left(  n-1\right)  \left(  n-2\right)  \cdots\left(
n-r\right)  }.
\]
Therefore, the sequence%
\[
\left\{  \lambda_{n}^{\left(  p\right)  }n^{r/p}/\left(  n\right)
_{r}\right\}  _{n=1}^{\infty}%
\]
is nondecreasing, and so it is converging, completing the proof of
(\ref{exlimlp}) and (\ref{bndlp}) for $p>1$.
\end{proof}

\bigskip

\subsection{Asymptotic equivalence of $\lambda^{\left(  p\right)  }\left(
\mathcal{P}\right)  $ and $\pi\left(  \mathcal{P}\right)  $}

Let $\mathcal{P}$ be a hereditary property of $r$-graphs. If $G\in
\mathcal{P}_{n}$ is such that $\left\vert G\right\vert =ex\left(
\mathcal{P},n\right)  $, then Theorems \ref{pro1} and \ref{th1} imply that%

\begin{equation}
r!ex\left(  \mathcal{P},n\right)  /n^{r/p}\leq\lambda^{\left(  p\right)
}\left(  G\right)  \leq\lambda^{\left(  p\right)  }\left(  \mathcal{P}%
,n\right)  n^{r/p-r} \label{inled}%
\end{equation}
and therefore
\[
\lambda^{\left(  p\right)  }\left(  \mathcal{P},n\right)  n^{r/p}/\left(
n\right)  _{r}\geq ex\left(  \mathcal{P},n\right)  /\binom{n}{r},
\]
Letting $n\rightarrow\infty,$ we find that
\begin{equation}
\lambda^{\left(  p\right)  }\left(  \mathcal{P}\right)  \geq\pi\left(
\mathcal{P}\right)  . \label{limgin}%
\end{equation}

In fact, there is almost always equality in this relation as stated in the
following theorem, proved in \cite{NikA}.

\begin{theorem}
\label{th3}If $\mathcal{P}$ is a hereditary property of $r$-graphs and $p>1$,
then
\begin{equation}
\lambda^{\left(  p\right)  }\left(  \mathcal{P}\right)  =\pi\left(
\mathcal{P}\right)  . \label{maineq}%
\end{equation}

\end{theorem}

Since a result of this scope is not available in the literature, some remarks
are due here. First, using (\ref{maineq}), every result about $\pi\left(
\mathcal{P}\right)  $ of a hereditary property $\mathcal{P}$ gives a result
about $\lambda^{\left(  p\right)  }\left(  \mathcal{P}\right)  $ as well, so
we readily obtain a number of asymptotic results about $\lambda^{\left(
p\right)  }$. As we shall see in Subsection \ref{Flats} in many important
cases such asymptotic results can be converted to explicit non-asymptotic
ones. But equality (\ref{maineq}) is more significant, as finding $\pi\left(
\mathcal{P}\right)  $ now can be reduced to maximization of a smooth function
subject to a smooth constraint, and so $\lambda^{\left(  p\right)  }\left(
G\right)  $ offers advantages compared to $\left\vert G\right\vert $ in
finding $\pi\left(  \mathcal{P}\right)  $.

\subsection{Forbidden blow-ups}

It is well-known (see, e.g., \cite{Kee11}, Theorem 2.2) that if $H\left(
k_{1},\ldots,k_{h}\right)  $ is a fixed blow-up of $H\in\mathcal{G}^{r}\left(
h\right)  ,$ then%
\begin{equation}
\pi\left(  Mon\left(  H\right)  \right)  =\pi\left(  Mon\left(  H\left(
k_{1},\ldots,k_{h}\right)  \right)  \right)  . \label{edblo}%
\end{equation}

Theorem \ref{th3} obviously implies a similar result for $\lambda^{\left(
p\right)  }.$

\begin{theorem}
If $p>1$ and $H\left(  k_{1},\ldots,k_{h}\right)  $ is a fixed blow-up of
$H\in\mathcal{G}^{r}\left(  h\right)  ,$ then
\[
\lambda^{\left(  p\right)  }\left(  Mon\left(  H\right)  \right)
=\lambda^{\left(  p\right)  }\left(  Mon\left(  H\left(  k_{1},\ldots
,k_{h}\right)  \right)  \right)  .
\]

\end{theorem}

As seen below, a similar statement exists for $\lambda_{\min}^{\left(
p\right)  }$ as well$.$ However, we have no statement similar to Theorem
\ref{th3} for $\lambda_{\min}^{\left(  p\right)  }$ and thus our proof of
Theorem \ref{thblo} uses the Hypergraph Removal Lemma and other fundamental
results about $r$-graphs.

\begin{theorem}
\label{thblo}If $p>1$ and $H\left(  k_{1},\ldots,k_{h}\right)  $ is a fixed
blow-up of $H\in\mathcal{G}^{r}\left(  h\right)  ,$ then%
\[
\lambda_{\min}^{\left(  p\right)  }\left(  Mon\left(  H\right)  \right)
=\lambda_{\min}^{\left(  p\right)  }\left(  Mon\left(  H\left(  k_{1}%
,\ldots,k_{h}\right)  \right)  \right)
\]

\end{theorem}

\begin{proof}
For the purposes of this proof write $k_{H}\left(  G\right)  $ for the number
of subgraphs of $G$ which are isomorphic to $H$. We start by recalling the
Hypergraph Removal Lemma, one of the most important consequences of the
Hypergraph Regularity Lemma, proved independently by Gowers \cite{Gow07} and
by Nagle, R\"{o}dl, Schacht and Skokan \cite{NRS06}, \cite{RoSk04}.\medskip

\textbf{Removal Lemma} \emph{Let }$H$\emph{ be an }$r$\emph{-graph of order
}$h$\emph{ and let }$\varepsilon>0.$\emph{ There exists }$\delta=\delta
_{H}\left(  \varepsilon\right)  >0$\emph{ such that if }$G$\emph{ is an }%
$r$\emph{-graph of order }$n,$\emph{ with }$k_{H}\left(  G\right)  <\delta
n^{h},$\emph{ then there is an }$r$\emph{-graph }$G_{0}\subset G$\emph{ such
that }$\left\vert G\right\vert \geq\left\vert G\right\vert -\varepsilon n^{r}%
$\emph{ and }$k_{H}\left(  G_{0}\right)  =0.\medskip$

In \cite{Erd64} Erd\H{o}s showed that for every $\varepsilon>0$ there exists
$\delta>0$ such that if $G$ is an $r$-graph with $\left\vert G\right\vert
\geq\varepsilon n^{r},$ then $K_{r}\left(  k,\ldots,k\right)  \subset G$ for
some $k\geq\delta\left(  \log n\right)  ^{1/\left(  r-1\right)  }.$ As noted
by R\"{o}dl and Schacht \cite{RoSc12} (also by Bollob\'{a}s, unpublished) this
result of Erd\H{o}s implies the following general assertion.\medskip

\textbf{Theorem A }\emph{Let }$H$\emph{ be an }$r$\emph{-graph of order }%
$h$\emph{ and let }$\varepsilon>0.$\emph{ There exists }$\delta=\delta
_{H}\left(  \varepsilon\right)  >0$\emph{ such that if }$G$\emph{ is an }%
$r$\emph{-graph of order }$n,$\emph{ with }$k_{H}\left(  G\right)
\geq\varepsilon n^{h},$\emph{ then }$H\left(  k,\ldots,k\right)  \subset
G$\emph{ for some }$k=\left\lceil \delta\left(  \log n\right)  ^{1/\left(
h-1\right)  }\right\rceil .\medskip$

Suppose now that $H$ is an $r$-graph of order $h,$ let $H\left(  k_{1}%
,\ldots,k_{h}\right)  $ be a fixed blow-up of $H,$ and set $k=\max\left\{
k_{1},\ldots,k_{h}\right\}  .$ Take $G\in Mon\left(  H\left(  k_{1}%
,\ldots,k_{h}\right)  \right)  _{n}$ such that
\[
\lambda_{\min}^{\left(  p\right)  }\left(  G\right)  =\lambda_{\min}^{\left(
p\right)  }\left(  Mon\left(  H\left(  k_{1},\ldots,k_{h}\right)  \right)
,n\right)
\]
For every $\varepsilon>0,$ choose $\delta=\delta_{H}\left(  \varepsilon
\right)  $ as in the Removal Lemma. Since $H\left(  k,\ldots,k\right)
\nsubseteq G,$ Theorem A implies that if $n$ is sufficiently large, then
$k_{H}\left(  G\right)  <\delta n^{h}.$ Now the Removal Lemma implies that
there is an $r$-graph $G_{0}\subset G$ such that $\left\vert G\right\vert
\geq\left\vert G\right\vert -\varepsilon n^{r}$ and $k_{H}\left(
G_{0}\right)  =0.$ Clearly, we can assume that $V\left(  G_{0}\right)
=V\left(  G\right)  .$ By Proposition \ref{pro10}, we see that
\[
\lambda_{\min}^{\left(  p\right)  }\left(  G\right)  \geq\lambda_{\min
}^{\left(  p\right)  }\left(  G_{0}\right)  -\left(  \varepsilon
r!n^{r}\right)  ^{1-1/p},
\]
and hence,
\begin{align*}
\frac{\lambda_{\min}^{\left(  p\right)  }\left(  Mon\left(  H\left(
k,\ldots,k\right)  \right)  ,n\right)  n^{r/p-1}}{\left(  n-1\right)  _{r-1}}
&  \geq\frac{\left(  \lambda_{\min}^{\left(  p\right)  }\left(  G_{0}\right)
-\left(  \varepsilon r!n^{r}\right)  ^{1-1/p}\right)  n^{r/p-1}}{\left(
n-1\right)  _{r-1}}\\
&  \geq\lambda_{\min}^{\left(  p\right)  }\left(  Mon\left(  H\right)
\right)  -o\left(  1\right)  -\frac{\left(  \varepsilon r!\right)
^{1-1/p}n^{r-r/p}n^{1-r/p}}{n^{r-1}}\\
&  =\lambda_{\min}^{\left(  p\right)  }\left(  Mon\left(  H\right)  \right)
+o\left(  1\right)  +\left(  \varepsilon r!\right)  ^{1-1/p}.
\end{align*}
Since $\varepsilon$ can be made arbitrarily small, we see that
\[
\lambda_{\min}^{\left(  p\right)  }\left(  MonH\left(  k_{1},\ldots
,k_{h}\right)  \right)  \geq\lambda_{\min}^{\left(  p\right)  }\left(
Mon\left(  H\right)  \right)  ,
\]
completing the proof of Theorem \ref{thblo}.
\end{proof}

\subsection{\label{Flats}Flat properties}

According to Theorem \ref{th3} every hereditary property $\mathcal{P}$
satisfies either $\lambda^{\left(  1\right)  }\left(  \mathcal{P}\right)
>\pi\left(  \mathcal{P}\right)  $ or \emph{ }$\lambda^{\left(  1\right)
}\left(  \mathcal{P}\right)  =\pi\left(  \mathcal{P}\right)  ;$ if the latter
case we shall call $\mathcal{P}$\emph{ flat}. Flat properties possess truly
remarkable features with respect to extremal problems, some of which are
presented below.

Let us note that, in general, $\pi\left(  \mathcal{P}\right)  $ alone is not
sufficient to estimate $ex\left(  \mathcal{P},n\right)  $ for small values of
$n$ and for arbitrary hereditary property $\mathcal{P}.$ However, flat
properties allow for tight, explicit upper bounds on $ex\left(  \mathcal{P}%
,n\right)  $ and $\lambda^{\left(  p\right)  }\left(  \mathcal{P}\right)  $.
Let us first outline a class of flat properties, whose study has been started
by Sidorenko \cite{Sid87}:\medskip

\emph{A graph property }$P$\emph{ is said to be multiplicative if }$G\in
P_{n}$\emph{ implies that }$G\left(  k_{1},\ldots,k_{n}\right)  \in P$\emph{
\ for every vector of positive integers }$k_{1},\ldots,k_{n}.$ \emph{This is
to say, a multiplicative property contains the blow-ups of all its
members.}\medskip

Multiplicative properties are quite convenient for extremal graph theory, and
they are ubiquitous as well. Indeed, following Sidorenko \cite{Sid87}, call a
graph $F$\emph{ covered} if every two vertices of $F$ are contained in an
edge. Clearly, complete $r$-graphs are covered, and for $r=2$ these are the
only covered graphs, but for $r\geq3$ there are many noncomplete ones. For
example, the Fano plane $3$-graph $F_{7}$ is a noncomplete covered graph.
Obviously, if $F$ is a covered graph, then $Mon\left(  F\right)  $ is both a
hereditary and a multiplicative property.

Below we illustrate Theorems \ref{th4} and \ref{th5} using $F_{7}$ as
forbidden graph because it is covered and $\pi\left(  Mon\left(  F_{7}\right)
\right)  $ is known. There are other graphs with these features, e.g., Keevash
in \cite{Kee11}, Sec. 14, lists graphs, like \textquotedblleft expanded
triangle\textquotedblright, \textquotedblleft$3$-book with $3$
pages\textquotedblright, \textquotedblleft$4$-book with $4$
pages\textquotedblright\ and others. Using these and similar references, the
reader may easily come up with other illustrations.

Another example of hereditary, multiplicative properties comes from vertex
colorings. Let $\mathcal{C}\left(  k\right)  $ be the family of all $r$-graphs
$G$ with $\chi\left(  G\right)  \leq k.$ Note first that $\mathcal{C}\left(
k\right)  $ is hereditary and multiplicative property. This statement is more
or less obvious, but it does not follow by forbidding covered subgraphs. The
following proposition summarizes the principal facts about $\mathcal{C}\left(
k\right)  $.

\begin{proposition}
\label{pro9} For all $k$ the class $\mathcal{C}\left(  k\right)  $ is a
hereditary and multiplicative property, and $\pi\left(  \mathcal{C}\left(
k\right)  \right)  =\left(  1-k^{-r+1}\right)  .$
\end{proposition}

The following theorem has numerous applications. It is shaped after a result
of Sidorenko \cite{Sid87}.

\begin{theorem}
\label{th4}If $\mathcal{P}$ is a hereditary, multiplicative property, then it
is flat; that is to say, $\lambda^{\left(  p\right)  }\left(  \mathcal{P}%
\right)  =\pi\left(  \mathcal{P}\right)  $ for every $p\geq1.$
\end{theorem}

To illustrate the usability of Theorem \ref{th4} note that the $3$-graph
$F_{7}$ is covered, and, as determined in \cite{FuSi05} and \cite{KeSu05},
$\pi\left(  Mon\left(  F_{7}\right)  \right)  =3/4,$ so we immediately get
that if $p\geq1,$ then
\[
\lambda^{\left(  p\right)  }\left(  Mon\left(  F_{7}\right)  \right)  =3/2.
\]
However, the great advantage of flat properties is that they allow tight upper
bounds on $\lambda^{\left(  p\right)  }\left(  G\right)  $ and $\left\vert
G\right\vert $ for every graph $G$ that belongs to a flat property, as stated below.

\begin{theorem}
\label{th5}If $\mathcal{P}$ is a flat property, and $G\in\mathcal{P}_{n},$
then
\begin{equation}
\left\vert G\right\vert \leq\pi\left(  \mathcal{P}\right)  n^{r}/r!.
\label{Seq1}%
\end{equation}
and for every $p\geq1,$%
\begin{equation}
\lambda^{\left(  p\right)  }\left(  G\right)  \leq\pi\left(  \mathcal{P}%
\right)  n^{r-r/p}. \label{upbo}%
\end{equation}
Both inequalities (\ref{Seq1}) and (\ref{upb}) are tight.
\end{theorem}

Taking again the Fano plane as an example, we obtain the following tight inequality:

\begin{corollary}
\label{cor2} If $G$ is a $3$-graph of order $n$, not containing the Fano
plane, then for all $p\geq1,$
\begin{equation}
\lambda^{\left(  p\right)  }\left(  G\right)  \leq\frac{3}{4}n^{3-3/p}.
\label{fpin}%
\end{equation}

\end{corollary}

This inequality is essentially equivalent to Corollary 3 in \cite{KLM13},
albeit it is somewhat less precise. We believe however, that Theorem \ref{th5}
shows clearly why such a result is possible at all.

With respect to chromatic number, an early result of Cvetkovi\'{c}
\cite{Cve72} states: \emph{if }$G$\emph{ is a }$2$\emph{-graph of order }%
$n$\emph{ and chromatic number }$\chi,$\emph{ then}
\[
\lambda\left(  G\right)  \leq\left(  1-1/\chi\right)  n.
\]
This bound easily generalizes for hypergraphs.

\begin{corollary}
Let $G$ be an $r$-graph of order $n$ and let $p\geq1.$ If $\chi\left(
G\right)  =k$, then
\[
\lambda^{\left(  p\right)  }\left(  G\right)  \leq\left(  1-k^{-r+1}\right)
n^{r-r/p};
\]

\end{corollary}

Furthermore, recalling that complete graphs are the only covered $2$-graphs,
it becomes clear that the bound (\ref{upbo}) is analogous to Wilf's bound
\cite{Wil86}: \emph{if }$G$\emph{ is a }$2$\emph{-graph of order }$n$\emph{
and clique number }$\omega,$\emph{ then}%
\begin{equation}
\lambda\left(  G\right)  \leq\left(  1-1/\omega\right)  n. \label{wilin}%
\end{equation}
This has been improved in \cite{Nik02}, namely: \emph{if }$G$\emph{ is a }%
$2$\emph{-graph with }$m$\emph{ edges and clique number }$\omega,$\emph{ then}%
\begin{equation}
\lambda\left(  G\right)  \leq\sqrt{2\left(  1-1/\omega\right)  m}.
\label{edin}%
\end{equation}
It turns out that the proof of (\ref{edin}) generalizes to hypergraphs, giving
the following theorem, which strengthens (\ref{upbo}) exactly as (\ref{edin})
strengthens (\ref{wilin}).

\begin{theorem}
\label{th6}If $\mathcal{P}$ is a flat property, and $G\in\mathcal{P}$, then
\begin{equation}
\lambda^{\left(  p\right)  }\left(  G\right)  \leq\pi\left(  G\right)
^{1/p}\left(  r!\left\vert G\right\vert \right)  ^{1-1/p}. \label{turin}%
\end{equation}

\end{theorem}

Let us emphasize the peculiar fact that the bound (\ref{turin}) does not
depend on the order of $G,$ but it is asymptotically tight in many cases. In
particular, for $3$-graphs with no $F_{7}$ we obtain the following tight bound:

\begin{corollary}
If $G\in\mathcal{G}^{r}$ and $G$ does not contain the Fano plane, then
\[
\lambda^{\left(  p\right)  }\left(  G\right)  \leq3\cdot2^{1-3/p}\left\vert
G\right\vert ^{1-1/p}.
\]

\end{corollary}

\medskip

There are flat properties $\mathcal{P}$ of $2$-graphs that are not
multiplicative, e.g., let $\mathcal{P}=Her\left(  C_{4}\right)  ,$ i.e.,
$\mathcal{P}$ is the class of all graphs with no induced $4$-cycle. Trivially,
all complete graphs belong to $\mathcal{P}$ and so%
\[
\lambda^{\left(  1\right)  }\left(  \mathcal{P}\right)  =\pi\left(  G\right)
=1.
\]
However, obviously $\mathcal{P}$ is not multiplicative, as $C_{4}=K_{2}\left(
2,2\right)  $ and $\pi\left(  Her\left(  C_{4}\right)  \right)  =0.$ As a
consequence of this example, we come up with the following sufficient
condition for flat properties.

\begin{theorem}
\label{th9}If $\mathcal{F}$ is a set of $r$-graphs each of which is a blow-up
of a covered graph, then $Her\left(  \mathcal{F}\right)  $ is flat.
\end{theorem}

Apparently Theorem \ref{th9} greatly extends the range of flat properties,
however further work is needed to determine the limits of its applicability.

\subsection{Intersecting families}

We are interested in the following question of classical flavor: \emph{let
}$n$\emph{ be sufficiently large and }$G\in\mathcal{G}^{r}\left(  n\right)
$\emph{ be such that every two edges share at least }$t$\emph{ vertices. How
large }$\lambda^{\left(  p\right)  }\left(  G\right)  $\emph{ can be?}

Erd\H{o}s, Ko and Rado have shown that if $G\in\mathcal{G}^{r}\left(
n\right)  $ satisfies the premise, then $\left\vert G\right\vert <\binom
{n-t}{r-t}$ unless $G=S_{t,n}^{r}.$ Moreover Erd\H{o}s, Ko and Rado a stronger
stability result: there is $c=c\left(  r,t\right)  >0$ such that if
$\left\vert G\right\vert >c\binom{n-t}{r-t-1},$ then $G\subset S_{t,n}^{r}.$
This fact is enough to prove the following result, first shown by Keevash,
Lenz and Mubayi \cite{KLM13} in a more complicated setup.

\begin{theorem}
Let $G\in\mathcal{G}^{r}\left(  n\right)  .$ If\emph{ }$n$ is sufficiently
large and every two edges of $G$ share at least $t$ vertices, then
$\lambda^{\left(  p\right)  }\left(  G\right)  <\lambda^{\left(  p\right)
}\left(  S_{t,n}^{r}\right)  ,$ unless $G=S_{t,n}^{r}.$
\end{theorem}

\begin{proof}
[Sketch of the proof]Assume that $\lambda^{\left(  p\right)  }\left(
G\right)  \geq\lambda^{\left(  p\right)  }\left(  S_{t,n}^{r}\right)  .$ By
the bounds (\ref{upb}) and (\ref{stal}) one has
\begin{align*}
r!\left\vert G\right\vert  &  \geq\left(  \lambda^{\left(  p\right)  }\left(
G\right)  \right)  ^{p/\left(  p-1\right)  }\geq\left(  \lambda^{\left(
p\right)  }\left(  S_{t,n}^{r}\right)  \right)  ^{p/\left(  p-1\right)  }\\
&  =\left(  \frac{\left(  r\right)  _{t}\left(  r-t\right)  ^{\left(
r-t\right)  /p}\left(  n-t\right)  _{r-t}}{r^{r/p}\left(  n-t\right)
^{\left(  r-t\right)  /p}}\right)  ^{p/\left(  p-1\right)  }\geq\left(
n-t\right)  ^{r-t}\\
&  =\left(  \frac{\left(  r\right)  _{t}\left(  r-t\right)  ^{\left(
r-t\right)  /p}}{r^{r/p}}\right)  ^{p/\left(  p-1\right)  }\left(  n-t\right)
^{r-t}.
\end{align*}
Hence, if $n$ is sufficiently large, then $\left\vert G\right\vert
>c\binom{n-t}{r-t-1}$ and by the stability result of Erd\H{o}s, Ko and Rado
$G\subset S_{t,n}^{r}$ and so $\lambda^{\left(  p\right)  }\left(  G\right)
\leq\lambda^{\left(  p\right)  }\left(  S_{t,n}^{r}\right)  .$ Since the
eigenvectors to $\lambda^{\left(  p\right)  }\left(  S_{t,n}^{r}\right)  $
have no zero entries, if $G\neq S_{t,n}^{r},$ then $\lambda^{\left(  p\right)
}\left(  G\right)  <\lambda^{\left(  p\right)  }\left(  S_{t,n}^{r}\right)  $.
\end{proof}

\section{\label{Rands}Random graphs}

A \emph{random }$r$\emph{-graph }$G^{r}\left(  n,p\right)  $ is an $r$-graph
of order $n$, in which any $r$-set $e\in V^{\left(  r\right)  }$ belongs to
$E\left(  G\right)  $ with probability $p,$ independently of other members of
$V^{\left(  r\right)  }$. In this definition, $p$ is not necessarily constant
and may depend on $n$.

In $G^{r}\left(  n,p\right)  $ the distribution of the set degrees is
binomial, e.g., the distribution of the the $\left(  r-1\right)  $-set degrees
is%
\[
\mathbb{P}\left(  d\left(  U\right)  =k\right)  =\binom{n-r+1}{k}p^{k}\left(
1-p\right)  ^{n-r+1-k}%
\]
where $U\in V^{\left(  r-1\right)  }.$ This fact, together with inequality
(\ref{ginl}), Theorem \ref{th_Sdeg1} and Proposition \ref{pro10}, leads to the
following estimate.

\begin{theorem}
\label{th_Rand1}If $0<p<1$ is fixed and $q>1,$ then almost surely,
\[
\lambda^{\left(  q\right)  }\left(  G^{r}\left(  n,p\right)  \right)
=r!p\binom{n}{r}/n^{r/q}=\left(  p+o\left(  1\right)  \right)  n^{r-r/q}%
\]

\end{theorem}

Fro $q=2$ Friedman and Wigderson in \cite{FrWi95} proved a stronger statement
requiring only that $p=\Omega\left(  \log n/n\right)  $. In fact, Theorem
\ref{th_rand1} can also be strengthened for variable $p,$ but we leave this
for future exploration.

Clearly, for odd $r$ we have $\lambda_{\min}^{\left(  p\right)  }\left(
G^{r}\left(  n,p\right)  \right)  =-\lambda^{\left(  q\right)  }\left(
G^{r}\left(  n,p\right)  \right)  ,$ but for even $r$ the picture is
completely different.

\begin{theorem}
\label{th_rand1}If $r$ is odd, $0<p<1$ is fixed and $q>1,$ then almost surely,%
\[
\left\vert \lambda_{\min}^{\left(  q\right)  }\left(  G^{r}\left(  n,p\right)
\right)  \right\vert \leq n^{\left(  r+1\right)  /2-r/q+o\left(  1\right)  }%
\]

\end{theorem}

For $r=2$ and $q\geq2,$ Theorem \ref{th_rand1} follows from well-known results
but all other cases require an involved new proof which will be given elsewhere.

Let us note that the parameter \textquotedblleft second largest
eigenvalue\textquotedblright, defined by Friedman and Wigderson in
\cite{FrWi95} has nothing to do with eigenvalues in the sense of the present paper.

\section{\label{Cons}Concluding remarks}

In this section we offer a final discussion of the approach taken in this
paper and outline some areas for further research. First, we showed above that
the fundamental parameters $\lambda^{\left(  p\right)  }$ and $\lambda_{\min
}^{\left(  p\right)  }$ comply well with most definitions of hypergraph and
hypermatrix eigenvalues. This fact is important, because for odd rank the
algebraic eigenvalues of Qi are different from the variational eigenvalues of
Lim, and their blending is not in sight. For hypergraphs Cooper and Dutle
\cite{CoDu11} and Pearson and Zhang \cite{PeZh12} adopted the algebraic
approach of Qi, but unfortunately did not get much further than the largest
eigenvalue. Indeed, the authors of \cite{CoDu11} have have put significant
effort in computing the mind-boggling algebraic spectra of some simple
hypergraphs, but even for complete $4$-graphs the work is still unfinished;
see, Dutle \cite{Dut11} for some spectacular facts. One gets the impression
that finding the algebraic eigenvalues of a hypergraph require involved
computations, but very few of them are of any use to hypergraphs.\medskip

On the other hand, $\lambda^{\left(  p\right)  }$ and $\lambda_{\min}^{\left(
p\right)  }$ alone give rise to a mountain of hypergraph-relevant problems. As
we saw, $\lambda^{\left(  p\right)  }$ blends seamlessly various spectral and
nonspectral parameters, thus becoming a cornerstone of a general analytic
theory of hypergraphs. So there is a good deal of work to be done here, before
the definitions settle and one could pursue the study of the \textquotedblleft
second largest\textquotedblright\ and any other eigenvalues of
hypergraphs.\medskip

There are several areas deserving intensive further exploration. First, this
is a Perron-Frobenius theory for hypergraphs, with possible extension to
cubical nonnegative hypermatrices. We state two problems motivated by Section
\ref{PFsec}:

\begin{problem}
Given $1<p<r,$ characterize all $G\in\mathcal{W}^{r}$ such that $\lambda
^{\left(  p\right)  }\left(  G\right)  $ has a unique positive eigenvector eigenvector.
\end{problem}

\begin{problem}
Given $1<p<r,$ characterize all $G\in\mathcal{W}^{r}$ such that if $\lambda>0$
and $\left[  x_{i}\right]  \in\mathbb{S}_{p,+}^{n-1}$ satisfy the
eigenequations eigenequations for $\lambda^{\left(  p\right)  }\left(
G\right)  ,$ then $\lambda=\lambda^{\left(  p\right)  }\left(  G\right)  $.
\end{problem}

\medskip

Second, $\lambda^{\left(  p\right)  }$ seems a powerful device for studying
extremal problems for hypergraphs. So far the relation is one directional:
solved extremal problems for edges are transformed into bounds for
$\lambda^{\left(  p\right)  },$ in many cases stronger than the original
results. However, it is likely that the other direction may work as well if
closer relations between $\lambda^{\left(  p\right)  }$ and the graph
structure are established.\medskip

Third, it is challenging to study the function $f_{G}\left(  p\right)
:=\lambda^{\left(  p\right)  }\left(  G\right)  $ for any fixed $G\in
\mathcal{G}^{r}.$ In particular, the following questions seem to be not too difficult.

\begin{question}
Is the function $f_{G}\left(  p\right)  $ differentiable for $p>r?$ Is
$f_{G}\left(  p\right)  $ analytic for $p>r?$
\end{question}

Also, the example of non-differentiable $f_{G}\left(  p\right)  $ given in
subsection \ref{Fsec} suggests the following question.

\begin{question}
For which graphs $G\in\mathcal{G}^{r}$ the function $f_{G}\left(  p\right)  $
is differentiable for every $p>1?$
\end{question}

\medskip

Fourth, for $2$-graphs relations of $\lambda^{\left(  2\right)  }$ and degrees
have proved to be extremely useful for applications. However, for hypergraphs
little is known in this vein at present, particularly regarding set degrees.
Of the many possible questions we choose only the following, rather
challenging, one.

\begin{question}
If $r\geq3$ and $G\in\mathcal{G}^{r}\left(  n\right)  $ is it always true that
$\lambda\left(  G\right)  \geq\left(  \frac{1}{n}\sum_{u\in V\left(  G\right)
}d^{r/\left(  r-1\right)  }\left(  u\right)  \right)  ^{1-1/r}?$
\end{question}

\medskip

Finally, we need more powerful algebraic and analytic methods to calculate and
estimate $\lambda^{\left(  p\right)  }$ and $\lambda_{\min}^{\left(  p\right)
}$. It is exasperating that $\lambda^{\left(  p\right)  }$ of the cycle
$C_{n}^{r}$ is not known for $1<p<r$, and even the following natural question
is difficult to answer.

\begin{question}
For even $r\geq4$ what is $\lambda_{\min}$ of the complete $r$-graph of order
$n?$
\end{question}

\bigskip

\textbf{Acknowledgement }The author is indebted to Xiying Yuan for her careful
reading of early versions of the manuscript.\medskip

\section{\label{basics}Notation and some basic facts}

\subsection{\label{Nors}Vectors, norms and spheres}

The entries of vectors considered in this paper are either real or complex
numbers and their type is specified if necessary. A vector with entries
$x_{1},\ldots,x_{n}$ will be denoted by $\left[  x_{i}\right]  $ and sometimes
by $\left(  x_{1},\ldots,x_{n}\right)  $ or by lower case bold letters. In
particular, $\mathbf{j}_{n}$ stands for the all ones vector of size $n.$ A
nonnegative (positive) vector is a real vector $\mathbf{x}$ with nonnegative
(positive) entries, in writing $\mathbf{x}\geq0$ ($\mathbf{x}>0$). Given a
real number $p$ such that $p\geq1,$ the $l^{p}$ norm of a complex vector
$\left[  x_{i}\right]  $ is defined as%
\[
\left\vert \left[  x_{i}\right]  \right\vert _{p}:=\sqrt[p]{\left\vert
x_{1}\right\vert ^{p}+\cdots+\left\vert x_{n}\right\vert ^{p}}.
\]
Recall that $\left\vert \mathbf{x}\right\vert _{p}+\left\vert \mathbf{y}%
\right\vert _{p}\geq\left\vert \mathbf{x}+\mathbf{y}\right\vert _{p}$ and
$\left\vert \beta\mathbf{x}\right\vert _{p}=\left\vert \beta\right\vert
\left\vert \mathbf{x}\right\vert _{p}$ for any vectors $\mathbf{x}$ and
$\mathbf{y},$ and any number $\beta$. Given $p\geq1,$ the $\left(  n-1\right)
$-dimensional unit sphere $\mathbb{S}_{p}^{n-1}$ in the $l^{p}$ norm is the
set of all real $n$-vectors $\left[  x_{i}\right]  $ satisfying%
\begin{equation}
\left\vert x_{1}\right\vert ^{p}+\cdots+\left\vert x_{n}\right\vert ^{p}=1.
\label{dsph}%
\end{equation}
Clearly $\mathbb{S}_{p}^{n-1}$ is a compact set. For convenience, we write
$\mathbb{S}_{p,+}^{n-1}$ for the set of the nonnegative vectors in
$\mathbb{S}_{p}^{n-1},$ which is compact as well.

If $p=1,$ then $\mathbb{S}_{1}^{n-1}$ is not smooth, but $\mathbb{S}_{p}%
^{n-1}$ is smooth for any $p>1$. Indeed, note that
\[
\frac{d}{dx}\left\vert x\right\vert ^{p}=\left\{
\begin{array}
[c]{cc}%
px^{p-1} & \mathtt{if}\text{ }x>0\\
-p\left(  -x\right)  ^{p-1} & \mathtt{if}\text{ }x<0
\end{array}
\right.  ,
\]
and so
\[
\frac{d}{dx}\left\vert x\right\vert ^{p}=px\left\vert x\right\vert ^{p-2}.
\]
Therefore the partial derivatives of the left side of (\ref{dsph}) exist and
are continuous.

Moreover, note that if $p>1,$ the function $f\left(  x\right)  :=x\left\vert
x\right\vert ^{p-2}$ is increasing in $x$ for all real $x$ and therefore is
bijective.\emph{\medskip}

\subsection{Classical inequalities}

Below we summarize some classical inequalities, but we shall start with a
simple version of the Lagrange multiplier theorem, widely used for proving
inequalities in general.

\begin{theorem}
\label{LMT}Let $f\left(  x_{1},\ldots,x_{n}\right)  $ and $g\left(
x_{1},\ldots,x_{n}\right)  $ be real functions of the real variables
$x_{1},\ldots,x_{n}$, and suppose that the partial derivatives of $f\left(
x_{1},\ldots,x_{n}\right)  $ and $g\left(  x_{1},\ldots,x_{n}\right)  $ exist
and are continuous. If $c$ is a fixed number and $\left(  x_{1}^{\ast}%
,\ldots,x_{n}^{\ast}\right)  $ satisfies
\[
f\left(  x_{1}^{\ast},\ldots,x_{n}^{\ast}\right)  =\max_{g\left(  x_{1}%
,\ldots,x_{n}\right)  =c}f\left(  x_{1},\ldots,x_{n}\right)  ,
\]
then there exists a number $\mu$ such that
\[
\frac{\partial f}{\partial x_{i}}\left(  x_{1}^{\ast},\ldots,x_{n}^{\ast
}\right)  =\mu\frac{\partial g}{\partial x_{i}}\left(  x_{1}^{\ast}%
,\ldots,x_{n}^{\ast}\right)
\]
for every $i=1,\ldots,n.$
\end{theorem}

Note first that Theorem \ref{LMT} gives only a necessary condition for a
maximum; second, it remains true if $\max$ is replaced by $\min.\medskip$

Classical inequalities are extremely useful for eigenvalues of hypergraphs, so
we list a few of them for ease of use; see \cite{HLP88} for more details. Let
us start with a generalization of the Cauchy-Schwarz inequality for more than
two vectors: \emph{If }$k\geq2,$\emph{ and }$\mathbf{x}_{1}:=\left[
x_{i}^{\left(  1\right)  }\right]  ,$\emph{ }$\mathbf{x}_{2}:=\left[
x_{i}^{\left(  2\right)  }\right]  ,\ldots,\mathbf{x}_{k}:=\left[
x_{i}^{\left(  k\right)  }\right]  $\emph{ are nonnegative }$n$\emph{-vectors,
then}%
\begin{equation}
\sum_{i=1}^{n}\prod_{j=1}^{k}x_{i}^{\left(  j\right)  }\leq\left\vert
\mathbf{x}_{1}\right\vert _{k}\cdots\left\vert \mathbf{x}_{k}\right\vert _{k}.
\label{CSgen}%
\end{equation}
\emph{Equality holds if and only if all vectors are collinear to one of
them.\medskip}

Another generalization of the Cauchy-Schwarz inequality is \emph{H\"{o}lder's
inequality}: \emph{Let }$x=\left[  x_{i}\right]  $\emph{ and }$y=\left[
y_{i}\right]  $\emph{ be real nonzero vectors. If the positive numbers }%
$s$\emph{ and }$t$\emph{ satisfy }$1/s+1/t=1,$ \emph{then }%
\begin{equation}
x_{1}y_{1}+\cdots+x_{n}y_{n}\leq\left\vert \mathbf{x}\right\vert
_{s}\left\vert \mathbf{y}\right\vert _{t}. \label{Holdin}%
\end{equation}
\emph{If equality holds, then }$\left(  \left\vert x_{1}\right\vert
^{s},,\ldots,\left\vert x_{n}\right\vert ^{s}\right)  $\emph{ and }$\left(
\left\vert y_{1}\right\vert ^{t},\ldots,\left\vert y_{n}\right\vert
^{t}\right)  $\emph{ are collinear.\medskip}

Next, we give the \emph{Power Mean} or the \emph{PM inequality}: \emph{Let
}$x_{1},\ldots,x_{k}$\emph{ be nonnegative real numbers. If }$0<p<q,$\emph{
then}%
\begin{equation}
\left(  \frac{x_{1}^{p}+\cdots+x_{k}^{p}}{k}\right)  ^{1/p}\leq\left(
\frac{x_{1}^{q}+\cdots+x_{k}^{q}}{k}\right)  ^{1/q} \label{PMin}%
\end{equation}
\emph{with equality holding if and only if }$x_{1}=\cdots=x_{k}.\medskip$

Let $\mathbf{x}=\left[  x_{i}\right]  $ be a real vector and $\mathbf{S}%
_{r}\left(  \mathbf{x}\right)  $ be the $r$'th\emph{ symmetric function }of
$x_{1},\ldots,x_{n}.$ In particular,
\[
\mathbf{S}_{1}\left(  \mathbf{x}\right)  :=x_{1}+\cdots+x_{n}\text{
\ \ \ \ }\mathbf{S}_{n}\left(  \mathbf{x}\right)  :=x_{1}\cdots x_{n}.
\]
Here is \emph{Maclaurin's inequality, }which\emph{ }is very useful for
hypergraphs: \emph{If }$x=\left[  x_{i}\right]  $\emph{ is a nonnegative real
vector, then}%
\begin{equation}
\frac{\mathbf{S}_{1}\left(  \mathbf{x}\right)  }{n}\geq\cdots\geq\left(
\frac{\mathbf{S}_{r}\left(  \mathbf{x}\right)  }{\binom{n}{r}}\right)
^{1/r}\geq\cdots\geq\left(  \mathbf{S}_{n}\left(  \mathbf{x}\right)  \right)
^{1/n}. \label{Maclin}%
\end{equation}
The cases of equality in (\ref{Maclin}) are somewhat tricky, so we formulate
only the case which is actually needed in the paper: If $\mathbf{x}=\left[
x_{i}\right]  $ is a nonnegative vector and
\begin{equation}
\frac{\mathbf{S}_{1}\left(  \mathbf{x}\right)  }{n}=\left(  \frac
{\mathbf{S}_{r}\left(  \mathbf{x}\right)  }{\binom{n}{r}}\right)  ^{1/r}
\label{Maclin1}%
\end{equation}
for some $1<r\leq n,$ then $x_{1}=\cdots=x_{n}.$\emph{\medskip}

The inequality
\begin{equation}
\frac{x_{1}+\cdots+x_{n}}{n}\geq\left(  x_{1}\cdots x_{n}\right)  ^{1/n}
\label{AM-GM}%
\end{equation}
is a particular case of (\ref{Maclin}), with equality if and only if
$x_{1}=\cdots=x_{n}.$ We shall refer to (\ref{AM-GM}) as the \emph{arithmetic
mean - geometric mean} or the \emph{AM-GM inequality}.\medskip

We finish this subsection with \emph{Bernoulli's inequality}: \emph{If }%
$a$\emph{ and }$x$\emph{ are real numbers satisfying }$a>1,$\emph{ }%
$x>-1$\emph{ and }$x\neq0,$\emph{ then }%
\begin{equation}
\left(  1+x\right)  ^{a}>1+ax. \label{Berin}%
\end{equation}

\subsection{\label{Pfor}Polyforms}

For reader's convenience we give here some remarks about polyforms of
$r$-graphs. First, if $G$ is an weighted $r$-graph, then $P_{G}\left(
\mathbf{x}\right)  $ is homogenous of degree $r$, that is to say,\ for any
real $s,$ $P_{G}\left(  s\mathbf{x}\right)  =s^{r}P_{G}\left(  \mathbf{x}%
\right)  ;$ also, $P_{G}\left(  \mathbf{x}\right)  $ is even for even $r$ and
odd for odd $r.$

Crucial to many calculations are the following observations used in the paper
without explicit reference.

\begin{proposition}
If $G\in\mathcal{W}^{r}\left(  n\right)  $ and $\left[  x_{i}\right]  $ is an
$n$-vector, then for each $k=1,\ldots,n$
\[
\frac{\partial P_{G}\left(  \left[  x_{i}\right]  \right)  }{\partial x_{k}%
}=r!\sum_{\left\{  k,i_{1},\ldots,i_{r-1}\right\}  \in E\left(  G\right)
}G\left(  \left\{  k,i_{1},\ldots,i_{r-1}\right\}  \right)  x_{i_{1}}\cdots
x_{i_{r-1}}.
\]
This implies also that
\[
rP_{G}\left(  \left[  x_{i}\right]  \right)  =\sum_{k=1}^{n}x_{k}%
\frac{\partial P_{G}\left(  \left[  x_{i}\right]  \right)  }{\partial x_{k}}.
\]

\end{proposition}

If $G$ is a complete $r$-graph, then $P_{G}\left(  \mathbf{x}\right)
=r!\mathbf{S}_{r}\left(  \mathbf{x}\right)  .$ Likewise, we see the following proposition.

\begin{proposition}
If $G\in\mathcal{G}^{r}$ is a complete $k$-partite graph, with vertex sets
$V_{1},\ldots,V_{k}$, then
\[
P_{G}\left(  \left[  x_{i}\right]  \right)  =r!\mathbf{S}_{r}\left(
\mathbf{y}\right)  ,
\]
where $\mathbf{y}=\left(  y_{1},\ldots,y_{k}\right)  ,$ is defined by%
\[
y_{s}:=%
{\displaystyle\sum\limits_{i\in V_{s}}}
x_{i},\text{ \ \ \ \ }s=1,\ldots,k.
\]

\end{proposition}

\subsection{A mini-glossary of hypergraphs}

The reader is referred to \cite{Ber87} for introductory material on
hypergraphs. We reiterate that in this paper \textquotedblleft
graph\textquotedblright\ stands for \textquotedblleft uniform
hypergraph\textquotedblright; thus, \textquotedblleft
ordinary\textquotedblright\ graphs are referred to as \textquotedblleft%
$2$-graphs\textquotedblright. Graphs extend naturally to weighted $r$-graphs,
as explained in Section \ref{WGs}.

We write $\mathcal{G}^{r}$ for the family of all $r$-graphs and $\mathcal{G}%
^{r}\left(  n\right)  $ for the family of all $r$-graphs of order $n.$
Likewise, $\mathcal{W}^{r}$ stands for the family of all weighted $r$-graphs
and $\mathcal{W}^{r}\left(  n\right)  $ for the family of all weighted
$r$-graphs of order $n.$ Given a weighted graph $G,$ we write:\medskip

- $V\left(  G\right)  $ for the vertex set of $G;$

- $E\left(  G\right)  $ for the edge set of $G;$

- $\left\vert G\right\vert $ for $\sum\left\{  G\left(  e\right)  :e\in
E\left(  G\right)  \right\}  ;$

- $G\left[  U\right]  $ for the graph induced by a set $U\subset V\left(
G\right)  .$\bigskip

In the following definitions, if not specified otherwise, \textquotedblleft
graph\textquotedblright\ stands for weighted graph.\medskip

$k$-\textbf{chromatic }graph - the vertices can be partitioned into $k$ sets
so that each edge intersects at least two sets.\textbf{ }An edge
maximal\textbf{ }$k$-chromatic graph $G$ is called \textbf{complete}
$k$\textbf{-chromatic}; the complement of a complete $k$-chromatic $G$ is a
union of $k$ disjoint complete graphs;

\textbf{chromatic number }of a graph $G,$ in writing $\chi\left(  G\right)  ,$
is the smallest $k$ for which $G$ is $k$-chromatic. Using colors, $\chi\left(
G\right)  $ is the smallest number of colors needed to color the vertices of
$G$ so that no edge is monochromatic;

\textbf{complement }of a graph - the complement of a graph $G\in
\mathcal{G}^{r}$ is the graph $\overline{G}\in\mathcal{G}^{r}$ with $V\left(
\overline{G}\right)  =V\left(  G\right)  $ and $E\left(  \overline{G}\right)
=\left(  V\left(  G\right)  \right)  ^{\left(  r\right)  }\backslash E\left(
G\right)  $;

\textbf{complete }graph - a graph having all possible edges; $K_{n}^{r}$
stands for the complete $r$-graph of order $n$;

\textbf{connected} graph\textbf{ -}\ for any partition of the vertices into
two sets, there is an edge that intersects both sets;

\textbf{degree} - given a graph $G$ and $u\in V\left(  G\right)  ,$ the degree
of\emph{ }$u$\ is $d\left(  u\right)  =\sum\left\{  G\left(  e\right)  :e\in
E\left(  G\right)  \text{ and }u\in e\right\}  $; \ $\delta\left(  G\right)  $
and $\Delta\left(  G\right)  $ denote the minimum and maximum vertex degrees
of $G;$ more generally, if $U\subset V\left(  G\right)  ,$ the degree of
$U$\ is $d\left(  U\right)  =\sum\left\{  G\left(  e\right)  :e\in E\left(
G\right)  \text{ and }U\subset e\right\}  $;

$\beta$\textbf{-star }with vertex $v$ - a set of edges such that the
intersection of every two edges is $v;$

$\beta$\textbf{-degree }of a vertex - given a graph $G$ and $v\in V\left(
G\right)  ,$ the $\beta$-degree\textbf{ }$d^{\beta}\left(  v\right)  $ of
$v$\ is the maximum size of a $\beta$-star with vertex $v;$

$\Delta^{\beta}\left(  G\right)  ,$ $\delta^{\beta}\left(  G\right)  $\textbf{
}- the maximum and the minimum $\beta$-degrees of the vertices of $G;$

\textbf{graph property -} a family of graphs closed under isomorphisms; a
property closed under taking subgraphs is called \textbf{monotone;} a property
closed under taking induced subgraphs is called \textbf{hereditary};

\textbf{isolated }vertex - a vertex not contained in any edge;

$k$-\textbf{linear }graph - a graph $G$ is $k$-linear if every two edges of
$G$ share at most $k$ vertices; a $1$-linear graph is called linear;

\textbf{order}\ of a graph - the number of its vertices;

$k$-\textbf{partite} graph\textbf{ - }a\textbf{ }graph whose vertices can be
partitioned into $k$ sets so that no edge has two vertices from the same set.
An edge maximal $k$-partite graph is called \textbf{complete }$k$%
-\textbf{partite};

\textbf{rank }of a graph\textbf{ }- the cardinality of its edges; e.g.,
$r$-graphs have rank $r$;

\textbf{regular }graph\textbf{ - }each vertex has the same degree;

$k$-\textbf{set regular}\ graph - each set of $k$ vertices has the same degree;

$k$-\textbf{section }of a graph $G$ is the $k$-graph $G_{\left(  k\right)
}\in\mathcal{G}_{k}$ with $V\left(  G_{\left(  k\right)  }\right)  =V\left(
G\right)  $ and $E\left(  G_{\left(  k\right)  }\right)  $ is the set of all
$k$-subsets of edges of $G;$

\textbf{size }of a graph $G$ - the number $\left\vert G\right\vert
=\sum\left\{  G\left(  e\right)  :e\in E\left(  G\right)  \right\}  $;

\textbf{Steiner }$\left(  k,r,n\right)  $-\textbf{system} is a graph in
$\mathcal{G}^{r}\left(  n\right)  $ such that any set of $k$ vertices is
contained in exactly one edge;\textbf{ }a\textbf{ Steiner triple system }is
a\textbf{ }Steiner $\left(  2,3,n\right)  $-system;

\textbf{subgraph} - if $H\in\mathcal{W}^{r}$ and $G\in\mathcal{W}^{r}$, $H$ is
a subgraph of $G$, if $V\left(  H\right)  \subset V\left(  G\right)  ,$ and
$e\in E\left(  H\right)  $ implies that $H\left(  e\right)  =G\left(
e\right)  ;$ a subgraph $H$ of $G$ is called \textbf{induced }if $e\in
E\left(  G\right)  $ and $e\subset V\left(  H\right)  $ implies that $H\left(
e\right)  =G\left(  e\right)  ;$

$k$\textbf{-tight }graph - a graph $G\in\mathcal{W}^{r}$ is $k$-tight if
$E\left(  G\right)  \neq\varnothing$ and for any proper $U\subset V\left(
G\right)  $ containing edges, there is an edge $e$ such that $k\leq\left\vert
e\cap U\right\vert \leq r-1;$ a graph is $1$-tight if and only if it is connected;

\textbf{transversal }of a graph - a set of vertices intersecting each edge;
\textbf{odd (even) transversal - }a set of vertices intersecting each edge in
an odd (even) number of vertices; even transversals may have empty
intersections with edges;

\textbf{union }of graphs \textbf{-} if $G\in\mathcal{G}^{r}$ and
$H\in\mathcal{G}^{r}$, their union $G\cup H$ is an $F\in\mathcal{G}^{r}$
defined by $V\left(  F\right)  =V\left(  G\right)  \cup V\left(  H\right)  ,$
and $E\left(  F\right)  =E\left(  G\right)  \cup E\left(  H\right)  $. In
particular, $tG$ denotes the union of $t$ vertex disjoint copies of $G.$

\bigskip

\section{List of references}

\end{document}